\documentclass[twoside, final]{amsart}

\usepackage{amsmath, amsfonts, amsthm, amssymb}
\usepackage{exscale}
  \usepackage{cite}
\usepackage{epsfig}
\usepackage{amscd}
\usepackage{graphics}
\usepackage{color}
\usepackage{dsfont}

\usepackage[notref, notcite]{showkeys}

\renewcommand{\geq}{\geqslant}
\renewcommand{\leq}{\leqslant}

\def\reals{{\mathbb R}}

\def\Ci{{\mathcal C}^\infty}

\def\supp{\mathrm{supp}\,}

\def\O{{\mathcal O}}

\def\s{{\mathcal S}}

\def\Op{\mathrm{Op}\,}

\def\be{\begin{eqnarray*}}
\def\ee{\end{eqnarray*}}
\def\ben{\begin{eqnarray}}
\def\een{\end{eqnarray}}

\def\lll{\left\langle}
\def\rrr{\right\rangle}

\def\L2R{L_{\text{Rest}}^2}

\def\11{\mathds{1}}

\def\tpsi{\tilde{\psi}}
\def\tchi{\tilde{\chi}}
\def\L2c{L^2_{\text{comp}}}
\def\tphi{\tilde{\phi}}

\def\tu{u^{hi}}

\def\tB{\tilde{B}}

\newtheorem{theorem}{Theorem}[section]
\newtheorem{lemma}[theorem]{Lemma}
\newtheorem{proposition}[theorem]{Proposition}
\newtheorem{remark}[theorem]{Remark}

\newtheorem{corollary}[theorem]{Corollary}

\newtheorem*{main-theorem 1}{Main Theorem 1}
\newtheorem*{main-theorem 2}{Main Theorem 2}

\newtheorem*{theorem*}{Theorem}
\newtheorem*{remark*}{Remark}

\numberwithin{equation}{section}

\begin{document}

\title[Strichartz estimate for the water-wave problem]{
Strichartz estimates for \\
the water-wave problem with surface tension}

\author[Christianson]{Hans~Christianson}
\address{Massachusetts Institute of Technology, Department of Mathematics\\
77 Massachusetts Avenue, Cambridge, MA 02139-4307, USA}
\email{hans@\allowbreak math.\allowbreak mit.\allowbreak edu}

\author[Hur]{Vera~Mikyoung~Hur}
\email{verahur@\allowbreak math.\allowbreak mit.\allowbreak edu}

\author[Staffilani]{Gigliola~Staffilani}
\email{gigliola@\allowbreak math.\allowbreak mit.\allowbreak edu}

\date{\today}

\maketitle

\begin{abstract}
Strichartz-type estimates for one-dimensional surface water-waves under surface tension are studied, 
based on the formulation of the problem as a nonlinear dispersive equation. 
We establish a family of dispersion  estimates  on time scales depending on the size of the frequencies.
We infer that a solution $u$ of the dispersive equation we introduce satisfies 
local-in-time Strichartz estimates with loss in derivative:
 \[
\| u \|_{L^p([0,T]) W^{s-1/p,q}(\mathbb{R})} \leq C, 
\qquad  \frac{2}{p} + \frac{1}{q} = \frac{1}{2},
\]
where $C$ depends on $T$ and on the norms of the initial data in $H^s
\times H^{s-3/2}$.
The proof uses the frequency analysis and semiclassical Strichartz estimates 
for the linealized water-wave operator.
\end{abstract}

\setcounter{tocdepth}{1}

\tableofcontents

\section{Introduction}\label{S:intro}

The problem of {\em surface water waves}, in its simplest form, 
concerns the two-dimensional dynamics of an incompressible inviscid liquid of infinite depth and 
the wave motion on its one-dimensional surface layer,
under the influence of gravity and surface tension. 
The {\em moving} surface is given as a nonself-intersecting parametrized curve.
The liquid occupies the domain below the curve, 
where the liquid motion is described by the Euler equations under gravity.
The flow beneath the moving surface is required to be irrotational.
The {\em kinematic} and {\em dynamic} boundary conditions hold at the moving surface,
stating respectively that the normal component of velocity is continuous along the moving surface
and that the jump in pressure across the moving surface is proportional to its mean curvature.
The flow is assumed to be almost at rest at great depths, 
and the moving surface is taken to be asymptotically flat.

Provided with the initial surface profile and the initial state of fluid current,
the water-wave problem naturally poses as an initial value problem.
Early mathematical results for local well-posedness date back to 
\cite{Ovs, KaNi} and include \cite{Nal, Yos1, Yos2, Cra}.
Following the works by Sijue Wu \cite{Wu1, Wu2}
there has been considerable progress in the study of local well-posedness
for the water-wave problem as well as for a class of the Euler equations with free boundary.
We refer the interested reader to 
\cite{CL, AM1, AM2, Lan, Lin, CS1, SZ3}, and references therein. 
Recently, results for long-time existence \cite{Wu3, GNS} appeared 
for gravity water waves of infinite depth. 

Nonlinearity characteristic to the boundary conditions at the moving surface
significantly restricts the range of analytical tools available 
for the existence theory for the water-wave problem.
As a matter of fact, all results listed in the previous paragraph on local well-posedness
hinge upon obtaining high energy expressions and establishing their bounds,
namely the {\em energy method}.
While construction of such energy expressions is nontrivial 
and design of an iteration scheme is involved, nevertheless, 
results from the energy method do not provide any further information about solutions,
other  than that they remain as smooth as their initial states.
Better understanding of the dynamics of surface water waves
can be made with the help of a priori estimates other than energy estimates. 

On the other hand, the {\em dispersion relation} (see Remark \ref{SS:dispersion})
\begin{equation}\label{E:dispersion_relation}
c(k)=\left(\frac{S}{2}|k|+\frac{g}{|k|}\right)^{1/2} \frac{k}{|k|}
\end{equation}
of surface water waves provides a guiding principle of their {\em linear} dynamics. Here, 
$c(k)$ is the speed of the simple harmonic oscillation with the wave number $k$;
$S\geq 0$ is the coefficient of surface tension 
and $g\geq 0$ is the gravitational constant of acceleration. 
Under the influence of surface tension, i.e. $S>0$, 
the fact that the phase velocity $c(k)$ is asymptotically proportional to 
the square root of $k$ as $k \to \infty$ indicates 
a certain ``regularizing'' effect by the process of broadening out the surface profile.
In the gravity-wave setting, i.e. $S=0$ and $g>0$, in contrast, 
\eqref{E:dispersion_relation} does not induce such a regularizing effect\footnote{
Gravity waves may still be thought of as ``dispersive''
in the sense that wave components with different frequencies propagate
at different speeds; see \cite{Wu3}.}.

Dispersive properties have  paramount importance in the study of
nonlinear Schr\"odinger equations, the Korteweg-de Vries equation, 
nonlinear wave equations, and other nonlinear dispersive equations. In the recent work of
Alazard, Burq and Zuilly \cite{ABZ-water}, local smoothing effects are
obtained for water waves under surface tension (see also Appendix \ref{S:local_smoothing}).
Such a smoothing effect is a direct consequence 
of the dispersive property of surface water waves,
and it contrasts markedly with what energy estimates alone can tell. 
The present purpose is to investigate the dispersive property for
the water-wave problem with one-dimensional surface under surface tension
by establishing estimates of the solution under the mixed Sobolev norms, 
commonly referred to as {\em Strichartz estimates}.

\subsection{The main results}\label{SS:result}
The present treatment of the dispersive property for the water-wave problem
under surface tension is based on the formulation 
of the problem as a second-order in time nonlinear dispersive equation
\begin{equation}\label{E:main}
\partial_t^2 u-\frac{S}{2}H\partial_\alpha^3 u+gH\partial_\alpha u=
-2u \partial_t \partial_\alpha u -u^2 \partial_\alpha^2 u+R(u,\partial_t u),
\end{equation}
coupled with a transport-type equation \eqref{E:transp}. 
We shall derive it in Section \ref{S:formulation} and Section \ref{S:reformulation}.
Here, $u$ is related to the tangential velocity at the moving surface 
and it serves as the unknown; $t \in \mathbb{R}_+$ is the temporal variable and 
$\alpha \in \mathbb{R}$ is the (renormalized) arclength parametrization of the curve,
which serves as the spatial variable. 
The Hilbert transform, denoted by $H$, may be defined via the Fourier transform as 
$\widehat{Hf}(\xi)=-i\text{sgn}(\xi)\widehat{f}(\xi)$.
The remainder $R$ is of lower order 
compared to $2u \partial_t \partial_\alpha u$ and $u^2 \partial_\alpha^2 u$ in the sense that 
\[
\| R(u,\partial_t u)\|_{H^s} \leq C(\|u\|_{H^{s+1}}, \|\partial_t u\|_{H^s})
\]
for $s\geq 1$. Here and elsewhere, $H^s$ means the $L^2$-Sobolev space  of order $s$
in the variable $\alpha \in \mathbb{R}$.  


Our main results concern Strichartz estimates for the water-wave problem under surface tension
with loss in derivative.
In the course of the proof, its local well-posedness is proved. 

\begin{theorem}\label{Main1}
Let $S>0$ and $g\geq 0$ be held fixed. For $s> 2+1/2$ 
the initial value problem of \eqref{E:main}
prescribed with the initial conditions 
\[
u(0,\alpha)=u_0(\alpha) \quad \text{and}\quad \partial_t u(0,\alpha)=u_1(\alpha),
\]
where $(u_0, u_1) \in H^{s}(\mathbb{R}) \times H^{s-3/2}(\mathbb{R})$
is locally well-posed on a time interval $t \in [0, T]$ for some $T>0$, and the solution $u$  satisfies 
$(u(t), \partial_t u(t)) \in C([0,T]; H^s(\mathbb{R}) \times H^{s-3/2}(\mathbb{R}))$.

Moreover, if $s$ is sufficiently large, the solution $u$ satisfies the inequality
\begin{equation}\label{E:main-est}
\Big(\int_0^T\Big( \int_{-\infty}^\infty |
D^{s-1/p}_\alpha u(t,\alpha) |^qd\alpha\Big)^{p/q}dt\Big)^{1/q} \leq C,
\end{equation}
where $(p,q)$ satisfies the admissibility condition
\begin{equation}\label{E:p-q}
\frac{2}{p} + \frac{1}{q} = \frac{1}{2}, \qquad q < \infty,
\end{equation}
and $C>0$ depends on $s, q, p, T$ and $\| u_0 \|_{H^s(\mathbb{R})}, \|u_1 \|_{H^{s-3/2}(\mathbb{R})}$.
Here and in sequel, $D_\alpha=-i\partial_\alpha$.
\end{theorem}

If the solution is localized to dyadic frequency bands and semiclassical time scales, 
the estimate is better.

\begin{theorem}\label{Main2}
Let $\psi^j(D_\alpha)$ be a Fourier multiplier
supported in frequencies $2^{j-2} \leq | \xi | \leq 2^{j+2}$.
Under the hypothesis of Theorem \ref{Main1} with $s$ sufficiently large, 
the frequency-localized solution $\psi^j(D_\alpha) u$ satisfies
\begin{equation}
\label{E:main-est-sc}
\Big(\int_0^{2^{-j/2}T}\Big( \int_{-\infty}^\infty |
D^{s-1/2p}_\alpha \psi^j(D_\alpha) u(t,\alpha) |^qd\alpha\Big)^{p/q}dt\Big)^{1/q} \leq C,
\end{equation}
where $(p,q)$ satisfies \eqref{E:p-q} with $q \leq \infty$ 
and $C>0$ depends on $s, q, p, T$ and $\| u_0 \|_{H^s_\alpha}, \|u_1 \|_{H^{s-3/2}_\alpha}$.
\end{theorem}

\subsection*{Notations}
\label{notation-sect}
Recorded here are the notations and conventions used in the sequel.

Let $0 \leq k , l\leq \infty$ and $1 \leq p , q \leq \infty$.
By $W^{k,q}_\alpha(\mathbb{R})$ we mean the $L^q$ Sobolev space 
on $\alpha \in \mathbb{R}$ of order $k$,
and by $W^{l,p}_t([0,T])$ we mean the $L^p$ Sobolev space 
on the interval $t \in [0,T]$ of order $l$. 
By $H^l_t([0,T])$ the $L^2$ Sobolev space on the interval $t \in [0,T]$ of order $l$ .  
We will also use the Sobolev spaces of negative order, $H^k_\alpha(\mathbb{R})$ with $k<0$.
For $0\leq p,q \leq \infty$ we recall the definitions for the mixed Sobolev spaces  
$L^q_\alpha(\reals) L^p_t ([0,T])$ and $L^p_t([0,T])L^q_\alpha(\reals)$ by the norms of these spaces
\begin{eqnarray*}
\|f\|_{L^q_\alpha(\reals) L^p_t ([0,T])}&=&\left(\int_{\reals}\left(\int_0^T
|f(t,\alpha)|^p\,dt\right)^{q/p}\,d\alpha\right)^{1/q},\\
\|f\|_{L^p_t([0,T])L^q_\alpha(\reals)}&=&\left(\int_0^T\left(\int_{\reals}
|f(t,\alpha)|^q\,d\alpha\right)^{p/q}\,dt\right)^{1/p}.
\end{eqnarray*}

We write $L^q_\alpha L^p_T$ for $L^q_\alpha(\reals) L^p_t ([0,T])$
and $L^p_TL^q_\alpha$ for $L^p_t([0,T])L^q_\alpha(\reals)$
when there is no ambiguity. 
We use the analogous  convention for 
$W^{k,q}_\alpha W^{l,p}_T$, $W^{l,p}_T W^{k,q}_\alpha$, and $H^l_T H^k_\alpha $.

\subsection{Perspectives}\label{SS:perspectives}
The derivative loss of $1/p$
in Theorem \ref{Main1} is likely not sharp, as the following heuristic arguments indicate.  

For any dispersive equation in one spatial dimension, the $W^{s,1}_\alpha \to L^\infty_\alpha$ decay rate is $t^{-1/2}$  (with loss of $s$ derivatives depending on the equation).  If we linearize about the zero solution (see \eqref{E:ww-oscillatory} below), we see the solution satisfies Strichartz estimates with the admissibility condition \eqref{E:p-q} and a $1/2p$ derivative loss (see, for example, \cite{CHS-oi}), which is an improvement of $1/2p$ derivatives compared to Theorem \ref{Main1}.
Moreover, this equation satisfies the scaling symmetry\footnote{In the absence of the effect of gravity, $g=0$, 
the nonlinear equation \eqref{E:main} also enjoys this scaling symmetry.   This  follows from the scaling symmetry of the Euler equations and the dynamic boundary condition that 
the jump of pressure across the moving surface is proportional to the mean curvature of the surface.}
\[
u(t, \alpha) \mapsto 
\lambda^{1/2} u(\lambda^{3/2} t, \lambda \alpha)
\]
for any dilation factor $\lambda>0$, and we readily verify that Strichartz estimates with the admissibility condition \eqref{E:p-q} and $1/2p$ derivative loss is invariant with respect to this scaling.  We thus expect Strichartz estimates with admissibility condition \eqref{E:p-q} and $1/2p$ derivative loss to be optimal.  That is, Theorem \ref{Main1} represents twice the loss in derivative of the optimal estimate.

However, this optimal estimate cannot be obtained by interpolation with known estimates, even in weighted form.  Indeed, to compare to the local smoothing estimate  (\cite{ABZ-water} or Appendix \ref{S:local_smoothing}),  if we use Sobolev embeddings, we have
\[
\|  \lll \alpha \rrr^{-\rho} D_\alpha^{s-1/2p} u \|_{L^p(0,T]) L^q_\alpha} \leq 
C \| \lll \alpha \rrr^{-\rho} D_t^{1/2 - 1/p} D_\alpha^{s-1/2p + 1/2 - 1/q} u \|_{L^2([0,T]) 
L^2_\alpha}
\]
and if we use that $D_t$ is comparable to $D_\alpha^{3/2}$ 
(at least for a solution linearized about $0$), in turn, we have
\[
\|  \lll \alpha \rrr^{-\rho} D_t^{1/2-1/p} D_\alpha^{s+1/2 - 1/q} u \|_{L^2([0,T]) L^2_\alpha} 
\leq C \|  \lll \alpha \rrr^{-\rho} D_\alpha^{s+5/4 - 1/q - 2/p} u \|_{L^2([0,T]) L^2_\alpha}.
\]
By the local smoothing effect gain of $1/4$ derivative,   
we then bound $\| \lll \alpha \rrr^{-\rho}  D_\alpha^{s-1/2p} u \|_{L^p(0,T]) L^q_\alpha}$, $\rho >1/2$, 
in terms of the  initial data in $H^s(\mathbb{R}) \times H^{s-3/2}(\mathbb{R})$, provided that
\[
\frac{2}{p} + \frac{1}{q} = 1.
\]
This is weaker than the optimal estimate.  On the other hand, if we use H\"older's inequality plus energy conservation, we get a loss of $1/2p$ derivatives provided
\[
\frac{1}{2p} + \frac{1}{q} = \frac{1}{2}
\]
(see Figure \ref{F:fig1a}).
\begin{figure}
\hfill
\centerline{\input{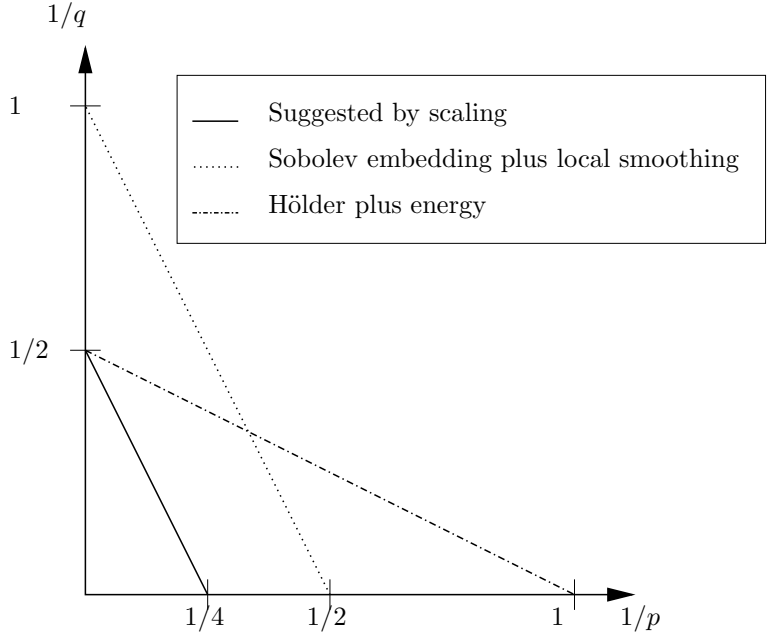}}
\caption{\label{F:fig1a} Fixed time scale, $1/2p$ derivative loss.  The $(p,q)$ relation suggested by scaling, from Sobolev embeddings plus
local smoothing effect, and from H\"older's inequality in time with Sobolev embeddings and energy conservation.}
\hfill
\end{figure}

To make a direct comparison of the estimate of Theorem \ref{Main1} with the optimal condition is not as clear, since we must use Sobolev embeddings somewhere.  If we do use an additional Sobolev embedding in the discussion above to make a comparison of $1/p$ derivative loss, the optimal admissibility condition becomes
\begin{equation}
\label{E:p-q-optimal-loss}
\frac{5}{2p} + \frac{1}{q} = \frac{1}{2},
\end{equation}
while that from smoothing is \eqref{E:p-q-optimal-loss} with the right hand side replaced by $1$, and that for energy estimates is 
\[
\frac{1}{p} + \frac{1}{q} = \frac{1}{2}
\]
(see FIgure \ref{F:fig1}).  Again we see that the estimate of Theorem \ref{Main1} cannot be obtained by interpolation between known estimates.

\begin{figure}
\hfill
\centerline{\input{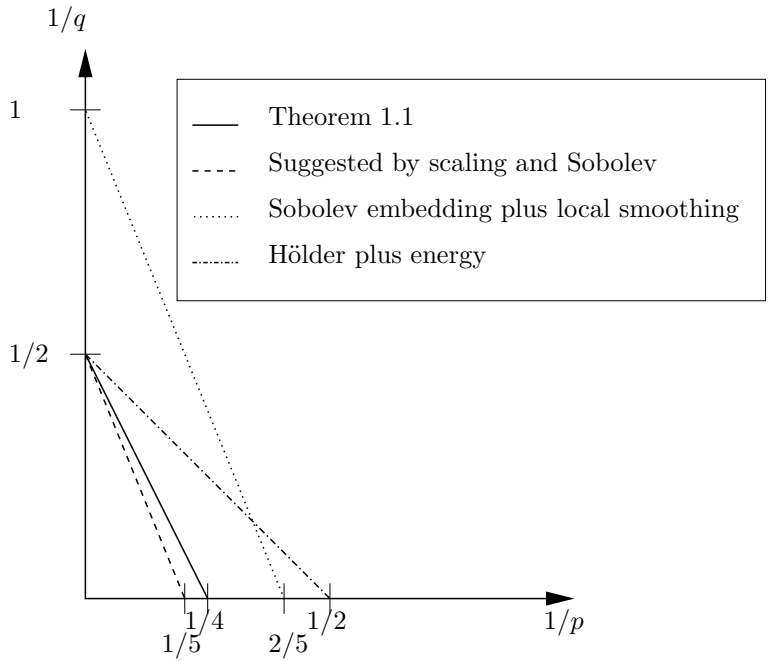}}
\caption{\label{F:fig1} Fixed time scale, $1/p$ derivative loss.  The $(p,q)$ relation given in Theorem \ref{Main1}, 
that suggested by scaling, from Sobolev embeddings plus
local smoothing effect, and from H\"older's inequality in time with Sobolev embeddings and energy conservation.}
\hfill
\end{figure}

On the semiclassical time scale $0 \leq t \leq 2^{-j/2}T$, 
our Strichartz estimate \eqref{E:main-est-sc} has a smaller loss in derivative, 
and the optimal scaling condition is the same as \eqref{E:p-q}.  Since the local smoothing cannot be improved 
on the semiclassical time scale, our estimate \eqref{E:main-est-sc} represents 
a larger gain over what Sobolev embeddings plus local smoothing could
tell us on the semiclassical time scale (see Figure \ref{F:fig3}).
\begin{figure}
\hfill
\centerline{\input{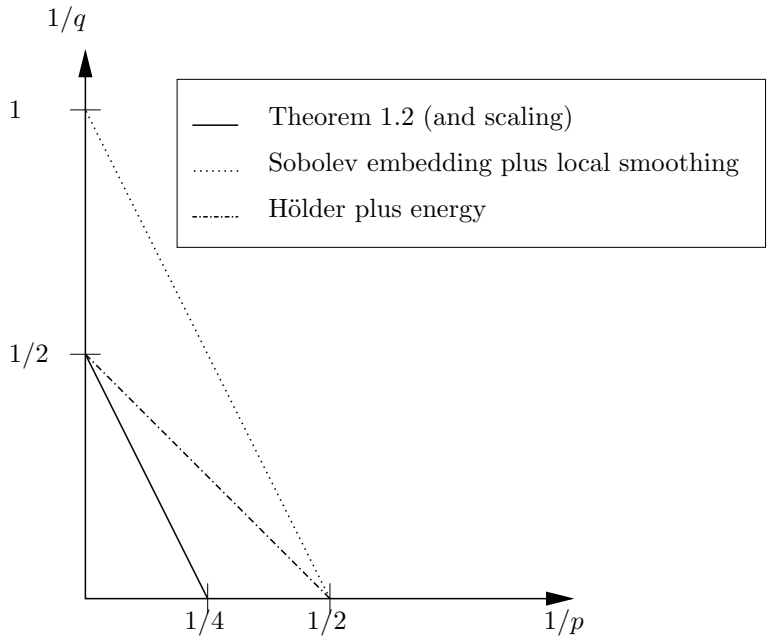}}
\caption{\label{F:fig3} Semiclassical  time scale, $1/2p$ derivative
  loss.  The $(p,q)$ relation given in Theorem
  \ref{Main2} (agrees with that suggested by scaling), 
Sobolev embeddings plus
local smoothing, and H\"older's inequality in time plus Sobolev
embeddings plus energy conservation.}
\hfill
\end{figure}


\subsection{Idea of the proofs}\label{SS:idea-str}
While \eqref{E:main} is dispersive, its nonlinearity is severe,
and as such in the study of its dispersive properties one must take its nonlinear effect into account. 
To better understand the strength of nonlinearity versus the weakness of dispersion 
we examine the local smoothing effect for \eqref{E:main}.
An application of Parseval's formula, together with a change of variables, shows that (\cite{KPV1} for instance) 
the solution of the linear homogeneous equation 
\begin{equation}\label{E:ww0}
\partial_t^2u-\frac{S}{2}H\partial_\alpha^3u=0, \qquad S>0
\end{equation}
gains $1/4$ derivative of smoothness over the initial data. 
An application of a $TT^*$ argument then shows that the solution 
of the corresponding inhomogeneous equation gains $2$ derivatives of smoothness 
over the inhomogeneity. But, this local smoothing effect is not enough 
to control nonlinear terms in \eqref{E:main} containing more than two spatial derivatives, 
e.g. $2u \partial_t \partial_\alpha u$.

To overcome this setback and to obtain Strichartz estimates for the solution of 
nonlinear equation \eqref{E:main}, we write it as  
\begin{equation}\label{E:main1}
\partial^2_tu -\frac{S}{2}H\partial_\alpha^3 u+gH\partial_\alpha u
+2u\partial_t\partial_\alpha u+u^2\partial_\alpha^2u=R(u, \partial_tu).\end{equation}
That is, we view $2u\partial_t\partial_\alpha u$ and $u^2\partial^2_\alpha u$
as ``linear'' components of the equation, but with variable coefficients 
which happen to depend on the solution itself.
In other words, we reduce the size of nonlinearity 
at the expense of making its linear part more complicated. 
We then make a serious effort to establish Strichartz estimates for the linear operator 
\begin{equation}\label{E:wwlinear}
\partial^2_t -\frac{S}{2}H\partial_\alpha^3 +gH\partial_\alpha 
+2V(t,\alpha)\partial_t\partial_\alpha +V^2(t,\alpha)\partial_\alpha^2,
\end{equation}
for a class of functions for the variable coefficient $V(t,\alpha)$.

The operator \eqref{E:wwlinear} may be thought of 
the operator $\partial_t^2-H\partial_\alpha^3$ perturbed by 
variable-coefficient but lower-order terms
$2V(t,\alpha)\partial_\alpha\partial_t+V^2(t,\alpha)\partial_\alpha^2$.
While the added terms are of lower order they are not constant,
and they bring a great deal of difficulty in the analysis of the paper,
which is the heart of the matter. 

In \cite{ABZ-water} and in Appendix \ref{S:local_smoothing}, 
in order to establish the local smoothing effect for the nonlinear equation \eqref{E:main},
similar approaches are employed. 


\subsubsection{Construction of the parametrix}\label{ss:parametrix}
Our approach to establishing microlocal Strichartz estimates for \eqref{E:wwlinear}
is based on the construction of its approximate solution. 

When $V(t,\alpha)=0$, the solution of the homogeneous equation \eqref{E:ww0} 
is given by the formula 
\begin{equation}\label{E:ww-oscillatory0}
\begin{split} 
u(t,\alpha)=\frac{1}{4\pi} \iint e^{i(\alpha-\beta)\xi} 
\Big((e^{it|\xi|^{3/2}}&+e^{-it|\xi|^{3/2}} ) u_0(\beta) \\
&+\frac{e^{it|\xi|^{3/2}}-e^{-it|\xi|^{3/2}} }{i|\xi|^{3/2}}u_1(\beta)\Big) d\beta d\xi,
\end{split}
\end{equation}
where $u_0$ and $u_1$ describe the initial data. 
Here, for the sake of exposition, we have assumed $S/2=1$ and $g=0$. 
Motivated by this, we make an oscillatory integral ansatz
\begin{equation*}
w(t,\alpha)  = \frac{1}{2 \pi} \iint e^{-i\beta \xi}( e^{i \varphi^+
(t,\alpha ,\xi) } f^+(\beta) + e^{i \varphi^-(t,\alpha ,\xi) } f^-(\beta) ) \, d \beta d \xi
\end{equation*}
to solve the problem associated to \eqref{E:wwlinear}.
The {\it phase functions} $\varphi^\pm$ is chosen to satisfy 
$\varphi^\pm(0,\alpha,\xi) = \alpha \xi$, and as such 
the recovery of the initial conditions entails solving for $f^\pm$ 
a system of elliptic pseudodifferential equations.

Applying the linear operator \eqref{E:wwlinear} to our ansatz, we consider 
the worst terms, produced when first-order derivatives fall on the phase functions. 
They make a first-order nonlinear equation \eqref{phi-hj-eqn} for $\varphi^\pm$, 
commonly referred to as the {\it eikonal} or {\it Hamilton-Jacobi} equation.
The usual approach to solving the Hamilton-Jacobi equation is 
through the technique of generating functions for the associated Hamiltonian. 
The equation \eqref{phi-hj-eqn} is, however, neither homogeneous nor polyhomogeneous (in
$\varphi_t^\pm$ and $\varphi^\pm_\alpha$),
and as such solutions are found on a time scale comparable to $|\xi|^{-1/2}$.
See Lemma \ref{HJ-lemma} for details.
We thus construct phase functions for each dyadic frequency band $|\xi| \sim 2^j$ on a
frequency-dependent time scale $t \sim 2^{-j/2}$.
The construction of the {\em leading-order parametrix} $w$ is detailed in Section \ref{S:parametrix}.

\subsubsection{Semiclassical Strichartz estimates}
We explain our strategy to establish Strichartz estimates 
for the linearized water-wave operator \eqref{E:wwlinear} under surface tension. 

Let us first discuss basic ideas for  Strichartz estimates 
for the one-dimensional free Schr\"odinger equation 
\begin{equation}\label{E:Schrodinger} 
i \partial_t u+\partial_\alpha^2 u=0, \qquad t, \alpha \in \reals 
\end{equation}
since we will use similar ideas. 
Prescribed with the initial condition $u(0,\alpha)=u_0(\alpha)$,
the solution of \eqref{E:Schrodinger} can be written via the Fourier transform as 
\[
u(t, \alpha) = \iint e^{i\xi(\alpha - \beta)} e^{it \xi^2} u_0(\beta) d \beta d \xi.
\]
We write this as a convolution with an integral kernel as
\[
u(t,\alpha) = \int K(t, \alpha, \beta) u_0 (\beta) d \beta, \qquad \text{where}\quad
K(t,\alpha, \beta) = \int  e^{i\xi(\alpha - \beta)} e^{it \xi^2} d \xi.
\]

The phase function $\varphi(\xi; t,\alpha, \beta) = \xi(\alpha - \beta) + t\xi^2$ has a critical point at 
\[
\partial_\xi \varphi(\xi_c) = \alpha - \beta + 2t \xi_c = 0, \quad \text{or} \quad \xi_c = (\beta - \alpha)/2t.
\]
Since $\partial_\xi^2 \varphi (\xi_c)= 2t$, moreover, the phase is nondegenerate for $t>0$. 
Then, by the standard method of stationary phase\footnote{
Estimate \eqref{E:dispersion-Schrodinger} is usually derived 
from the explicit formula for the kernel $K$, but here we want to stress a method 
that can be generalized for variable coefficient dispersive
differential operators.}, at least for $u_0$ localized in frequency, 
 we obtain $K = K_1 + \text{(smoothing)}$ with
\begin{equation}\label{E:dispersion-Schrodinger}
| K_1( t, \alpha, \beta) | \leq C t^{-1/2},
\end{equation}
where $C>0$ is independent of $t$, $\alpha$ and $\beta$. 

Next, we recall an abstract result which follows from the work of Ginibre-Velo \cite{GiVe-sm,GiVe-str} and recorded in the paper of Keel and Tao \cite{KT-str},
stating that a dispersion estimate leads to Strichartz estimates under the $L^p_tL^q_\alpha$-norm
for a range of $(p,q)$ depending on the strength of the dispersion 
(the power of $t$ in the dispersion estimate).

\begin{theorem}
\label{T:KT}
Let $(X, dx)$ be a measure space, let $H$ be a Hilbert space, and let
$U(t): H \to L^2(X)$ be a linear operator satisfying
\begin{align*}
& \text{\em  (i) }\; \| U(t) f \|_{L^2_x} \leq C_1 \| f \|_{H}, \text{ and}\\
& \text{\em  (ii) } \|U(t') U^*(t) g \|_{L^\infty_x} \leq C_1 |t-t'|^{- \sigma} \| g \|_{L^1_x}
\end{align*}
for some $\sigma >0$.  Then for every pair $(p,q)$ satisfying
\[
\frac{1}{p} + \frac{\sigma}{q} = \frac{ \sigma }{2},
\]
the estimate
\[
\| U(t) f \|_{L^p_t L^q_x} \leq C_2 \| f \|_H
\]
holds true, where $C_2>0$ depends only on $C_1$, $\sigma$, $p$ and $q$.
\end{theorem}

The semiclassical dispersion estimate we prove in this paper depends
also on the semiclassical parameter $2^{-j}$.  A rescaling in time and
application of Theorem \ref{T:KT} gives the following semiclassical
Strichartz estimate theorem (see, for example, \cite[Theorem B.10]{EvZw}).

\begin{theorem}[Semiclassical Strichartz estimates]
\label{T:general-sc-str}
Let $(X, dx)$ be a measure space, let $h_0>0$ fixed, let $H$ be a Hilbert space, and let
$U(t): H \to L^2(X)$ be a linear operator satisfying
\begin{align*}
& \text{\em  (i) }\; \| U(t) f \|_{L^2_x} \leq C_1 \| f \|_{H}, \text{ and}\\
& \text{\em  (ii) } \|U(t') U^*(t) g \|_{L^\infty_x} \leq C_2 h^{-\mu}|t-t'|^{- \sigma} \| g \|_{L^1_x}
\end{align*}
for some $\sigma >0$ and all $0 < h \leq h_0$.  Then for every pair $(p,q)$ satisfying
\[
\frac{1}{p} + \frac{\sigma}{q} = \frac{ \sigma }{2},
\]
the estimate
\[
\| U(t) f \|_{L^p_t L^q_x} \leq C_3 h^{-\frac{\mu}{p \sigma}} \| f \|_H
\]
holds true, where $C_3>0$ depends only on $C_1$, $C_2$, $\sigma$, $\mu$, $p$ and $q$.
\end{theorem}

In light of the above theorem, \eqref{E:dispersion-Schrodinger} gives that 
a solution of \eqref{E:Schrodinger} satisfies the estimate
\[ \|u\|_{L^p_tL^q_\alpha} \leq C\|u_0\|_{L^2}\]
where $(p,q)$ satisfies
\[
\frac{2}{p} + \frac{1}{q} = \frac{1}{2}.
\]
Furthermore, a scaling argument assures that the estimate is sharp.

\

Returning to our setting, we consider the linear problem
\begin{equation}\label{E:linear1}
\begin{cases}
\partial_t^2 U- H\partial_\alpha^3 U +  2V(t,\alpha) \partial_\alpha \partial_t U
+ V^2(t,\alpha)\partial_\alpha^2 U = R(t,\alpha)  \\
U(0,\alpha)= U_0(\alpha)\quad \text{and} \quad \partial_t U(0,\alpha) = U_1(\alpha),
\end{cases}
\end{equation}
on the time scale $[0, 2^{-j/2} T]$,\footnote{Here and elsewhere in the paper, the interval $[0, 2^{-j/2}T]$ can be substituted for any interval $I$ of length $2^{-j/2}T$ contained in the domain of $V$, by shifting $t$ to the beginning of the interval $I$.}
where $U, R, U_0$ and $U_1$ are localized to the dyadic frequency band 
$2^{j-2} \leq |\xi| \leq 2^{j+2}$. Here, for simplicity we take $S/2=1$ and $g=0$.
We write the oscillatory integrals 
\begin{equation}\label{E:ww-oscillatory}
\iint e^{\pm i t | \xi |^{3/2}} e^{i \xi (\alpha - \beta)} (U_0(\beta)
\mp i | \xi |^{-3/2} U_1(\beta)) d \beta d \xi 
\end{equation}
associated to the solution of the zero-coefficient equation, $V(t,\alpha)=0$.

Considering the corresponding phase for $\xi$ large and positive\footnote{
To avoid the singularity in the phase at $\xi = 0$, we assume our
initial data are localized to high frequencies.}, $t>0$, and with the $+$ sign,
let $\varphi(\xi; t, \alpha, \beta)=\xi(\alpha-\beta)+t\xi^{3/2}$. 
Its critical point $\xi_c$ is at
\begin{equation}\label{E:phixi-lin}
\partial_\xi \varphi(\xi_c) = \frac{3}{2} t \xi_c^{1/2} + \alpha - \beta = 0, \quad 
\text{or}\quad \xi_c=\frac49\left( \frac{\beta-\alpha}{t}\right)^2.
\end{equation}
Since
\begin{equation*}\label{E:phixixi-lin}
\partial_\xi^2 \varphi(\xi_c) = \frac{9}{8} \left( \frac{t^2}{ \beta - \alpha } \right),
\end{equation*}
the critical point is nondegenerate for $t>0$
and we are ready to use the method of stationary phase.
However, plugging these results into the stationary phase argument
does not yield a bound on the kernel uniform in $\alpha$ or $\beta$.

To remedy this,  we use propagation of singularities 
to estimate $(\alpha - \beta)/t$ in terms of derivatives in $\beta$ 
to obtain a dispersion rate of $t^{-1/2}$ with loss in derivative.
But, our parametrix is for the operator \eqref{E:wwlinear} with variable coefficients, 
which is considerably more complicated than \eqref{E:ww-oscillatory}. 
Moreover, the parametrix exists only for times $t \sim \xi^{-1/2}$, so we can only obtain this estimate on semiclassical time scales.
This approach is taken in \cite{BaCh-str, BGT-comp, Chr-disp-1, StTa, Tat-str1,Tat-str2,Tat-str3}
and many others, for the wave equations and the Schr\"odinger equations. 

In the proof of the microlocal dispersion estimate 
Lemma \ref{disp-lemma}, for the range of times $t \sim \xi^{-1/2}$, 
where $\xi$ is localized in a dyadic band $\xi \sim 2^j$, 
the relation \eqref{E:phixi-lin} implies that $(\alpha - \beta)/t$ is
bounded by $2^{j/2}$, and as a consequence,
the kernel corresponding to \eqref{E:ww-oscillatory} decays like $2^{j/4}t^{-1/2}$
on such a time scale. This decay rate explains the admissibility condition \eqref{E:p-q} in
the main result.  The loss in derivative comes from taking $\mu =
1/4$, $\sigma = 1/2$, and $h = 2^{-j}$ in Theorem \ref{T:general-sc-str}.

In Theorem \ref{T:general-str}, we use Theorem \ref{T:general-sc-str} to deduce 
the semiclassical Strichartz estimates for \eqref{E:linear1} as 
\begin{equation}\label{E:short-str1}
\| U \|_{L^p([0, 2^{-j/2}T]) L^{q}_\alpha} \leq
C (\| U_0 \|_{H^{1/2p}_\alpha} + \| U_1 \|_{H^{ 1/2p- 3/2}_\alpha}
 + \| R \|_{L^1([0, 2^{-j/2}T]) {H^{1/2p-3/2}_\alpha}}),
\end{equation}
where $(p,q)$ satisfies \eqref{E:p-q} and 
$C>0$ depend on $p,q$ and the Sobolev norms of $V$.



\subsubsection{Adding up the dyadic blocks and from linear to nonlinear}
We give  a brief outline of the argument that  allows us to move from  Theorem \ref{T:general-str} 
to Theorem \ref{Main1} and Theorem \ref{Main2}. 

Let us divide the interval $[0,T]$ into $2^{j/2}$ small intervals of the size $2^{-j/2}T$.
We apply \eqref{E:short-str1} on each short interval of size $2^{-j/2}T$
and we simply sum up $2^{j/2}$ many small time-scale estimates. 
In doing so we introduce an additional loss of $1/2p$ derivative. 
Then by appealing to Littlewood-Paley theory we sum up dyadic frequencies
to assert Corollary \ref{C:homog-str}.
We only pause here to remark that the parametrix is constructed only for high frequencies;
low frequencies can be estimated via energy estimates. 

In order to prove the estimate for the nonlinear equation \eqref{E:main1},
we employ the energy method to establish local existence and uniqueness of 
the solution of \eqref{E:main1} in the Sobolev classes. It is detailed in Section \ref{S:LWP}. 
Applying $\partial_\alpha^s$ to \eqref{E:main1}, we arrive at the linear equation 
\[
\partial^2_t \partial_\alpha^s u -\frac{S}{2}H\partial_\alpha^3\partial_\alpha^s u
+gH\partial_\alpha \partial_\alpha^s u
+2u\partial_t\partial_\alpha \partial_\alpha^s u+u^2\partial_\alpha^2 \partial_\alpha^s u
=\tilde R(u, \partial_tu)
\]
for $\partial_\alpha^su$, where $\tilde{R}$ is a collection of lower-order terms. 
By setting $u=V(t,\alpha)$  and 
$\tilde R(u,\partial_t u)=R(t,\alpha)$, and applying the above result, 
we assert that $\partial_\alpha^s u$ satisfies the estimates of Corollary \ref{C:homog-str}. 
As a consequence of uniqueness then 
$u$ satisfies Strichartz estimates as in the Theorem \ref{Main1}.

Theorem \ref{Main2} is obtained by repeating the argument above about how to move from 
linear to nonlinear problem via energy method to the Strichartz estimate \eqref{E:short-str1} 
for the dyadic-frequency localization. 

We finally remark that Theorem \ref{Main2} does not imply Theorem \ref{Main1} 
since frequency localization of the initial data is lost due to the presence of the nonlinearity.


\subsection{Organization}
The article consists of three main parts. 

The first part is to formulate the hydrodynamic problem of water waves under surface tension
as a nonlinear dispersive equation.
In Section \ref{S:formulation} we recall the formulation in \cite{AM1} of the water-wave problem. 
In Section \ref{S:reformulation} the system is further formulated as a second-order in time nonlinear
dispersive equation weakly coupled to a transport-type equation.

The second part concerns the semiclassical Strichartz estimates 
for the linearized water-wave equation under surface tension.
In Section \ref{S:parametrix} we construct a high-frequency parametrix
for each dyadic frequency band and on the frequency-dependent time scale.
In Section \ref{SS:energy} we prove that the parametrix possesses semiclassical Strichartz estimates.

The third part concerns results for the nonlinear problem. In Section \ref{S:LWP},
the local-in-time existence and uniqueness is established via the energy method. 
Finally, Section \ref{S:strichartz_nonlinear} presents the proof of
the Strichartz estimates for the nonlinear problem.

Appendix \ref{S:local_smoothing} contains a proof of the local smoothing
effect for \eqref{E:main} via the method of positive commutators, 
suggested to us by T. Alazard, N. Burq, and C. Zuily.  Appendices \ref{A:proofs} -\ref{A:energy} 
collect miscellaneous calculations in the course of the paper 
and linear energy estimates.

\newcommand{\U}{U^\bot}
\newcommand{\V}{U^\parallel}

\section{The hydrodynamic problem of surface water-waves}\label{S:formulation}

Recorded here is the approach taken in \cite{AM1} of the formulation of the water-wave problem 
when surface tension is acted on.
The idea is to employ a favorable parametrization of the moving surface
and choose convenient dependent variables. 

Throughout the paper, partial differentiation is represented 
either by the symbol $\partial$ or by subscript.
The complex plane $\mathbb{C}$ is identified with the real two-dimensional space $\mathbb{R}^2$,
whenever it is convenient to do so, 
via the mapping $\Phi: \mathbb{R}^2 \to \mathbb{C}$, $\Phi(x,y)=x+iy$. 
The conjugate of a complex number $z$ is denoted by $\bar{z}$. 

\subsection{The evolution of the moving surface and the vorticity strength}\label{SS:formulation}
The equation of the moving surface is written as $(x(t,\alpha),y(t,\alpha))$, where
$\alpha \in \mathbb{R}$ is the parametrization of the curve, and 
$\Phi(x(t,\alpha),y(t,\alpha))=z(t,\alpha)$. Let 
\[
s_\alpha^2=x_\alpha^2+y_\alpha^2 \quad \text{and}\quad 
\theta=\arctan (y_\alpha/x_\alpha)
\]
denote, respectively, the square of the arc length and 
the tangent angle that the curve forms with the horizontal direction.
The unit tangent and normal vectors of the curve are 
$\hat{\mathbf{t}}=(\cos \theta, \sin \theta)$ and $\hat{\mathbf{n}}=(-\sin \theta, \cos \theta)$,
respectively. 

The evolution equations of the moving surface are written 
\[
\partial_t (x,y)=\V\hat{\mathbf{t}} + \U\hat{\mathbf{n}}.
\]
In other words, $\V$ is the tangential velocity and $\U$ is the normal velocity of the moving surface.
Accordingly, 
\begin{align*}
\partial_t s_\alpha  = \partial_\alpha \V -\U \partial_\alpha \theta,  \qquad
\partial_t \theta =\frac{1}{s_\alpha} \partial_\alpha \U + \frac{\V}{s_\alpha} \partial_\alpha \theta,
\end{align*}
respectively. By insisting\footnote{
The normal velocity $\U$ is determined by the equations of motion,
while the tangential velocity $\V$ only serves to reparametrize the moving surface.
Adding an arbitrary tangential velocity does not change the shape of the surface, and thus
one may choose the tangential velocity to satisfy a certain condition.}
$\partial_t s_\alpha =0$, and furthermore, $s_\alpha=1$ 
for each $t \in \mathbb{R}_+$ and $\alpha \in \mathbb{R}$, 
we regard the evolution equation of the moving surface as 
\begin{equation}\label{E:theta}
\partial_t\theta =\partial_\alpha \U + \V\partial_\alpha \theta,
\end{equation}
where $\V$ is determined by solving $\partial_\alpha \V = \U\partial_\alpha \theta$.
Such a (renormalized) {arclength parametrization} is assumed initially, 
and the choice of tangential velocity will guarantee that the parametrization is maintained at later time.

\

Describing the dynamics on the moving surface, 
we employ the idea of vortex sheets in the two-fluid system, and we suppose that 
the interface separating the vacuum from the fluid
moves with different velocities along the tangential direction of the interface.

Let $\phi^{\pm}$ represent the velocity potentials of the upper and the lower fluids, respectively, 
and let $\rho^\pm$ be the densities of the upper and the lower fluids
and $p^\pm$ be the corresponding pressures. 
The Euler equations in the vacuum and the fluid region take the form 
\begin{equation}\label{E:bernoulli}
\partial_t \phi^\pm +\frac{1}{2}|\nabla \phi^\pm|^2+\frac{p^\pm}{\rho^\pm}=0,
\end{equation}
and the boundary conditions at the interface are written as
\[
[\nabla \phi^\pm]\cdot \hat{\mathbf{n}}=0 \quad \text{and} \quad [p]=S\partial_\alpha \theta,
\]
where $[\cdot]$ represents the jump of the quantity across the interface. 
We note that the arclength parametrization of the interface offers 
a particularly succinct  expression of the mean curvature.

Let $\gamma$ denote the {\em vortex sheet strength}\footnote{
The flow is irrotational. The vorticity, however, has a singular distribution supported on the interface.
The vortex sheet strength then measures concentration of vorticity along the interface.}. 
Introducing the Birkhoff-Rott integral\footnote{
In the recovery of the velocity from the vorticity distribution, 
we employ the Biot-Savart law to derive an integral expression, 
the limit of which at the interface is the Birkhoff-Rott integral.}
\begin{equation}\label{E:W}
\overline{\Phi}(\mathbf{W})(\alpha)=\frac{1}{2\pi i}\text{PV} \int^\infty_{-\infty} 
\frac{\gamma(\alpha')}{z(\alpha)-z(\alpha')} d\alpha',
\end{equation}
we express the limiting value of velocity at the interface as
\[
\nabla \phi^\pm (t, \Phi^{-1}(z)(t,\alpha))=
\mathbf{W}(t,\alpha) \pm \frac{1}{2}\gamma(t,\alpha) \hat{\mathbf{t}}.
\]
On the other hand, 
$\partial_t (x,y)=\mathbf{W}+(\V-\mathbf{W}\cdot \hat{\mathbf{t}})\hat{\mathbf{t}}$ and  
$\U=\mathbf{W}\cdot \hat{\mathbf{n}}$.

By combining Bernoulli's equation \eqref{E:bernoulli} with the boundary conditions
at the interface and by using the above notations, we derive the evolution equation of $\gamma$ 
\begin{equation}\label{E:gamma}
\partial_t \gamma=S\partial_\alpha^2 \theta
+\partial_\alpha((\V-\mathbf{W}\cdot \hat{\mathbf{t}})\gamma)
-2\mathbf{W}_t \cdot \hat{\mathbf{t}}-\frac{1}{2}\gamma\partial_\alpha\gamma
+2(\V-\mathbf{W}\cdot \hat{\mathbf{t}}) \mathbf{W}_\alpha \cdot \hat{\mathbf{t}}.
\end{equation}
The development is detailed in \cite[Appendix B]{AM1}.

In summary, the water-wave problem consists of \eqref{E:theta} and \eqref{E:gamma}.
A useful feature of the formulation is that surface tension enters the equation in the linear fashion.

\subsection{The system for the tangent angle and the modified tangent velocity}
\label{SS:system}
The choice of tangential velocity $\V$ produces in \eqref{E:gamma} 
nonlinear terms involving $\V-\mathbf{W}\cdot \hat{\mathbf{t}}$.
In order to express these terms in a more convenient way, 
we introduce the {\em modified tangential velocity}
\begin{equation}\label{D:u}
u=\frac{1}{2}\gamma-(\V-\mathbf{W}\cdot \hat{\mathbf{t}}),
\end{equation}
and we rewrite the system \eqref{E:theta} and \eqref{E:gamma}
in terms of $\theta$ and $u$, instead of $\gamma$.
Physically interpreted, $u$ measures the difference between 
the Lagrangian tangential velocity $\mathbf{W}\cdot \hat{\mathbf{t}} +\frac{1}{2}\gamma$ 
and tangential velocity $\V$ which guarantees arclength parametrization.
Once $(x(t,\alpha), y(t,\alpha))$ is given, the mapping $\gamma \mapsto u$ is one-to-one.

The first step is to approximate $\mathbf{W}$ in terms of the Hilbert transform. 
By expanding $\overline{\Phi}(\mathbf{W})$ in the Taylor fashion, one obtains
\begin{align*}
\overline{\Phi}(\mathbf{W})(\alpha)=& 
\frac{1}{2\pi i} \text{PV} \int^\infty_{-\infty}
\frac{\gamma(\alpha')}{z_\alpha(\alpha')(\alpha-\alpha')}d\alpha' \\
&+\frac{1}{2\pi i} \int^\infty_{-\infty} \gamma(\alpha')\left( \frac{1}{z(\alpha)-z(\alpha')}
-\frac{1}{z_\alpha(\alpha')(\alpha-\alpha')}\right)d\alpha'\\
:=&\frac{1}{2i}H\left(\frac{\gamma}{z_\alpha}\right)+\mathcal{K}[z]\gamma.
\end{align*}
Note that $\mathcal{K}[z]\gamma$ is not singular as the singularities 
in the expression of $\mathcal{K}[z]$ cancel.
Moreover, $\mathcal{K}[z]$ has the ``smoothing'' property
\begin{equation}\label{Es:K[z]}
\|\mathcal{K}[z]f\|_{H^s} \leq C(\|\theta\|_{H^{s+1-n}})\|f\|_{H^n} \qquad 
\text{for $s\geq 1$ and $n= 0,1$.}\end{equation}
The proof is very similar to that of \cite[Lemma 3.5]{Am}, and hence it is omitted.
The commutator operator
\[
[H, h]f(\alpha)=\frac{1}{\pi}\int^\infty_{-\infty} 
f(\alpha') \frac{h(\alpha')-h(\alpha)}{\alpha-\alpha'}d\alpha',
\]
has a similar smoothing property 
\begin{equation}\label{Es:H}
\|[H,h]f\|_{H^s}\leq C\|h\|_{H^{s+s'}}\|f\|_{H^{r-s'}}, \qquad
\text{for $s,s'\geq 0$ and $r>1/2$.}\end{equation}
The proof is found, for instance, in \cite[Lemma 2.14]{Yos1}.

The next step is to represent $\mathbf{W}_\alpha$ as
\begin{equation}\label{Es:W_alpha}
\mathbf{W}_\alpha \cdot\hat{\mathbf{n}} = \frac12H(\gamma_\alpha) +\mathbf{m}\cdot \hat{\mathbf{n}}
\quad\text{and}\quad
\mathbf{W}_\alpha \cdot\hat{\mathbf{t}} = 
-\frac12H(\gamma\theta_\alpha) +\mathbf{m}\cdot \hat{\mathbf{t}},
\end{equation}
where
\begin{equation}\label{D:m}
\overline{\Phi}(\mathbf{m})=z_\alpha \mathcal{K}[z] \left( 
\frac{\gamma_\alpha}{z_\alpha}-\frac{\gamma z_{\alpha\alpha}}{z_\alpha^2}\right)
+\frac{z_\alpha}{2i} \left[H, \frac{1}{z_\alpha^2}\right] 
\left(\gamma_\alpha-\frac{\gamma z_{\alpha\alpha}} {z_\alpha}\right).
\end{equation}
Indeed, by differentiating 
$\mathbf{W}=\frac{1}{2}H(\gamma \hat{\mathbf{n}}) +\text{(smooth remainder)}$
and using $\hat{\mathbf{n}}_\alpha=-\theta_\alpha \hat{\mathbf{t}}$ one obtains
\[
\mathbf{W}_\alpha = \frac{1}{2} H(\gamma_\alpha)\hat{\mathbf{n}}
-\frac{1}{2}H(\gamma \theta_\alpha)\hat{\mathbf{t}} +\text{(smooth remainder)}.
\]
The detailed calculation is found in \cite[Section 2.2]{Am}.

Using the results above, finally,
\eqref{E:theta} and \eqref{E:gamma} are written as 
\begin{subequations}\label{E:system}
\begin{align}
\partial_t u&=\frac{S}{2}\partial_\alpha^2 \theta -g\theta -u\partial_\alpha u 
+\partial_\alpha^{-1}(-r_2(t, \alpha)\partial_\alpha \theta
+(H\partial_\alpha u+r_1(t,\alpha))^2), \\
\partial_t \theta&=-u\partial_\alpha \theta+H\partial_\alpha u+r_1(t,\alpha),
\end{align}\end{subequations}
where
\begin{align}
r_1&(t,\alpha) = -H(\mathbf{m}\cdot \hat{\mathbf{t}} ) +\mathbf{m}\cdot \hat{\mathbf{n}}, \label{D:r_1} \\
r_2&(t,\alpha)=\mathbf{W}_t\cdot \hat{\mathbf{n}} 
+u\mathbf{W}_\alpha \cdot \hat{\mathbf{n}} +\frac12 \gamma \theta_t+\frac12\gamma u\theta_\alpha.
\label{D:r_2}
\end{align}
The detailed derivation is found in the proof of \cite[Proposition 2.1]{AM1}.

\subsection{Estimates for $r_1$ and $r_2$}
This subsection concerns the estimates of  the remainder terms in the system \eqref{E:system}.
We state the main result.

\begin{proposition}\label{P:remainders}
The remainders $r_1$ and $r_2$ in \eqref{D:r_1} and in \eqref{D:r_2}, respectively, satisfy
\begin{align}
\|&r_1\|_{H^s} \leq C(\|\theta\|_{H^2}, \|\theta\|_{H^{s+n}})(1+\|u\|_{H^{2-n}}) 
& & \text{for $s\geq 1$ and $n=0,1$}, \label{Es:r_1}\\
\|&r_2\|_{H^s} \leq C(\|\theta\|_{H^{s+2}})(1+\|u\|_{H^{s+1}})^2
& & \text{for $s\geq 1$.}\label{Es:r_2}\\
\intertext{Moreover, $r_2$ may be written as
$r_2=H\partial_t u +r_3,$ where}
\|&r_3\|_{H^s}  \leq C(\|\theta\|_{H^{s+1}})(1+\|u\|_{H^{s+1}})^2  & &\text{for $s\geq 1$}\label{Es:r_3}.
\end{align}
\end{proposition}

Our result is related to that in \cite{AM1}, but with the important difference that 
here $S>0$ is held fixed whereas in \cite{AM1} the estimates are uniform as $S \to 0$. 

\

The remainder term $r_1$ involves ``smoothing'' operators 
$\mathcal{K}[z]$ and $[H, \frac{1}{z_\alpha^2}]$.
On account of \eqref{Es:K[z]} and \eqref{Es:H}, it follows that 
\begin{align}
\|\mathbf{m}\|_{H^s} &\leq C(\|\theta\|_{H^{s+n}}) \|\gamma\|_{H^{2-n}} & & 
\text{for $s\geq 1$ and $n=0,1$}.
\label{Es:m} \\
\intertext{Then, it is immediate that}
\|r_1\|_{H^s} &\leq C(\|\theta\|_{H^2}, \|\theta\|_{H^{s+n}}) \|\gamma\|_{H^{2-n}} & & 
\text{for $s\geq 1$ and $n=0.1$}.
\label{Es:r_1a}
\end{align}

We further estimate $r_1$ in terms of $u$ (instead of $\gamma$) and $\theta$. 
Below is the basic regularity property of $\gamma$.

\begin{lemma}\label{L:regularity_gamma}
Let $S>0$ be held fixed. 
For $s\geq 1$, if $\theta \in H^{s+1/2}$, $u \in H^s$ and $\gamma \in H^{s-1}$ 
then $\gamma \in H^s$ and 
\[
\|\gamma\|_{H^s} \leq C\|u\|_{H^s} + C(\|\theta\|_{H^{s+1/2}}).
\]
\end{lemma}

Indeed, the definition of $u$ and \eqref{Es:W_alpha} yield that
$\partial_\alpha \gamma = 2\partial_\alpha u+H(\gamma \partial_\alpha \theta) 
-2\mathbf{m}\cdot \hat{\mathbf{t}}.$
The assertion then follows from \eqref{Es:m}.

The estimate \eqref{Es:r_1} finally follows by 
combining \eqref{Es:r_1a} with Lemma \ref{L:regularity_gamma}. 

A consequence of \eqref{Es:m} is that  $u=\frac12 \gamma+\text{(lower order terms)}$,
which is useful in the future consideration.

\begin{corollary}\label{C:regularity_V-Wt}
For $s\geq 1$, if $\theta \in H^{s+1/2}$, $u \in H^s$ and $\gamma \in H^{s-1}$ 
then $\V-\mathbf{W}\cdot \hat{\mathbf{t}} \in H^s$ and 
\[
\|\V-\mathbf{W}\cdot \hat{\mathbf{t}}\|_{H^s} \leq C(\|\gamma\|_{H^1}) \|u\|_{H^{s-1}}
+C(\|\theta\|_{H^s}).
\]
\end{corollary}
The assertion follows at once from $\partial_\alpha (\V-\mathbf{W}\cdot \hat{\mathbf{t}})=
-\mathbf{W}_\alpha \cdot \hat{\mathbf{t}}$.

\

The estimates for $r_2$ are more involved. 
Using \eqref{Es:W_alpha} and (\ref{E:system}b) we write
\begin{equation}\label{E:r_2_again} 
\begin{split}
r_2(t,\alpha)=\mathbf{W}_t \cdot\hat{\mathbf{n}} 
&+u\left(\frac12H(\gamma_\alpha)+\mathbf{m}\cdot \hat{\mathbf{n}} \right) \\
&+\frac12 \gamma(-u\theta_\alpha+Hu_\alpha+r_1(t,\alpha))+\frac12\gamma u\partial_\alpha \theta.
\end{split}\end{equation}
Much of our effort to estimate $r_2$ goes to show that 
the principal part of $\mathbf{W}_t \cdot \hat{\mathbf{n}}$,
and subsequently, the principal part of $r_2$ is $H\partial_t u$.

\begin{lemma}[Calculation of $\mathbf{W}_t \cdot \hat{\mathbf{n}}$]\label{L:W_t}
For $s\geq 1$, we have
\begin{align}
\|&\mathbf{W}_t \cdot \hat{\mathbf{n}}\|_{H^s} \leq C(\|\theta\|_{H^{s+2}}) +C(1+\|u\|_{H^{s+1}})^2
\label{Es:W_t}, \\
\Big\|&\mathbf{W}_t \cdot \hat{\mathbf{n}} -\frac12H(\gamma_t)\Big\|_{H^s}
\leq C(\|\theta\|_{H^{s+1}})+C(1+\|u\|_{H^{s+1}})^2. \label{Es:W_t_res}
\end{align}
\end{lemma}

\begin{proof}
By writing $\mathbf{W}_t$ in terms of $\gamma$ and $\mathcal{K}[z]$, we have 
\begin{alignat*}{2}
\mathbf{W}_t \cdot \hat{\mathbf{n}} =& \text{Re}(iz_\alpha \overline{\Phi}&&(\mathbf{W}_t)) \\
=& \text{Re} \Big( \frac{1}{2\pi}   && \hspace{-.05in} z_\alpha (\alpha) \text{PV} 
\int^\infty_{-\infty} \frac{\gamma_t (\alpha')}{z(\alpha)-z(\alpha')} d\alpha' 
 \Big) \\
& &&  \hspace{-.05in}- \text{Re} \left( \frac{1}{2\pi}z_\alpha(\alpha) \text{PV} 
\int^\infty_{-\infty} \gamma(\alpha') \frac{z_t(\alpha)-z_t(\alpha')}{(z(\alpha)-z(\alpha'))^2} d\alpha' \right) \\
=& \frac12 H(\gamma_t) &&  \hspace{-.05in}-\text{Re}\left(\frac{1}{2\pi}\text{PV}
\int^\infty_{-\infty} \frac{\gamma_t(\alpha')}{z_\alpha(\alpha')} 
\frac{z_\alpha(\alpha)-z_\alpha(\alpha')}{\alpha-\alpha'} d\alpha'\right) \\
& && \hspace{-.05in} +\text{Re}(iz_\alpha(\alpha) \mathcal{K}[z]\gamma_t) \\
& && \hspace{-.05in}+\text{Re}\left(\frac{1}{2\pi} z_\alpha(\alpha) \text{PV}
\int^\infty_{-\infty} \frac{\gamma(\alpha')}{z_\alpha(\alpha')} z_{\alpha t}(\alpha') 
\frac{1}{z(\alpha)-z(\alpha')} d\alpha'\right) \\
& && \hspace{-.05in}-\text{Re} \left( \frac{1}{2\pi} z_\alpha(\alpha) \text{PV} 
\int^\infty_{-\infty} \partial_{\alpha'} \left( \frac{\gamma(\alpha')}{z_\alpha(\alpha')}\right) 
\frac{z_t(\alpha)-z_t(\alpha')}{z(\alpha)-z(\alpha')} d\alpha'\right)\\
:=&  \frac12 H(\gamma_t) && \hspace{-.05in} +R_1+R_2+R_3+R_4.
\end{alignat*}
In the proof, we examine each $R_j$, $j=1,2, 3, 4$, separately.

In order to estimate $R_1$, we simplify the expression by introducing 
\[
q_3(\alpha, \alpha')=\frac{z_\alpha(\alpha)-z_\alpha(\alpha')}{\alpha-\alpha'} 
=\int^1_0 z_{\alpha\alpha}(\tau \alpha+(1-\tau)\alpha') d\tau.
\]
It is standard (see \cite{BHL}, for instance) that 
$$\|q_3\|_{H^{s-1}_\alpha}, \|q_3\|_{H^{s-1}_{\alpha'}}\leq C(\|\theta\|_{H^s}).$$
The Minkowski inequality and the Fubini theorem then apply to yield that
\begin{align*}
\int^\infty_{-\infty} |\partial_\alpha^s R_1(\alpha)|^2 d\alpha
\leq& C\int^\infty_{-\infty} \int^\infty_{-\infty} \left|\frac{\gamma_t(\alpha')}{z_\alpha(\alpha')}\right|^2
| \partial_\alpha^s q_3(\alpha, \alpha')|^2 d\alpha d\alpha' \\ 
\leq & C(\|\theta\|_{H^{s+1}}) \|\gamma_t\|_{L^2}^2,
\end{align*}
whence 
\[
\|R_1\|_{H^s} \leq C(\|\theta\|_{H^{s+1}}) \|\gamma_t\|_{L^2}.
\]

The smoothing property of $\mathcal{K}[z]$ in \eqref{Es:K[z]} implies to yield a similar estimate
\[
\|R_2\|_{H^s} \leq C(\|\theta\|_{H^{s+1}})\|\gamma_t \|_{L^2}.
\]
Further, in Appendix \ref{A:proofs} it is shown that
\begin{equation}\label{Es:gamma_t}
\|\gamma_t\|_{H^s}\leq C(\|\theta\|_{H^{s+2}})+C\|u\|_{H^{s+1}}\quad \text{for $s\geq 0$.}
\end{equation} 
By the above estimate for $s=0$, then, it follows that  
\[\| R_1\|_{H_s}, \|R_2\|_{H^s} \leq C(\|\theta\|_{H^{s+1}})+C\|u\|_{H^{s+1}} 
\qquad \text{for $s \geq 1$}. \] 

Next, upon writing $\mathbf{W}$ in terms of the Hilbert transform and $\mathcal{K}[z]$, we have 
\[
R_3=-\text{Re}\left( \frac12 z_\alpha(\alpha)H\left( \frac{\gamma}{z_\alpha^2} z_{\alpha t}\right)
+iz_\alpha(\alpha) \mathcal{K}[z]\left(\frac{\gamma}{z_\alpha^2}z_{\alpha t}\right) \right),
\]
whence for $s \geq 1$ the following inequalities hold:
\begin{align*}
\|R_3\|_{H^s} & \leq C(\|\theta\|_{H^s})\left(\left\| \frac{\gamma}{z_\alpha^2} z_{\alpha t}\right\|_{H^s}
+C(\|\theta\|_{H^{s+1}}) \left\| \frac{\gamma}{z_\alpha^2}z_{\alpha t}\right\|_{L^2}\right) \\
&\leq C(\|\theta\|_{H^{s+1}}) (\|\gamma\|_{H^s} \|\theta_t \|_{H^s}
+\|\gamma\|_{H^1}\|\theta_t\|_{L^2}) \\
&\leq C(\|\theta\|_{H^{s+1}}) (1+\|u\|_{H^{s+1}})^2.
\end{align*}
The last inequality uses (\ref{E:system}b). Indeed, 
\[
\|\theta_t\|_{H^s} \leq C\|u\|_{H^{s+1}}+C(\|\theta\|_{H^{s+1}}) \qquad \text{for $s\geq 1$}.
\]

In order to estimate $R_4$, similarly, we write 
\begin{equation*}
\begin{split}
R_4=\text{Re}\Big( \frac12 z_\alpha[H, z_t]& \Big(\frac{1}{z_\alpha} \partial_\alpha
\Big(\frac{\gamma}{z_\alpha}\Big)\Big) \\
&+iz_\alpha z_t \mathcal{K}[z]\left(\partial_\alpha\left( \frac{\gamma}{z_\alpha}\right)\right) 
-z_\alpha \mathcal{K}[z]\left( z_t \partial_\alpha \left(\frac{\gamma}{z_\alpha}\right)\right) \Big).
 \end{split}
 \end{equation*}
We claim that 
\begin{equation}\label{Es:z_t}
\|z_t\|_{H^s} \leq C\|u\|_{H^s} +C(\|\theta\|_{H^{s+1}}) \quad \text{for $s \geq 1$}.
\end{equation}
To see this, we write 
\[
z_t
=\Phi\left((\mathbf{W}\cdot \hat{\mathbf{n}}) \hat{\mathbf{n}}
+(\mathbf{W}\cdot \hat{\mathbf{t}}) \hat{\mathbf{t}} 
+(\V-\mathbf{W}\cdot \hat{\mathbf{t}}) \hat{\mathbf{t}}\right).
\]
By \eqref{Es:K[z]} and the result of Corollary \ref{C:regularity_V-Wt} then follows
\begin{align*}
\|& \mathbf{W}\|_{H^s} \leq C\|\gamma\|_{H^s} +C(\|\theta\|_{H^{s+1}})\|\gamma\|_{H^1}, \\
\|&\V-\mathbf{W}\cdot \hat{\mathbf{t}}\|_{H^s} \leq 
C(\|\gamma\|_{H^1})\|u\|_{H^{s-1}}+C(\|\theta\|_{H^s})
\end{align*}
for $s \geq 1$, which proves the claim.
With \eqref{Es:z_t} immediately follows that
\begin{align*}
\|R_4\|_{H^s} &\leq C(\|\theta\|_{H^{s+1}})\left( 
\|z_t\|_{H^s} \left\| \partial_\alpha \left(\frac{\gamma}{z_\alpha}\right) \right\|_{H^1} 
+ \left\|z_t \partial_\alpha \left(\frac{\gamma}{z_\alpha}\right)\right\|_{L^2}\right)\\
&\leq C(\|\theta\|_{H^{s+1}})(1+\|u\|_{H^2}+\|u\|_{H^s})^2.
\end{align*}

Finally, combining estimates for $R_1$ through $R_4$ yields that
\[
\|R_1+R_2+R_3+R_4\|_{H^s} \leq C(\|\theta\|_{H^{s+1}})+C(1+\|u\|_{H^{s+1}})^2.
\]
This together with \eqref{Es:gamma_t}  asserts \eqref{Es:W_t} and \eqref{Es:W_t_res}.
\end{proof}

Returning to the estimate of $r_2$, we estimate
terms in \eqref{E:r_2_again} other than $\mathbf{W}_t \cdot \hat{\mathbf{n}}$
in the usual way by using \eqref{Es:m}, and therefore \eqref{Es:r_2} follows.
To establish \eqref{Es:r_3}, we write $r_2=H\partial_tu +r_3$, where 
\[
r_3=\partial_t(\V-\mathbf{W}\cdot \hat{\mathbf{t}})+R_1+R_2+R_3+R_4.
\]
Since 
\[\partial_t \partial_\alpha (\V-\mathbf{W}\cdot \hat{\mathbf{t}}) =
-\frac12H(\gamma_t\theta_\alpha)-\frac12H(\gamma\theta_{\alpha t})
+\partial_t(\mathbf{m}\cdot \hat{\mathbf{t}}),
\]
it follows \eqref{Es:r_3}. This completes the proof of Proposition \ref{P:remainders}.

We end this subsection with estimates of the time derivatives of $\mathcal{K}[z]f$ and $[H,h]f$,
which will be useful in the following section.


\begin{corollary}\label{C:(K[z]f)_t}
For $s \geq 1$ we have
\begin{align}
\|\partial_t(\mathcal{K}[z]f)\|_{H^s} \leq &C(\|\theta\|_{H^{s+1}})(1+\|u\|_{H^{s+1}}
+\|f\|_{H^s}+\|\partial_t f\|_{L^2}), \label{Es:K_t} \\
\|\partial_t[H,h]f\|_{H^s} \leq & \|\partial_t h\|_{H^s}\|f\|_{H^1}+\|h\|_{H^{s+1}}\|\partial_t f\|_{L^2} 
\label{Es:H_t}
\end{align}
for $s \geq 1$. 
\end{corollary}
The proofs are in Appendix \ref{A:proofs}.

\begin{remark}[The dispersion relation]\label{SS:dispersion}\rm
We linearize \eqref{E:system} about a flat equilibrium $u=0$ and $\theta=0$ to obtain 
\[
\begin{cases}
\partial_t u=\frac{S}{2}\partial_\alpha^2 \theta -g\theta, \\
\partial_t \theta=H\partial_\alpha u.
\end{cases}
\]
By considering the plane-wave solution $u=\exp ik(\alpha-c(k)t)$, 
we arrive at the {\em dispersion relation}
\begin{equation*}\label{E:dispersion-r}
c(k)=\left(\frac{S}{2}|k|+\frac{g}{|k|}\right)^{1/2} \frac{k}{|k|},
\end{equation*}
where $c(k)$ is the phase velocity corresponding to the wave number $k$;
$\omega=c(k)k$ is the frequency.

Colloquially, when $S>0$, waves of high frequencies (short waves) propagate faster 
than waves of low frequencies (long waves). 
Broadening out the wave profile, it in consequence induces a certain ``smoothing effect''. 
When $S=0$, on the other hand, such a smoothing effect is not expected. 
Alternatively put, the above dispersion relation indicates that 
the linear system of surface water waves exhibits a regularizing effect
when the effects of surface tension are accounted for.

The present purpose is to quantitatively analyze such a smoothing effect 
in terms of integrability under the mixed Sobolev norms. 
\end{remark}

\section{Reformulation: the water-wave problem as a dispersive equation}\label{S:reformulation}

The formulation of the water-wave problem under surface tension ultimately takes 
the form of a second-order in time nonlinear dispersive equation, 
coupled with a transport-type equation. 

When both the effects of surface tension and gravity are present, $S>0$ and $g>0$, in \eqref{E:system}
the surface-tension term $\frac{S}{2}\partial_\alpha^2 \theta$ is of higher order
compared to the gravity term $g\theta$ and it dominates the linear dynamics.
For simplicity of exposition, thus, the effects of gravity are neglected and  
further the coefficient of the surface tension term is normalized so that 
\[g=0\quad \text{and} \quad \frac{S}{2}=1.\]

\subsection{Reduction to the dispersive equation}\label{SS:formulation-2}  
By differentiating (\ref{E:system}a) in the $t$-variable 
we obtain the {\em second-order} in time nonlinear dispersive equation
\begin{multline*}
 \partial_t^2u-H\partial_\alpha^3u=-2u\partial_\alpha \partial_t u-u^2\partial_\alpha^2 u
 -\partial_\alpha \theta \partial_\alpha^2 u \\
-3\partial_\alpha u\partial_t u - 3u(\partial_\alpha u)^2
+\partial_\alpha^2 r_1 +2u\partial_\alpha r_4+ur_4+r_5.
\end{multline*}
Here, $r_1$ is defined in \eqref{D:r_1}, and 
\begin{align}
r_4=&\partial_\alpha^{-1} (r_2(t,\alpha)\partial_\alpha \theta+
(H\partial_\alpha u+r_1(t,\alpha))^2),\label{D:r_4} \\
r_5=&\partial_\alpha^{-1}\partial_t(-H(\partial_t u)\partial_\alpha \theta
+r_3(t,\alpha) \partial_\alpha \theta+(H\partial_\alpha u+r_1(t,\alpha))^2). \label{D:r_5}
\end{align}
By \eqref{Es:r_1} and \eqref{Es:r_2} it follows that 
\begin{align}
\|r_4&\|_{H^s} \leq C(\|\theta\|_{H^{s+1}})(1+\|u\|_{H^s})^2 \qquad \text{for $s\geq 1$.}
\label{Es:r_4}
\end{align}

Further, the leading term of $\partial_\alpha^{-1}\partial_t(-H(\partial_t u) \partial_\alpha \theta)$ 
cancels -$\partial_\alpha \theta \partial_\alpha^2 u$ so that 
the highest-order nonlinear terms in the above equation do not involve $\theta$ explicitly.
Indeed, successive substitutions of $\partial_t u$ and $\partial_t \theta$ by \eqref{E:system} result in that
\begin{align*}
r_5=&-\partial_\alpha \theta H\partial_t(\partial_\alpha^2\theta -u\partial_\alpha u+r_4)за
-\partial_t\partial_\alpha \theta H(\partial_t u) \\ 
&+\partial_t r_3\partial_\alpha \theta
+r_3 \partial_t\partial_\alpha \theta +2(H\partial_\alpha u+r_1)
(H\partial_t \partial_\alpha u+\partial_t r_1) \\
=& \partial_\alpha \theta \partial_\alpha^2 u
-\partial_\alpha^{-1}(\partial_\alpha^2\theta \partial_\alpha^2 u) \\
&-\partial_\alpha^{-1}\Big((\partial_\alpha \theta) H(-2u\partial_\alpha \partial_t u
-u^2\partial_\alpha^2 u -\partial_\alpha\theta \partial_\alpha^2u \\
&\qquad \qquad \qquad \qquad -3\partial_\alpha u \partial_t u-3u\partial_\alpha^2 u
+\partial_t r_4 -u\partial_\alpha r_4-2r_4 \partial_\alpha u)\Big) \\
&+\partial_\alpha^{-1}(r_3-H\partial_t u)
(H\partial_\alpha^2 u-\partial_\alpha \theta\partial_\alpha u-u\partial_t u
-u^2\partial_\alpha u+\partial_\alpha r_1+ur_4 ) \\
&+\partial_\alpha^{-1}(\partial_\alpha \theta \partial_t r_3 )
+2\partial_\alpha^{-1}(H\partial_\alpha u+r_1)(H\partial_\alpha \partial_t u+\partial_t r_1).
\end{align*}
Therefore, we arrive at the following equation 
\begin{equation*}
\partial_t^2u-H\partial_\alpha^3u=-2u\partial_\alpha \partial_t u-u^2\partial_\alpha^2 u,
+R(u,\partial_tu, \theta).
\end{equation*}
The remainder is given as
\begin{equation}\label{D:R}
R(u, \partial_t u, \theta)=\partial_\alpha^2 r_1+u\partial_\alpha r_4+2ur_4
-\partial_\alpha^{-1}(\partial_\alpha^2u(\partial_tu+u\partial_\alpha u-r_4))+r_6,
\end{equation}
where 
\begin{equation}\label{D:r_6}
\begin{split}
\partial_\alpha r_6=&-(\partial_\alpha \theta) H(-2u\partial_\alpha \partial_t u
-u^2\partial_\alpha^2 u -\partial_\alpha\theta \partial_\alpha^2u-3\partial_\alpha u \partial_t u \\ 
&\qquad \qquad \qquad \qquad \qquad-3u\partial_\alpha^2 u
+\partial_t r_4 -u\partial_\alpha r_4-2r_4 \partial_\alpha u) \\
&+(r_3-H\partial_t u)(H\partial_\alpha^2 u-\partial_\alpha \theta\partial_\alpha u-u\partial_t u
-u^2\partial_\alpha u+\partial_\alpha r_1+ur_4 ) \\
&+\partial_\alpha \theta \partial_t r_3 
+2(H\partial_\alpha u+r_1)(H\partial_\alpha \partial_t u+\partial_t r_1).
\end{split}\end{equation}

The remainders $r_1$, $r_4$ and $r_6$ involve $\theta$, 
which incidentally is determined by solving (\ref{E:system}b) when $u$ is prescribed. 
As such, $R$ may be thought of depending $u$ and $\partial_t u$ only. 
In this sense, we write it as $R(u,\partial_t u)$.
The remainder term $R(u,\partial_t u)$ is of lower order compared to 
$u\partial_\alpha \partial_t u$ and $u^2\partial_\alpha^2 u$.
More precisely, in the following subsection we will show that
\begin{align}
\|R(u,\partial_t u&)\|_{H^s} \leq C(\|u\|_{H^{s+1}}, \|\partial_t u\|_{H^s})
\label{Es:R} 
\end{align}
for $s\geq 1$.

The water waves under surface tension is finally viewed as the nonlinear dispersive equation
\begin{equation}\label{E:u}
\partial_t^2u-H\partial_\alpha^3u=-2u\partial_\alpha \partial_t u-u^2\partial_\alpha^2 u+R(u,\partial_tu),
\end{equation}
where $R(u, \partial_t u)$, defined in \eqref{D:R}, is determined 
with the help of the transport-type equation 
\begin{equation}\label{E:transp}
\partial_t \theta=-u\partial_\alpha \theta+H\partial_\alpha u+r_1(t,\alpha).
\end{equation}

Nothing is lost in deriving \eqref{E:u} (coupled with \eqref{E:transp}) 
from \eqref{E:system}. To see this, we write \eqref{E:u} as the first-order in time system as
\[
\begin{cases}
\partial_t u +u\partial_\alpha u=v, \\
\partial_t v+ u\partial_\alpha v=H\partial_\alpha^3 u+R(u,\partial_t u)
-\partial_t u \partial_\alpha u-u\partial_\alpha^2u.\end{cases}
\]
Comparing the first equation of the above system with \eqref{E:system} dictates that 
the first equation of the above is equivalent to (\ref{E:system}a)
if we set $v=\partial_\alpha^2\theta +r_4$.
It is then straightforward to see that the second equation of the above system 
is equivalent to (\ref{E:system}b) up to constants of integration,
which are zero under the assumption that the wave profile and its derivatives vanish at infinity.

\begin{proposition}\label{P:equivalence}
The equation \eqref{E:u}, where $R(u, \partial_t u)$ is defined by \eqref{D:R}, 
is equivalent to \eqref{E:system}.
\end{proposition}

Similarly, the initial value problem of \eqref{E:u} prescribed with the initial conditions 
$u(0,\alpha)=u_0(\alpha)$ and $\partial_t u(0,\alpha)=u_1(\alpha)$ 
is equivalent to the initial value problem of \eqref{E:system} 
with the initial conditions $u(0,\alpha)=u_0(0)$ and $\theta(0,\alpha)=\theta_0(\alpha)$ 
provided that the compatibility condition
\begin{equation*}
u_1=\partial_\alpha^2 \theta_0-u_0\partial_\alpha u_0+
\partial_\alpha^{-1}(-r_2(0,\alpha)\partial_\alpha \theta_0+
(H\partial_\alpha u_0+r_1(0,\alpha))^2)
\end{equation*}
holds true.

The transport-type equation \eqref{E:transp} is to help to determine 
certain terms in the expression of $R$ in \eqref{E:u} in terms of $u$ and $\partial_t u$ only, 
and we only need \eqref{Es:R} in the forthcoming analysis.

As explained in Section 1, a  useful feature of the formulation \eqref{E:u} is that its dispersive character is visible in the linear part. 
Moreover, the highest-order nonlinear terms in \eqref{E:u} do not involve $\theta$ explicitly.

Another useful feature of \eqref{E:u} is that 
it suggests a natural expression for high energy of the nonlinear problem. See Section \ref{S:LWP}.

While \eqref{E:u} is dispersive, its nonlinearity is rather severe, 
and as such in the study of the dispersive property for \eqref{E:u}
we must take its nonlinearity into account. Indeed, we view \eqref{E:u} as
\[
\partial_t^2u-H\partial_\alpha^3u+2u\partial_\alpha \partial_t u+u^2\partial_\alpha^2 u=R(u,\partial_tu).
\]
That is, we view $2u\partial_\alpha \partial_t u$ and $u^2\partial_\alpha^2 u$
as ``linear'' components of the equation, but with variable coefficients
which happen to depend on the solution itself.
Then, we make efforts to establish the dispersive property for the linear operator
\[
\partial_t^2-H\partial_\alpha^3+2V(t,\alpha)\partial_t \partial_\alpha +V^2(t,\alpha)\partial_\alpha^2
\]
for a general class of functions for $V(t,\alpha)$.

\subsection{Estimates for the remainder}\label{SS:R}
The proof of \eqref{Es:R} involves estimates of various remainder terms in \eqref{D:R} 
in terms of $u$ and $\partial_t u$ only (instead of $\theta$). 
To this end, we estimate $\theta$ in terms of $u$ and $\partial_t u$ once and for good. 

\begin{lemma}\label{L:theta_estimate}
For $s\geq 0$ it follows that 
\begin{equation}\label{Es:theta}
\|\theta\|_{H^{s+2}} \leq C(\|u\|_{H^2})(1+\|\partial_t u\|_{H^s} +\|u\|_{H^{s+1}})^2.
\end{equation}
\end{lemma}

The proof is given in Appendix \ref{A:proofs}.
With the use of \eqref{Es:theta}, we obtain the estimates for the various remainders 
in terms of $u$ and $\partial_t u$ as
\begin{align*}
\|r_1\|_{H^s} & \leq C(\|u\|_{H^2}, \|u\|_{H^{s-1}}, \|\partial_t u\|_{H^{s-2}}), \\
\|r_2\|_{H^s} & \leq C(\|u\|_{H^{s+1}}, \|\partial_t u\|_{H^s}), \\
\|r_3\|_{H^s} & \leq C(\|u\|_{H^{s+1}}, \|\partial_t u\|_{H^{s-1}}), \\
\|r_4\|_{H^s} & \leq C(\|u\|_{H^2}, \|u\|_{H^s}, \|\partial_t u\|_{H^{s-1}}),
\end{align*}
where $s\geq 1$.
We also estimate for $\partial_t r_1$, $\partial_t r_2$, (and in turn, $\partial_t r_3$).

\begin{lemma}\label{L:r_t}
For $s\geq 1$, 
\begin{align}
\|& \partial_t r_1\|_{H^s} \leq C(\|\partial_t u\|_{H^1}, \|u\|_{H^{s+1}}, \|\partial_t u\|_{H^{s-1}}),
\label{Es:r_1t} \\
\|& \partial_t r_2\|_{H^s} \leq C(\|u\|_{H^{s+2}}, \|\partial_t u\|_{H^{s+1}}),\\
\label{Es:r_2t}
\| & \partial_t r_3\|_{H^s} \leq C( \|u\|_{H^{s+2}}, \|\partial_t u\|_{H^{s+1}}),\\
\| & \partial_t r_4\|_{H^s} \leq C( \|u\|_{H^{s+1}}, \|\partial_t u\|_{H^{s}}).
\end{align}
\end{lemma}
The proofs of \eqref{Es:r_1t} and \eqref{Es:r_2t} are given in Appendix \ref{A:proofs}. Therefore,
\[
\| r_6\|_{H^s} \leq C(\|u\|_{H^{s+1}}, \|\partial_t u\|_{H^s}),
\]
and \eqref{Es:R} follows.

\def\phi{\varphi}

\section{Construction of the dyadic frequency parametrix}\label{S:parametrix}

This section contains the detailed construction of semiclassical parametrices 
for the linearized water-wave equation.

\def\tu{u^j}
\def\ttu{\tilde{u}}
\subsection{The oscillatory-integral ansatz}
Let us invoke the standard notations $D_t = -i \partial_t$ and $D_\alpha=-i\partial_\alpha$
and let us denote
\begin{equation}\label{D:P}
P = D_t^2 -   iHD_\alpha^3+2V(t,\alpha)D_\alpha D_t + V^2(t,\alpha) D_\alpha^2,
\end{equation}
where the coefficient function $V \in H^l([0,T])H^k(\reals)$ is given 
for some $T>0$ fixed and for $l,k>0$ sufficiently large. 
We assume that $V$ is real-valued, and we tacitly identify 
$V$ with an $H^{l}_tH^{k}_\alpha$ extension supported in a slightly larger set in $t$.

The operator $P$ is obtained by replacing the nonlinear coefficient $u$ in 
\[\partial_t^2 -H\partial_\alpha^3 +2u \partial_\alpha \partial_t +u^2\partial_\alpha^2,\]
which defines the nonlinear equation \eqref{E:u}, by a variable coefficient $V(t,\alpha)$,
and thus it is related to the linearized equation of \eqref{E:u}.

Let the frequency cut-off function $\psi^0\in \Ci( \reals )$ satisfy 
$\psi^0(\xi) \equiv 1$ for $| \xi | \geq M+1$ and 
$\psi^0 (\xi) \equiv 0$ for $| \xi | \leq M$ for some $M>0$ large to be fixed later.  
Let the dyadic frequency cut-off function $\psi\in \Ci( \reals )$ 
satisfy $\psi(\xi) \equiv 1$ on $\xi \in [2^{-1/4}, 2^{1/4}]$,
be supported on $[2^{-3/4}, 2^{3/4}]$, and
\begin{equation*}
 \sum_{j \geq j_0} \psi(2^{-j} |\xi| ) \equiv 1 \quad \text{for } | \xi| \geq M+1, 
\end{equation*}
where $2^{j_0} \geq M$, and let
\begin{equation*}
\psi^j( \xi ) = \psi(2^{-j} | \xi | ).
\end{equation*}
That is, $\psi^j(\xi)=1$ on $\xi \in [2^{j-1/4}, 2^{j+1/4}]$ and it is supported on $[2^{j-3/4}, 2^{j+3/4}]$.
A function $f$ is said to satisfy the {\em dyadic-frequency localization} if
\begin{equation}\label{dy-loc}
\psi^j( D_\alpha) f = f.
\end{equation}

Let $j \geq j_0$ be held fixed throughout this section, where $2^{j_0}\geq M$. 
Let $\tu_0 \in L^2(\reals)$ and $\tu_1 \in H^{-3/2}(\reals)$ 
satisfy the dyadic frequency localization \eqref{dy-loc}. 
Our goal is to construct a {\em dyadic-frequency parametrix} to $Pu=0$
with the frequency-localized initial data $\tu_n$'s, $n=0,1$,
for frequencies comparable to $2^j$ and on a time scale comparable to
$2^{-j/2}$.
More precisely, we shall find a function $w^j$ approximately solving 
\[
\begin{cases}
Pu=0 \quad \text{ in } [0,2^{-j/2}T]_t \times \reals_\alpha, \\ 
u(0,\alpha)=\tu_0(\alpha) \quad \text{and}\quad \partial_t u(0,\alpha)=\tu_1(\alpha)
\end{cases}
\]
(with errors bounded in Sobolev spaces)
for the frequency interval $2^{j-2} \leq |\xi| \leq 2^{j+2}$. 
Here, we only consider the time interval $t \in [0,2^{-j/2}T]$, 
although the results apply on any semiclassical time interval of length $2^{-j/2} T$.  
The over all dependence on the coefficient function $V$ is in the fixed time interval $t \in [0,T]$. 
For simplicity of exposition, we often write $|\xi| \sim 2^j$ 
to mean the dyadic frequency band $2^{j-2} \leq |\xi| \leq 2^{j+2}$.

Motivated by the oscillatory integral representation in \eqref{E:ww-oscillatory0} 
for the zero-coefficient case, we make the ansatz 
\begin{equation}\label{w0-ansatz}
w^j(t,\alpha) = \frac{1}{2 \pi} \iint e^{-i\beta \xi }( e^{i\phi^{j,+}(t,\alpha,\xi)} f^{j,+}(\beta) 
+ e^{i \phi^{j,-}(t,\alpha,\xi)} f^{j,-}(\beta))   d\beta\, d \xi
\end{equation}
for the {\em (leading-order) parametrix}. 
Here, $\phi^{j,\pm}$ and $f^{j,\pm}$ satisfy the dyadic frequency localization \eqref{dy-loc}.
In order to satisfy the initial conditions, we insist that 
\[
\phi^{j,\pm}(0,\alpha, \xi) = \alpha \xi.
\]
The {\em phase functions} $\phi^{j,\pm}$ are taken in the class of rough symbols,
which is described below.


\subsection*{Symbol classes} 
For $k \geq 0$ and $m \in \reals$, 
denoted by $\s_{k}^{m}$ the {\em class of rough symbols} is defined to be the set
\[\s_k^m = \{ a(\alpha, \xi) \in  \lll \xi \rrr^{m} W^{k,\infty}_\alpha( \reals ) \Ci_\xi(\reals):
  | \partial_\alpha^{k'} \partial_\xi^{m'} a | \leq C_{k',m'} \lll \xi  \rrr^{m  - m'} \text{ for } k'  \leq k \},
\]
where $\lll \xi \rrr=(1+\xi^2)^{1/2}$.
Symbols in $\s_k^m$ are not necessarily smooth in the $\alpha$-variable
(as opposed to classical symbols), but
they share in common with classical symbols decay properties in the $\xi$-variable.
We write $\s_k$ for $\s_k^0$ when there is no ambiguity.  

The {\em quantization} of a symbol $a$ in $\s_{k}^{m}$ is the usual (left) quantization. 
It is initially defined as an operator on Schwartz functions $f$ as
\[
\Op (a)(\alpha,D) f(\alpha) = \frac{1}{2 \pi} \iint  a(\alpha, \xi) e^{i(\alpha-\beta) \xi } f(\beta)
\, d\beta \, d \xi,
\]
and then extended in the distributional sense.
We write $\Psi_k^m$ for the corresponding space of quantized operators.

The main property of the symbol classes $\s_k^m$ is the $L^2$ boundedness.
\begin{lemma}\label{ex-L2-lemma}
If $a \in \s_{k}^{0}$ for $k>0$ sufficiently large, then
$\Op(a)$ extends to a bounded linear operator 
\be
\Op (a) : L^2( \reals ) \to L^2 ( \reals )
\ee
with the operator norm depending on at most $k$ derivatives of $a(\alpha, \xi)$ 
measured in the $L^\infty$ norm.

\end{lemma}

The proof is exactly the same as in the setting of smooth symbols, 
keeping track of how many derivatives are used.

Finally, by $W^{l,\infty}_T \s_k^m$ we denote the space $W^{l,\infty}_t([0,T]) \s_k^m$, 
the space of $W^{l,\infty}([0,T])$ functions taking values in rough symbols in $\s_k^m$.

\

The main result of this section is the existence of the dyadic-frequency parametrix 
of the form in \eqref{w0-ansatz} with certain properties of the phase functions $\phi^\pm$. 

\begin{proposition}[Existence of the leading-order parametrix]
\label{parametrix-prop-2}Let $V \in H^{l}([0,T]) H^{k}(\reals)$ 
for some $T>0$ and $l,k \gg 1$ sufficiently large
and let $j\geq j_0\geq 1$, where $2^{j_0}\geq M$ is sufficiently large.

If $(\tu_0, \tu_1) \in L^2(\reals) \times H^{ - 3/2}(\reals)$ satisfy
the dyadic-frequency localization \eqref{dy-loc}, then
there exist the phase functions
$$\phi^{j,\pm}(t,\alpha, \xi)=  \alpha \xi \pm |\xi|^{3/2}(t+\vartheta^{j,\pm}(t,\alpha,\xi))$$ 
for $0 \leq t\leq 2^{-j/2}T$ and $2^{j-2} \leq | \xi | \leq 2^{j+2}$ 
with $\vartheta^\pm \in W^{l, \infty}_{2^{-j/2}T}\s_{k}^{0}$,
and $f^\pm$ which satisfy the dyadic-frequency localization \eqref{dy-loc} and the estimate
\begin{equation}\label{Es:f^pm}
\| f^\pm \|_{L^2_\alpha} \leq C_1(\| \tu_0 \|_{L^2_\alpha} +\|\tu_1\|_{H^{-3/2}_\alpha}),
\end{equation}
so that $w^j$ defined by \eqref{w0-ansatz} satisfies 
\begin{equation*}
\begin{cases} P w = E \quad \text{ in }  [0,{2^{-j/2}T}]_t \times \reals_\alpha, \\
w(0,\alpha) = \tu_0(\alpha)  \quad \text{and} \quad \partial_t w(0,\alpha) = \tu_1(\alpha)  
\end{cases}
\end{equation*}
with the pointwise error estimate 
\begin{equation}\label{error-2}
\|  E (t) \|_{L^2_\alpha } 
\leq C_2 ( \| \tu_0 \|_{H^{1}} +\|\tu_1 \|_{H^{-1/2}}), \qquad 0 \leq t
\leq 2^{-j/2}T .
\end{equation}
Here, $C_1,C_2 >0$ are polynomials in
$ \| V \|_{H^{l'}([0,T])H^{k'}(\reals)}$ for some values of $l'$ and $k'$ 
in the ranges $0 \leq l' \leq l$, $0 \leq k' \leq k$.
\end{proposition}

\begin{remark}\rm
As we will see later, on the short time scales, the oscillatory
integral defined by $\phi^{j,\pm}$ preserves dyadic localization (see Lemma \ref{freq-lemma}).  
In particular, multiplying the integral by
$2^{js}$, any estimate we prove on $w^{j,\pm}$ in $H^{s'}_\alpha$ 
has an immediate analogue in $H^{s'+s}_\alpha$.
\end{remark}
 
At first glance, the error $E$ looks bad, for it requires in \eqref{error-2}
one more derivative of the initial data.  However, the error is in the
inhomogeneity, and as such it only appears when we try to measure the
difference between the actual solution and the parametrix.  The 
{\em energy estimates} \eqref{Es:linear_energy1} associated to the
linear inhomogeneous problem with zero initial data, on the other hand, 
control $3/2$ more derivatives of the solution as compared to the inhomogeneity. 
Hence, $E$ is bad but controllable.

\begin{remark}[Remark on the ansatz]\label{R:ansatz}\rm
We explain why we take \eqref{w0-ansatz} as our parametrix. 

In order to solve $Pu=0$ with the initial data localized to high frequencies,
one would try a fine parametrix with {\em amplitudes} $A^\pm$ as 
\begin{align*}
w (t,\alpha)= \frac{1}{2 \pi} \iint e^{-i\beta \xi }( e^{i \phi^{+}(t,\alpha,\xi)}& 
A^{+}(t,\alpha,\xi)f^{+}(\beta) \\
& + e^{i \phi^{-}(t,\alpha,\xi)}A^{-}(t,\alpha,\xi) f^{-}(\beta) )\,   d\beta \,d \xi.
\end{align*}
We require that $\phi^\pm$ satisfy $\varphi^\pm(0,\alpha,\xi) = \alpha \xi$
and that $A^\pm(0,\alpha,\xi)$ and $A^\pm_t(0,\alpha, \xi)$ are elliptic, and as such 
the recovery of the initial conditions entails solving for $f^\pm$ 
a system of elliptic pseudodifferential equations.

Applying $Pu=0$ to our ansatz, we group terms according to their orders in $\xi$. 
The worst terms, produced when derivatives fall on the phase functions, 
make a first-order nonlinear equation for $\varphi^\pm$, 
commonly referred to as the {\it eikonal} or {\it Hamilton-Jacobi} equation.
The other terms form a linear equation, commonly referred to as the {\it transport} equation, 
for $A^\pm$ with coefficients depending on $\varphi^\pm$ and its derivatives.

The usual approach to solving the Hamilton-Jacobi equation is 
through the technique of generating functions for the associated Hamiltonian. 
The equation is, however, neither homogeneous nor polyhomogeneous 
(in $\varphi_t^\pm$ and $\varphi^\pm_\alpha$),
and as such solutions are found on a time scale comparable to $|\xi|^{-1/2}$.
See Lemma \ref{HJ-lemma}.
We thus construct phase functions (and amplitudes)
for each dyadic frequency band $|\xi| \sim 2^j$ on a time interval
comparable to $2^{-j/2}$.

With $\varphi^\pm$ determined, the usual approach to solving the transport equation is 
to expand $A^\pm$ as a formal series as
\[
A^\pm(t,\alpha,\xi) = \sum_{n\geq 0} A^{\pm,n}(t,\alpha,\xi)\xi^{-n/2}
\]
and determining $A^{\pm, n}$ recursively. In practice,
one takes a finite number of terms in the formal series and estimates the resulting error. 
In our application, we only take the very first term, $A^{\pm,0}\equiv 1$, in the amplitude expansion. 
As a result, we have \eqref{w0-ansatz} as our leading-order parametrix.

The full amplitude expansion can be computed 
assuming more regularity on $V$, but 
it results in polynomial growth of the lower order terms in the amplitude.
\end{remark}

We now construct the parametrix. 
Let us fix $j \geq j_0\geq 1$, where $2^{j_0}\geq M$ is sufficiently large. 
Our parametrix as well as the phase function depend on the dyadic-frequency band 
$2^{j-2} \leq |\xi| \leq  2^{j+2}$.
However, we simply write $w$ and $\phi^\pm$, respectively.
To avoid excessive notation, we furthermore consider only the term with the superscript $+$ 
and write $\phi =\phi^{j,+}$, $f = f^{j,+}$, and 
\begin{equation}\label{w+ansatz}
w(t,\alpha) = \frac{1}{2 \pi} 
\iint e^{-i\beta \xi } e^{i \phi(t,\alpha,\xi)}  f(\beta)  \, d\beta \, d \xi.
\end{equation}
The construction is completely analogous for the term with the superscript $-$.  
We further assume $\xi >0$.  After our construction is complete, 
it will be justified that the parametrix preserves the sets 
$2^{j-2} \leq \xi \leq 2^{j+2} $ and $ -2^{j+2} \leq \xi \leq -2^{j-2} $.
Thus, we tacitly assume $\xi>0$ and $2^{j-2} \leq \xi \leq  2^{j+2}$ throughout our construction.

We compute $Pw$ using the ansatz \eqref{w0-ansatz}. It is straightforward that
\begin{alignat}{3}
&D_t^2 w(t,\alpha) &&=  \frac{1}{2 \pi} \iint e^{-i\beta \xi } e^{i \phi} (  \phi_t^2 
   - i \phi_{tt}  ) f(\beta)  d\beta d \xi,  \label{D-eqn-1} \\
&D_\alpha^2 w(t,\alpha) &&=  \frac{1}{2 \pi} \iint e^{-i\beta \xi } e^{i \phi} ( \phi_\alpha^2 
  - i \phi_{\alpha\alpha} ) f(\beta)  d\beta d \xi, \\
&D_t D_\alpha w (t,\alpha) &&=  \frac{1}{2 \pi} \iint e^{-i\beta \xi } e^{i \phi}
( \phi_\alpha  \phi_t    -i \phi_{\alpha t}    ) f(\beta)  d\beta d \xi.  \label{D-eqn-n}
\end{alignat}
We recall that the subscripts mean partial differentiation.
In Appendix \ref{A:parametrix} we show that 
$iHD_\alpha^3w = - H\partial_\alpha^3 w$ can be expressed in a similar fashion as
\begin{equation} \label{dxh-eqn}
\begin{split}
iH &D_\alpha^3 w(t,\alpha)  \\ 
& =  \frac{1}{2 \pi}\iint  e^{-i\beta \xi} e^{i \phi(t, \alpha, \xi)} 
 \Big( |  D_{\alpha}|^3 + 3 | D_{\alpha}|^2 | \phi_\alpha| + 3 | D_{\alpha}| |\phi_\alpha|^2  
+ |\phi_\alpha|^3  \\ & \qquad \qquad \qquad \qquad \qquad \qquad 
+3i(| D_{\alpha} | + |\phi_\alpha|) \phi_{\alpha\alpha} 
- \phi_{\alpha\alpha\alpha}   \Big) 
   f(\beta)  d\beta d\xi.
\end{split}
\end{equation}
For simplicity of exposition, in what follows, 
we take $D_\alpha$ and $\phi_\alpha$ to be positive and avoid the absolute value and the sign. 
Again, this is justified in Lemma \ref{freq-lemma} by showing that 
negative frequencies and positive frequencies do not interfere. 

Applying $Pw=0$ to the results in \eqref{D-eqn-1}-\eqref{dxh-eqn}, 
we obtain the following equation for $\phi$:
\begin{equation} \label{A-eqn-1'}
\begin{split}
 \phi_t^2 -i \phi_{tt}& + 2V(t, \alpha) ( \phi_\alpha \phi_t -i \phi_{t\alpha}  
-i\phi_\alpha \partial_t - \partial_t \partial_\alpha )  \\
&+ V^2(t,\alpha)(\phi_\alpha^2 -i \phi_{\alpha\alpha} )-(\phi_\alpha^3 
 - 3i\phi_{\alpha} \phi_{\alpha\alpha}
 - \phi_{\alpha\alpha\alpha}) = 0.
\end{split}
\end{equation} 

In solving the above equation, we only consider terms
that are produced when first-order derivatives fall on the phase function only.
They make a nonlinear equation for $\phi$, commonly referred to as 
the {\em eikonal} or {\em Hamilton-Jacobi} equation. 
For other applications, a lower order error may be required.  
This can be brought about by considering a linear {\em transport} equation for amplitudes as well;
see Remark \ref{R:ansatz}.

\subsection{Construction of the phase functions}\label{SS:phase}
This subsection concerns solving the nonlinear {\em Hamilton-Jacobi} equation 
and determining the phase functions.

\begin{lemma}[The Hamilton-Jacobi equation]
\label{HJ-lemma}
Given a coefficient function $V$ in a bounded subset of $H^{l}([0,T])H^{k}(\reals)$
for some $T>0$ and for $l,k \gg 1$ sufficiently large, and 
given $j \geq j_0$ with $2^{j_0}\geq M>0$ sufficiently large, the following equation 
\begin{equation} \label{phi-hj-eqn}
\phi_t^2(t,\alpha,\xi) +  2V(t,\alpha) \phi_t(t,\alpha,\xi) \phi_\alpha(t,\alpha,\xi) 
+ V^2(t,\alpha)  \phi_\alpha^2(t,\alpha,\xi) -  \phi_\alpha^3(t,\alpha,\xi) = 0
\end{equation}
with the initial condition 
\[
\phi(0,\alpha,\xi) = \alpha \xi
\]
has two solutions $\phi^{j,\pm}$ for $2^{j-2} \leq \xi \leq 2^{j+2}$
on the time interval $0 \leq t \leq  2^{-j/2}T$.
Moreover, 
\begin{equation}
\phi^{j,\pm}(t,\alpha,\xi) = \alpha \xi \pm \xi^{3/2} (t +\vartheta^{j,\pm} (t, \alpha, \xi)) \label{phi-est}
\end{equation}
where $$\vartheta^{j,\pm}(t,\alpha,\xi) = \O(t (|V(t,
\alpha)| + |V(t, \alpha)|^2))  \in W^{l,\infty}_{2^{-j/2}T}\s_{k}^{0}$$ satisfies
$$ | \partial_\xi^{m'} \partial_\alpha^{k'} \vartheta^{\pm} | \leq C_{m'} 2^{-j/2}
\lll \xi \rrr^{-m'} \qquad \text {for $k' \leq k-1$}.$$
The derivatives of $\phi^{j,\pm}$ have the following properties  
\begin{align} 
\phi_\alpha^{j,\pm} (t,\alpha,\xi) = & \xi \pm \xi^{3/2} \vartheta^{\pm}_\alpha(t, \alpha, \xi) 
=  \xi (1 + \O(t)) \in W_{2^{-j/2}T}^{l,\infty}\s_{k-1}^1, \label{phi-est-1}  \\
\phi^{j,\pm}_t (t,\alpha,\xi)  =&  \pm \xi^{3/2}(1 +\vartheta_t^{\pm} (t,  \alpha, \xi)) 
 =\xi^{3/2}(1+\mathcal{O}(\xi^{-1/2}))  \in W^{l, \infty}_{{2^{-j/2}T}} \s_{ k-1}^{3/2} \label{phi-est-5}
\end{align}
for $2^{j-2} \leq \xi \leq 2^{j+2}$.
That is, $\phi_\alpha^{j,\pm}$ is of order $\xi$ on the dyadic frequency band $2^{j-2} \leq \xi \leq 2^{j+2}$ 
and on the small frequency-dependent time scale $0 \leq t \leq  2^{-j/2}T$ 
and $\phi_t^{j,\pm}$ is of order $\xi^{3/2}$. 
\end{lemma}

The class of rough symbols $\s_{ k}^m$ is defined in the previous subsection.

We recall that only $\phi=\phi^{j,+}$ is considered in writing \eqref{A-eqn-1'}. 
That enforces $\xi$ be positive, which is tacitly assumed 
in this subsection and the following one.

\begin{remark}\rm
A few comments are needed about Lemma \ref{HJ-lemma}.  

First, 
the result applies equally well on any time interval $I$ of length $2^{-j/2}T$
which is in the domain of $V$.  The initial conditions for
$\phi^\pm$ are then prescribed at the beginning of $I$ rather than at $0$.

Second, since \eqref{phi-hj-eqn} is quadratic in $\phi_t$, there are two phases $\phi^\pm$.  
This may be thought of as an analog of incoming/outgoing solutions 
of the wave equation, although we do not exercise that level of sophistication here.
We write the two branches of \eqref{phi-hj-eqn} as
\begin{equation}\label{phi-pm-root-eqn}
\phi^\pm_t = -V(t,\alpha)   \phi^\pm_\alpha \pm (\phi^\pm_\alpha)^{3/2}, 
\end{equation}
and we will work on the ``factorized'' Hamilton-Jacobi equation in sequel.

A standard method of existence for a Hamilton-Jacobi equation 
of the kind of \eqref{phi-pm-root-eqn} would be 
to construct $\phi^\pm$ as a generating function of a symplectic transformation 
which arises as a solution of the corresponding system of ordinary differential equations
\[
 \begin{cases}
\dot{\alpha} =  -V(t,\alpha)  + \frac{3}{2} \eta^{1/2}, \\
\dot{\eta} =   V_\alpha(t,\alpha) \eta
\end{cases}
\]
with the initial condition
\[
\alpha(0) = \beta, \qquad \eta(0) = \zeta
\]
for each $\beta, \zeta \in \reals$.  
Here and elsewhere, the dot above a variable denotes the differentiation with respect to the $t$-variable.
This system has a solution only for 
a time scale comparable to $\xi^{-1/2}$, which is what we are after,
but we desire a finer control. 

The idea lies in that \eqref{phi-pm-root-eqn} is sought of as  perturbation 
of $\phi_t^\pm=(\phi_\alpha^\pm)^{3/2}$ under a lower-order term.
With the initial condition $\phi(0,\alpha,\xi)=\alpha \xi$ the solutions of this equation 
are found explicitly to be $\phi_0(t,\alpha, \xi)=\alpha \xi \pm t \xi^{3/2}$.
Then, it seems reasonable to find solutions to \eqref{phi-pm-root-eqn}
as a perturbation of  $\phi_0$.
The added term in \eqref{phi-pm-root-eqn}, while being of lower order, 
destroys the homogeneity of the equation 
and it causes serious difficulties in the application of the Hamilton-Jacobi theory.  
\end{remark}

\begin{proof}
For simplicity of exposition, we will suppress the dependence of the phase 
on dyadic frequencies $\xi \sim 2^j$ and we will prove for the $+$ sign only.
Let $\phi = \phi^{j,+}$ denote the solution; the proof for the $-$ sign is identical.  

Let us consider the initial value problem
\begin{equation}\label{phi-eqn-root}
\phi_t =  -V(t,\alpha) \phi_\alpha + \phi_\alpha^{3/2}, \qquad
\phi (0,\alpha,\xi) = \alpha \xi, 
\end{equation}
where $\phi$ is a function of  $t, \alpha$ with a parameter $\xi$. 
As is remarked above,  we construct the solution as a perturbation of 
\[
\phi_0(t,\alpha,\xi)=\alpha \xi+ t\xi^{3/2},
\]
which is a solution to a homogeneous equation 
$\phi_t=\phi_\alpha^{3/2}$ with the same initial condition. 
Specifically, we make the ansatz
\[
\phi(t,\alpha,\xi)=\alpha \xi+ \xi^{3/2} (t +\vartheta(t,\alpha,\xi)),
\]
where 
\[
\vartheta(t,\alpha,\xi)=\tilde{\vartheta}(t,2^{-j/2}\alpha,\xi) 
\quad \text{and} \quad \tilde{\vartheta} \in \s_k.\]

Substituting our ansatz for $\phi$ into \eqref{phi-eqn-root} yields 
the following initial value problem 
\begin{equation}\label{gamma-HJ-eqn}
\begin{cases}
\tilde{\vartheta}_t =-\xi^{-1/2}V(t,2^{j/2}\alpha) (1+2^{-j/2}
\xi^{1/2}\tilde{\vartheta}_\alpha)+(1+2^{-j/2} \xi^{1/2}\tilde{\vartheta}_\alpha)^{3/2}-1, \\
\tilde{\vartheta}(0,\alpha,\xi)=0
\end{cases}
\end{equation}
where $\tilde{\vartheta}$ is evaluated at $(t,\alpha,\xi)$.
The corresponding Hamiltonian is
\[
\mathfrak{q}(t,\alpha,\eta) = - \xi^{-1/2} V(t,2^{j/2}\alpha) (1+2^{-j/2} \xi^{1/2}\eta) 
+ (1+2^{-j/2} \xi^{1/2}\eta)^{3/2}-1
\]
and the corresponding Hamiltonian system is
\begin{subequations}\label{Ham-system-1} 
\begin{gather}
\begin{cases}
\dot{\alpha} ={\displaystyle \frac{\partial \mathfrak{q}}{\partial \eta} }
= -2^{-j/2}V(t,2^{j/2} \alpha)  + \frac{3}{2} 2^{-j/2} \xi^{1/2}(1+2^{-j/2} \xi^{1/2}\eta)^{1/2}, \\
\dot{\eta} ={\displaystyle -\frac{\partial \mathfrak{q}}{\partial \alpha} }
=  2^{j/2}\xi^{-1/2}V_\alpha(t,2^{j/2}\alpha) (1+2^{-j/2} \xi^{1/2}\eta)
\end{cases} 
\intertext{with the initial conditions} 
\alpha(0) = \beta \quad \text{and}\quad \eta(0) = \zeta,
\end{gather}
\end{subequations}
where $\beta, \zeta \in \mathbb{R}$ and $\zeta \in [-\epsilon,\epsilon]$ for some $\epsilon>0$.  
We recall that the dot above a variable denotes the differentiation with respect to the $t$-variable.
Since under the assumption of $\xi \sim 2^j$ the right sides of (\ref{Ham-system-1}a)
satisfy the Lipschitz condition with respect to $\alpha$ and $\eta$ 
with the Lipschitz constants comparable to $2^{j/2}$, 
it is standard from the theory of ordinary differential equations  
that a unique solution of \eqref{Ham-system-1} 
exists on the time interval $0\leq t \leq  2^{-j/2}T$ 
for the range of the initial conditions given above, 
and the solution is at least as smooth as the right hand side.

We write $\alpha = \alpha^t(\beta, \zeta)$, $\eta = \eta^t(\beta, \zeta)$, and 
$\kappa^t(\beta, \zeta) = (\alpha,\eta)$
for the symplectomorphism given by the solution of \eqref{Ham-system-1}. 
Let $\omega$ be the $1$-form
\[
\omega = - \tau dt + \eta d\alpha + \beta d \zeta,
\]
and let $\Lambda$ be the surface
\[
\Lambda = \{ (t, \mathfrak{q}(t,\kappa^t(\beta, \zeta)), \kappa^t(\beta, \zeta), \beta , \zeta):
(t, \beta, \zeta) \in [0,2^{-j/2} T]_t \times \reals_\beta \times [-\epsilon, \epsilon]_\zeta \}.
\]
Since $\Lambda$ is a graph 
it is an embedded three-dimensional submanifold of $T^*\reals^3$. 
Since $d \omega$ is a symplectic structure on $T^* \reals^3$
and since the fact that $\kappa^t$ is symplectic implies $d \omega |_{\Lambda} =0$, additionally, 
$\Lambda$ is a Lagrangian submanifold. 
Such an embedded Lagrangian submanifold $\Lambda$ can be written as the graph of a closed
$1$-form, say $\sigma(t, \beta, \zeta)$.  Since $\reals^3$ is simply connected, the
Poincar\'e lemma implies that there exists $\tilde{\vartheta}(t, \beta, \zeta)$ such that
\[d \tilde{\vartheta}  = \sigma.\]

We will prove in Appendix \ref{A:parametrix} that the mapping 
\begin{equation}\label{y-x}
\beta \mapsto \alpha^t(\beta, \zeta)
\end{equation}
is invertible for each $\zeta \in [-\epsilon, \epsilon]$ and $0\leq t \leq 2^{-j/2}T$,
and we write $\beta =\beta(\alpha,\zeta)$.  
Comparing $\omega$ with $d \tilde{\vartheta}$ written in the $(t,\alpha,\zeta)$ coordinates,
we find that
\[
d \tilde{\vartheta} = -\tau dt +  \eta d\alpha + \beta d \zeta.
\]
This implies 
\be
\frac{\partial \tilde{\vartheta} }{\partial \zeta} = \beta(\alpha, \zeta), \quad
\frac{\partial \tilde{\vartheta} }{\partial \alpha} = \eta(\beta(\alpha, \zeta), \zeta), \quad \text{and}\quad 
\frac{\partial \tilde{\vartheta} }{\partial t} = \mathfrak{q}(t, \alpha, \eta),
\ee
which, in turn, implies for $\zeta \in [-\epsilon,\epsilon]$,
\begin{multline*}
\tilde{\vartheta}_t(t,\alpha, \zeta) 
 = -\xi^{-1/2} V(t,2^{j/2}\alpha)
    (1+2^{-j/2} \xi^{1/2}\tilde{\vartheta}_\alpha(t,\alpha, \zeta)) \\
+ (1+2^{-j/2} \xi^{1/2}\tilde{\vartheta}_\alpha)^{3/2}(t,\alpha, \eta) -1.
\end{multline*}
Therefore, \eqref{phi-eqn-root} follows once we substitute for
$\vartheta$ and $\phi$. 

The estimates \eqref{phi-est-1} and \eqref{phi-est-5} follow easily 
by differentiating the relations \eqref{phi-eqn-root} and
\eqref{Ham-system-1}, using the timescale $0 \leq t \leq 2^{-j/2}T$
and substituting $\vartheta (t,\alpha,\xi) =\tilde{\vartheta}(t,2^{-j/2} \alpha, \xi)$.
This completes the proof.
\end{proof}

\subsection{Finishing up the construction: recovery of initial conditions}
\label{SS:parametrix}

It remains to determine $f^{j,\pm}$ to satisfy the initial conditions
\[w^j(0,\alpha)=u_0^j(\alpha), \qquad D_t w^j(0,\alpha)=-iu_1^j(\alpha).\] 
Since $\phi^{j,\pm}(0,\alpha,\xi)=\alpha\xi$ it follows that 
\begin{align*}
w^j(0,\alpha)
& =  \frac{1}{2 \pi}\iint e^{i(\alpha-\beta) \xi}
(  f^{j,+}(\beta)  +  f^{j,-}(\beta)) d\beta\, d\xi  = f^{j,+}( \alpha) + f^{j,-}(\alpha), \\
D_t w^j (0,\alpha) 
 & =  \frac{1}{2 \pi}\iint e^{i(\alpha-\beta) \xi}
( \phi^{j,+}_t (0,\alpha, \xi) f^{j,+}(\beta)  + \phi^{j,-}_t (0,\alpha, \xi) f^{j,-}(\beta))d\beta  \,d \xi\\
& =:  A^{j,+}_1(\alpha,D_\alpha) f^{j,+} + A^{j,-}_1(\alpha,D_\alpha) f^{j,-} .
\end{align*}
Thus, we are led to the elliptic system for $f^{j,\pm}$ as 
\[
\begin{cases}
 f^{j,+} +  f^{j,-} = \tu_0, \\
A^{j,+}_1(\alpha,D_\alpha) f^{j,+} + A^{j,-}_1(\alpha,D_\alpha) f^{j,-} =  -i \tu_1.
\end{cases}
\]

The results of Lemma \ref{HJ-lemma} state that 
\[
\phi_t^{j,\pm}(0,\alpha,\xi)=\pm \xi^{3/2}(1+\O(\xi^{-1/2})) \in \s_{k-1}^{3/2}
\]
behave like the classical symbol $\pm \xi^{3/2}$. 
Consequently, $A^{j,\pm}_1$ are elliptic pseduodifferential operators,
which are approximately $\pm D_\alpha^{3/2}$,
and they have approximate inverses, denoted by $(A^{j,\pm}_1)^{-1}$.    
We only pause here to note that
the existence of approximate inverses here means the existence of {\em honest} inverses 
with the same estimates as the approximate inverse maps.  
Indeed, the error involved in our setting is $\O(\xi^{-1/2})$, which, 
in the dyadic-frequency band $\xi \sim 2^j$ is $\O(2^{-j/2})$.  
Hence the approximate inverse mapping has an
error which is bounded by $\O(2^{-j/2})$ on the appropriate Hilbert
space and the honest inverse can be obtained by the natural Neumann series.

Therefore, the operators $1-(A_1^{j,-})^{-1}A_1^{j,+}$ and $1-(A_1^{j,+})^{-1}A_1^{j,-}$
are bounded on $L^2$ and invertible with $L^2$ bounded inverses.  We set
\begin{align*}
f^{j,+} = & (1-(A_1^{j,-})^{-1}A_1^{j,+})^{-1}(\tu_0+i(A_1^{j,-})^{-1}\tu_1),\\
f^{j,-} = & (1-(A_1^{j,+})^{-1}A_1^{j,-})^{-1}(\tu_0+i(A_1^{j,+})^{-1}\tu_1).
\end{align*}
Then, \eqref{Es:f^pm} follows immediately.

This completes the construction of the parametrix and the proof of Proposition \ref{parametrix-prop-2}.

\subsection{The Fourier integral operators}

We establish the basic boundedness property and
propagation of singularities for our leading-order parametrix.  
Let us first establish a basic $L^2$-boundedness result.

\begin{lemma}
\label{L2-L2-bdd}
Let $F(t)$ for $0\leq  t  \leq 2^{-j/2}T$ be defined, initially on the Schwartz class, by the formula
\[
F(t)f(\alpha) = \iint e^{-i\beta \xi} e^{i \phi^{j} (t, \alpha, \xi)}
f(\beta)\, d\beta  d \xi 
\]
where $\phi^j$ is either of $\phi^{j,\pm}$ and is constructed in Lemma \ref{HJ-lemma}.  
Then $F(t)$ extends to a
bounded linear operator $L^2_\alpha \to L^2_\alpha$ and
\[
\| F(t) \|_{L^2_\alpha \to L^2_\alpha} \leq 1 + \O(t).
\]
\end{lemma}

The proof of the lemma shows that $F(t)$ is {\it almost unitarity} point-wise in $t$.  
  If $\phi^\pm$ have been constructed on an
interval of length $2^{-j/2}T$ in the domain of $V(t,\alpha)$, 
then the result holds equally true for each point on the interval.

\begin{proof}
We will prove the assertion for $\phi^j = \phi^{j,+}$, the proof for $\phi^{j,-}$ being analogous.  
Further, We will prove the estimate for $F^*$ rather than $F$, 
since some nontrivial cancellation makes the proof much easier.  
Indeed, we will show that $F^*(t) = F^{-1} A(\alpha, 2^{-j/2} D_\alpha)$ 
for a pseudodifferential operator $A$ with rough symbol, $A = 1 + \O(t)$.  Since
\begin{align*}
\| F^*(t) f \|_{L^2_\alpha}^2  = \lll F(t) F^*(t) f, f \rrr 
 \leq \| F(t) F^*(t) f \|_{L^2_\alpha} \| f \|_{L^2_\alpha},
\end{align*}
it suffices to prove the assertion for 
\begin{equation}
\label{F*F}
F(t) F^*(t) f(\alpha) = \iint e^{i(\phi^j(t,\alpha,\xi) -
  \phi^j(t,\alpha',\xi))} f(\alpha') d \alpha' d \xi.
\end{equation}
Recalling the results of $\phi$ from Lemma \ref{HJ-lemma}, we have
\begin{align*}
\phi^j(t,\alpha,\xi) -  \phi^j(t,\alpha',\xi) & = (\alpha - \alpha') \xi +
\xi^{3/2} (\vartheta^j(t, \alpha, \xi) - \vartheta^j(t, \alpha', \xi)) \\
& = (\alpha - \alpha')\xi(1 + \tilde{\vartheta}^j(t, \alpha, \alpha',\xi))
\end{align*}
where
$\tilde{\vartheta}^j = \O(t) \in  W^{l,\infty}_{2^{-j/2}T}\s_{k-1}^0$
satisfies
\[
\partial_\alpha^{k_1} \partial_{\alpha'}^{k_2} \tilde{\vartheta}^j = 2^{-j(k_1+k_2)/2}\O(t)
\]
for $k_1 + k_2 \leq k-1$.  
We perform the change of variables $\eta = \xi(1 + \tilde{\vartheta}^j(t, \alpha, \alpha',\xi))$ 
in \eqref{F*F} to obtain
\[
F(t) F^*(t) f(\alpha) = \iint e^{i(\alpha - \alpha') \xi} A(t, \alpha,
\alpha', \xi) f(\alpha') d \alpha' d \xi,
\]
where a symbol
\[
A (t,\alpha,\alpha',\xi)= 1 + \O(t) \in W^{l,\infty}_{2^{-j/2}T}\s_{k-1}^0
\]
satisfies
\[
\partial_\alpha^{k_1} \partial_{\alpha'}^{k_2} A = 2^{-j(k_1+k_2)/2}\O(t) 
\qquad \text{for $1 \leq k_1+k_2 \leq k-1$.}
\] 
Then the Calder\'on-Vaillancourt theorem implies the assertion.
\end{proof}

The following lemma regarding the Fourier integral operator related to
\eqref{w0-ansatz} shows how to pass derivatives through the
oscillatory integral, and it also justifies localizing to positive or negative $\xi$.

\begin{lemma}\label{freq-lemma}
Let $m \in \reals$, $l \geq2$, $k\geq 3$ and let $j \geq j_0\geq 1$ with $2^{j_0}\geq M$ sufficiently large.
Let $\phi^{j,\pm}(t,\alpha,\xi) \in W^{l,\infty}_{2^{-j/2}T}\s^{3/2}_{ k}$ 
be as constructed in Lemma \ref{HJ-lemma}.

Suppose that $B(\alpha,\xi) \in \Psi_{k'}^m$, 
where $k' \geq k+4$.
Then, for any $f \in H^m_\alpha(\reals)$ satisfying the dyadic frequency localization \eqref{dy-loc}
the following
\begin{multline*} 
B(\alpha,D_\alpha) \iint e^{-i\beta \xi} e^{i \phi^{j,\pm} (t, \alpha, \xi)} f(\beta)\, d\beta  d \xi \\
=  \iint e^{-i\beta \xi} e^{i \phi^{j,\pm} (t, \alpha, \xi)} 2^{jm}\tB(t, \alpha, \xi)f(\beta)\, d\beta  d \xi 
+ (Ef)(t,\alpha)
\end{multline*}
holds, where 
$\tB \in W^{l-1,\infty}_{2^{-j/2}T} \s_{ k-1}^0$ is supported in $c_0'
2^j \leq  \xi  \leq c_1' 2^j $
for some $0 < c_0' < c_1'$ independent of $j$, and $Ef$ satisfies the
dyadic-frequency localization \eqref{dy-loc} and $Ef$ enjoys the estimate 
\be
 \| E f \|_{ L^2([0,2^{-j/2}T])L^2_\alpha} \leq C \|f
 \|_{H^{m-1}_\alpha}.\ee

Furthermore, if $\tpsi^j ( D_\alpha) \in \Psi^0_{k'}$ is equal to $1$ in
the dyadic region $2^{j-2} \leq |\xi|  \leq 2^{j+2}$, then
\[
\tpsi^j( D_\alpha)  \iint e^{-i\beta \xi} e^{i \phi^{j,\pm} (t, \alpha,
  \xi)} f(\beta)\, d\beta  d \xi = \iint e^{-i\beta \xi} e^{i \phi^{j,\pm} (t, \alpha,
  \xi)} f(\beta)\, d\beta  d \xi 
 \]
modulo a lower order error.
\end{lemma}

As usual, the constant $C>0$ in the error estimate depends on up to $k$ derivatives in $\alpha$ of
$B(\alpha, \xi)$ and $k$ derivatives in $\alpha$ of $V(t,\alpha)$.
The requirement of 4 derivative comes from keeping track of the number of derivatives 
used to control the error terms in the Egorov theorem. 
Of course, the error estimate can be improved 
upon more careful use of the Egorov theorem.


\begin{proof}
As usual, we prove only for $\phi^j = \phi^{j,+}$; the proof for $\phi^{j,-}$ is similar.

Let us consider the oscillatory integral operator
\[
(Ff)(t,\alpha) = \iint e^{-i\beta \xi} e^{i \phi^j (t, \alpha, \xi)}  f(\beta)\, d\beta  d \xi .
\]
The results of Lemma \ref{HJ-lemma} say that
\[
\phi^j(t,\alpha,\xi)  = \alpha \xi + \xi^{3/2}(t + \vartheta^j(t, \alpha, \xi))
\]
with $\vartheta^j(t,\alpha,\xi)= \O(t) \in W^{l,\infty}_{2^{-j/2}T}\s_{k}$ and that 
\[ \phi^j_\alpha(t,\alpha,\xi) = \xi(1 + \O(t)) \in W^{l,\infty}_{2^{-j/2}T}\s_{k-1}^1.\] 
Since 
\begin{align*}
\phi^j_\xi (t,\alpha,\xi) = \alpha + \frac32 \xi^{1/2}(t + \vartheta^j(t,\alpha,\xi)) 
+ \xi^{3/2}  \vartheta^j_\xi(t,\alpha,\xi) ,
\end{align*}
it follows that
\begin{equation}\label{phi-x-xi-est}
\xi^2 \leq C \left( \left| \frac{ \partial \phi^j}{\partial \alpha}
  \right|^2 + \xi^2 \left| \frac{ \partial \phi^j}{\partial \xi}
  \right|^2 \right) \qquad \text{for}\quad  2^{j-2}\leq  \xi \leq 2^{j+2},
\end{equation}
where $C>0$ is independent of $\xi$.  In light of this and Lemma
\ref{L2-L2-bdd}, then $F$ is an elliptic Fourier integral operator.



Let us also choose $\chi_2( \xi )$ such that $\chi_2(\xi)=1$ on $[1/2,2]$ and it is supported in $[1/4, 4]$. 
Let $\chi_2^j(\xi) = \chi_2(2^{-j}\xi)$.
We define a modified phase function
\begin{equation}\label{E:tphi}
\tphi^j(t, \alpha, \xi) = \alpha \xi + t  \xi ^{3/2} + \chi_2^j(\xi) \vartheta^j(t, \alpha, \xi)  \xi^{3/2}.
\end{equation}  
Note that the modified phase function $\tphi^j$ is defined for all $\xi \geq M$ 
and $\tphi^j(t,\alpha,\xi)=\phi^j(t,\alpha,\xi)$ on the dyadic-frequncy band $2^{j-2} \leq \xi \leq 2^{j+2}$.

We note that the phase $\phi^j$ is the generating function of the symplectomorphism  
in the proof of Lemma \ref{HJ-lemma} in the dyadic band $2^{j-2} \leq
\xi \leq 2^{j+2}$, which is a lower order perturbation of the symplectomorphism
\begin{equation}
\label{tkappa}
 \alpha \mapsto \alpha + \frac32 t  \xi ^{1/2}, \qquad \xi \mapsto \xi. 
\end{equation}
The phase $\tphi^j$ generates the same symplectomorphism in the dyadic
region, and extends it to be \eqref{tkappa} in the rest of phase
space.  Let $\kappa^t$ be this extended symplectomorphism.

In light of the Egorov theorem \cite{Ego}, then
it follows that $F$ transforms symbols 
according to the symplectic transformation $\kappa^t$.

It remains to show that $\kappa^t$ maps dyadic frequencies to dyadic
frequencies and preserves the order of the symbol.  
Indeed, the $\xi$ component of $\kappa^t$ is $\xi(1 + \O(t ))$, whence 
\be
\{c_02^j \leq \xi  \leq c_1 2^j \} \subset
 \{ (\kappa^t)_2(\alpha, \xi) : c_0'2^j \leq \xi  \leq c_1' 2^j \} 
 \subset \{ c_0''2^j \leq \xi  \leq c_1'' 2^j \}
\ee
for some positive constants $c_0<c_1, c_0'<c_1'$, and $c_0''<c_1''$,
where $(\kappa^t)_2$ denotes the second component of $\kappa^t$.

Therefore, for any pseudodifferential operator $B \in \Psi_{k'}^m$, where $k' \geq k+4$, it follows that
\be
B(\alpha,D_\alpha) \iint  e^{-i\beta \xi} e^{i \phi^j (t, \alpha, \xi)}  f(\xi) d\beta d \xi 
=  \iint e^{-i\beta \xi} e^{i \phi^j (t, \alpha, \xi)} (\tB f)(\beta) d\beta d \xi
\ee
for some pseudodifferential operator $\tB \in \Psi_{k-1}^{m}$ with principal
symbol
\be
\sigma (\tB) = e(t,\alpha,\xi) (\kappa^t)^* \sigma(B),
\ee
where $e \in \s_{k-1}^0$ is elliptic on the support of $(\kappa^t)^*
\sigma(B)$.  This completes the proof.
\end{proof}

\begin{remark}\rm
The assertion of Lemma \ref{freq-lemma} holds true when replacing $\xi$ by $-\xi$. 
This justifies considering only positive $\xi$ in the construction of the parametrix in the previous section.
In what follows, we will drop the assumption that we work on positive $\xi$.
In other words, the phase functions take the form
\[
\phi^{j,\pm}(t,\alpha,\xi)=\alpha\xi \pm |\xi|^{3/2}(t+\vartheta^{j,\pm}(t,\alpha, \xi))
\]
and similarly for the amplitudes and others.
\end{remark}

\section{Strichartz estimates for the linearized equation}\label{SS:energy}

\def\tZ{\widetilde{Z}}

A dispersion estimate on a semiclassical time scale 
implies semiclassical Strichartz estimates for linearized water-wave problems under surface tension.  
Gluing these together Strichartz estimates on a fixed time scale
are established with loss in derivative. 

\subsection{Preparation for the proof}
Let us consider the initial value problem of the linear homogeneous equation 
\begin{equation}
\label{lin-ww-1}
\begin{cases} \partial_t^2u - H \partial_\alpha^3u +
    2V(t,\alpha) \partial_\alpha\partial_t u
     +  V^2(t,\alpha) \partial_\alpha^2 u  = 0, \\ 
u(0,\alpha) = u_0(\alpha) \quad \text{and} \quad
\partial_t u(0,\alpha) = u_1(\alpha),
\end{cases}
\end{equation}
where the coefficient function $V \in H^{l}([0,T])H^{k}_\alpha(\reals)$ is given 
for some $T >0$ fixed and for $l, k \geq 1$ sufficiently large,
Let us also consider the related 
initial value problem of the inhomogeneous equation with the zero data
\begin{equation}
\label{lin-ww-0-nh}
\begin{cases} \partial_t^2v - H\partial_\alpha^3v+  2V(t,\alpha) \partial_\alpha \partial_tv 
+ V^2(t,\alpha)\partial_\alpha^2v = R(t,\alpha), \\ 
v(0,\alpha) = 0 \quad \text{and} \quad \partial_t v(0,\alpha) = 0,
\end{cases}
\end{equation}
where $R \in L^2([0,T])H^s_\alpha( \reals )$ for some $s \geq 0$.

As a preliminary result, we establish the existence and uniqueness 
of the actual solutions of \eqref{lin-ww-1} and \eqref{lin-ww-0-nh} via the standard energy method.

\begin{theorem}[Existence and uniqueness for \eqref{lin-ww-1} and \eqref{lin-ww-0-nh}]\label{eu-thm}
Let $V \in H^{l}([0,T])H^{k}_\alpha(\reals)$ for some $T>0$ and for some $l,k>0$. 

For each pair of $u_0 \in H^{s+3/2}( \reals)$ and $u_1 \in H^{s}( \reals)$, where $0\leq s +3/2 \leq k$,
there exists a unique solution $u$ to \eqref{lin-ww-1} on the interval $0\leq t \leq T$ satisfying
\begin{equation}\label{Es:linear_energy0}
\| u \|_{L^\infty([0,T])H^{s +3/2}_\alpha(\reals)}+\|\partial_t u\|_{L^\infty([0,T])H^{s}_\alpha(\reals)} 
\leq C_1  ( \| u_0 \|_{H^{s+3/2}(\reals)} + \|u_1\|_{H^{s}(\reals)}).
\end{equation}
Furthermore, for each $R \in L^2([0,T])H^{s}_\alpha(\reals)$, 
there exists a unique solution $v$ to \eqref{lin-ww-0-nh} satisfying
\begin{equation}\label{Es:linear_energy1}
\| v_t \|_{L^\infty([0,T]) H^s_\alpha(\reals)} + \| v \|_{L^\infty([0,T])H^{s+3/2}_\alpha(\reals) } \leq C_2 \| R \|_{ L^2([0,T])H^{s}_\alpha(\reals)}.
\end{equation}
Here, $C_1, C_2>0$ are polynomial in $\|
V\|_{L^{\infty}([0,T])W^{s,\infty}_\alpha(\reals)}$ and $\| V_t \|_{L^{\infty}([0,T])W^{s,\infty}_\alpha(\reals)}$.
\end{theorem}

The existence and uniqueness is standard by combining 
the energy estimate \eqref{Es:linear_energy0} and \eqref{Es:linear_energy1}
with regularization and a Galerkin approximation. 
The detailed proofs of \eqref{Es:linear_energy0} and \eqref{Es:linear_energy1} 
are in Appendix \ref{A:energy}.

The main result is in this section is semiclassical Strichartz estimates 
for the linearized  water-wave problem under surface tension. 

\begin{theorem}\label{T:general-str}
Let $V \in H^l([0,T])H^k_\alpha(\reals)$ for some $T>0$ and for $l,k \gg 1$ sufficiently large
and let $j\geq j_0$, where $j_0 \gg 1$ is sufficiently large.  
Suppose that $U \in L^\infty([0,2^{-j/2}T]) H^{1/2p}(\reals)$ solves
\begin{equation}\label{E:U-eqn-1}
\begin{cases}
\partial_t^2 U- H\partial_\alpha^3 U +  2V(t,\alpha) \partial_\alpha \partial_t U
+ V^2(t,\alpha)\partial_\alpha^2 U = R(t,\alpha) \quad \text{for }  t \in [0, 2^{-j/2} T], \\
U(0,\alpha)= U_0(\alpha)\quad \text{and} \quad \partial_t U(0,\alpha) = U_1(\alpha),
\end{cases}
\end{equation}
where $(U_0, U_1) \in H^{1/2p}(\reals) \times H^{1/2p-3/2}(\reals)$ and 
$R \in L^2([0,T])H^{1/2p-3/2}(\reals)$. Suppose, further, that
$U$, $R$, $U_0$, and $U_1$ satisfy a dyadic frequency localization \eqref{dy-loc}
at frequency $2^j$.  Then, $U$ enjoys the estimate
\begin{equation}\label{E:short-ts-str}
\| U \|_{L^p([0, 2^{-j/2}T]) L^{q}_\alpha} \leq
C (\| U_0 \|_{H^{1/2p}_\alpha} + \| U_1 \|_{H^{ 1/2p- 3/2}_\alpha}
 + \| R \|_{L^1([0, 2^{-j/2}T]) {H^{1/2p-3/2}_\alpha}}),
\end{equation}
where $(p,q)$ satisfies 
\begin{equation*}
\frac{2}{p} + \frac{1}{q} = \frac{1}{2},
\end{equation*}
and 
$C$ depend on $p,q$ and the Sobolev norms of $V$ of orders bounded by $l$ and $k$.
\end{theorem}

Moreover, gluing together these estimates on $2^{j/2}$ many time intervals of length $2^{-j/2}T$ 
we obtain Strichartz estimates on a fixed time scale with loss in derivative.

\begin{corollary}
\label{C:homog-str}
Let $V \in H^l([0,T])H^k_\alpha\reals)$ for some $T>0$ and for some $l,k \gg 1$ sufficiently large.

If $(U_0, U_1) \in H^{1/p}(\reals) \times H^{-3/2+1/p}(\reals)$, then 
the solution $U\in L^\infty([0,T])H^{1/p}(\mathbb{R})$ 
of the initial value problem \eqref{lin-ww-1} satisfies the inequality
\begin{equation}\label{E:main-est-1}
\| U \|_{L^p([0,T])L^q(\reals)} \leq C_1( \|
U_0 \|_{H^{1/p}(\reals)} + \| U_1 \|_{H^{-3/2 + 1/p}(\reals)} ).
\end{equation}

If $R\in L^2([0,T])H^{-3/2+1/p}(\reals)$, then
the solution $v\in L^\infty([0,T])H^{ 1/p}(\mathbb{R})$ 
of the initial value problem \eqref{lin-ww-0-nh} satisfies the inequality
\begin{equation}\label{E:main-est-2}
\| v\|_{L^p([0,T])L^q(\reals)} \leq C_2 \| R \|_{L^1([0,T])H^{-3/2 + 1/p}(\reals)} .
\end{equation}
Here, $(p,q)$ satisfies the admissibility condition in Theorem \ref{T:general-str} with $q < \infty$, and
$C_1, C_2$ depend on $p,q$ and the Sobolev norms of $V$ of orders bounded by $l$ and $k$.
\end{corollary}

\subsection{Proof of Theorem \ref{T:general-str} and Corollary \ref{C:homog-str}}
We explain our strategy to proving Theorem \ref{T:general-str}. 
We first demonstrate that the homogeneous dyadic-frequency parametrix 
in Proposition  \ref{parametrix-prop-2} satisfies an $L^1_\alpha \to L^\infty_\alpha$ dispersion estimate
on a semiclassical time interval of length  $2^{-j/2}$.  
Then we use Theorem \ref{T:general-sc-str} to deduce the semiclassical Strichartz estimates for the parametrix.
We use the energy estimates to show that 
the parametrix is sufficient to estimate the actual solution in the Strichartz norms on semiclassical time scales.  To prove Corollary \ref{C:homog-str}, 
we sum up the estimates on semiclassical time scales to obtain an estimate on a fixed time scale 
with loss in derivative from the summation.

We begin by proving the short-time dispersion estimate.  Without loss in generality, we work on the time interval $t \in [0, 2^{-j/2}T]$.  The analogous statments are true for any interval of length $2^{-j/2}T$ in $[0,T]$.  
In order to estimate various quantities involving the parametrix \eqref{w0-ansatz} 
we have found it expedient to ``glue'' together on the time scale $0\leq t \leq 2^{-j/2}T$ 
the phases $\phi^{j,\pm}$, which are supported on $2^{j-2}\leq |\xi| \leq 2^{j+2}$,  
to obtain global phases with similar properties.  
As usual, we avoid excessive notation by considering the $\phi^{j,+}$ case only
and writing $\phi^j = \phi^{j,+}$ and similarly for other quantities with the $\pm$ parity.  
The analysis in the $\phi^{j,-}$ case is completely analogous.

For $\epsilon>0$ small, let us choose $\chi_1(t)$ such that $\chi_1(t)=1$ on $[0,T]$ 
and it is supported in $[0-\epsilon, T+\epsilon]$. 
Recall that $\chi_2( \xi )$ is such that $\chi_2(\xi)=1$ on $[1/2,2]$ and it is supported in $[1/4, 4]$. 
We need $\tchi_2$ satisfying $\tchi_2(\xi)\equiv 1$ on $\supp \chi_2(\xi)$ 
and $\supp \tchi_2 \subset [1/4 -\epsilon, 4 + \epsilon]$.  
Let $\chi_1^j(t) =\chi_1(2^{j/2}t)$, 
$\chi_2^j(\xi) = \chi_2(2^{-j}\xi)$, and $\tchi_2^j(\xi) = \tchi_2(2^{-j}\xi)$.  
We define a modified phase function
\begin{equation}\label{E:tphi2}
\tphi^j(t, \alpha, \xi) = \alpha \xi + t | \xi| ^{3/2} + 
\chi_1^j(t) \chi_2^j(|\xi|) \vartheta^j(t, \alpha, \xi) | \xi |^{3/2},
\end{equation}
where $\vartheta^j = \vartheta^{j,+}$ has been constructed in Proposition \ref{parametrix-prop-2}.   The modified phase function $\tphi^j$ is defined for all $t$ and all $|\xi | \geq M$ 
and we claim it has the properties 
\begin{equation}
\label{E:vartheta-t-1}
 \vartheta^j(t,\alpha,\xi)=t\mathcal{O}(|t|+|\xi|^{-1/2})  \in W^{l,\infty}_{2^{-j/2}T} \s^0_{k}.
\end{equation}
To see this, note that 
\begin{equation}
\label{E:vartheta-t-2}
\vartheta^j_t(t,\alpha,\xi)=\mathcal{O}(|t|+|\xi|^{-1/2}).
\end{equation}
Indeed, a simple substitution of $\phi$ in the eikonal equation yields that $\vartheta$ satisfies 
\begin{equation}
\label{E:theta2}
\vartheta_t(t,\alpha,\xi) = - |\xi|^{-1/2} V(t,\alpha) (1 + \vartheta_\alpha) + (1 + \vartheta_\alpha)^{3/2} -1,
\end{equation}
where $\vartheta_\alpha (t,\alpha,\xi)= \O(t)$.  
Since $\vartheta(0, \alpha, \xi) = 0$ by construction, $\vartheta(t,\alpha,\xi)= t\O(| t | + |\xi|^{-1/2})$ as claimed.

We redefine $w^{j,+}$ by replacing the phase function $\phi^j$ by $\tphi^j$.  
The new oscillatory integral agrees with the old one for $0 \leq t \leq 2^{-j/2}T$ 
and for the data $u_0^j$ and $u_1^j$ satisfying the dyadic-frequency localization \eqref{dy-loc}, 
but it has the virtue of being globally defined so that we can use the theory of Fourier integral operators.  

Let us consider the Fourier integral operator
\begin{equation}\label{int-10}
(Ff)(t,\alpha)=\frac{1}{2\pi} \iint e^{-i\beta \xi} e^{i \tphi^j(t,\alpha,\xi)}
\chi_1(t) \tchi_2^j(\xi) f(\beta)\, d\beta d \xi .
\end{equation}
It is standard from the results of Lemma \ref{L2-L2-bdd} and Lemma \ref{freq-lemma} that
\[
\| Ff(t) \|_{H^s_\alpha} \leq (1 + \O(t)) \| f \|_{H^s_\alpha}
\]
for $s \geq 0$ and each $0 \leq t \leq T$.  
As in Section \ref{SS:idea-str} for the Schr\"odinger equation, we write the oscillatory integral as
\[
F f(t, \alpha) = \int K(t, \alpha, \beta ) f(\beta) d\beta,\]
where
\begin{equation}\label{E:K}
K(t, \alpha, \beta) = \int e^{-i\beta \xi} e^{i \tphi^j(t,\alpha,\xi)} \chi_1(t) \tchi_2^j(\xi) d \xi.
\end{equation} 

\begin{lemma}[Microlocal dispersion estimates]\label{disp-lemma}
The kernel in \eqref{E:K} satisfies 
\[
\| K \|_{L^\infty_\beta} \leq C 2^{j/4}t^{-1/2} \qquad \text{ for } | t | \leq 2^{-j/2} T ,\]
so that
\[
\|F f(t) \|_{L^\infty_\alpha} \leq C 2^{j/4}t^{-1/2} \| f \|_{L^1_\alpha} \qquad \text{for $| t | \leq 2^{-j/2}T$.}
\]
\end{lemma}

\begin{proof}
The proof uses the method of stationary phase together with the assumption that $t$ is in an interval of length $2^{-j/2}T$ to estimate the error terms.  We assume as usual that $\xi$ is large and positive;
the large negative case is treated similarly.

We next compute
\begin{align*}
\partial_\xi (\tphi^j(\xi;t,\alpha)-\beta \xi)  = \alpha - \beta + \frac{3}{2} t |\xi|^{1/2} 
& + \frac{3}{2} \chi_1^j(t) \chi_2^j(|\xi|)|\xi|^{1/2}  \vartheta^j(t, \alpha, \xi) \\ 
&+ \chi_1^j(t) 2^{-j}\chi_2'(2^{-j}|\xi|)  |\xi|^{3/2} \vartheta^j(t, \alpha, \xi) \\
& + \chi_1^j(t) \chi_2^j(|\xi|)  | \xi |^{3/2}  \partial_\xi \vartheta^j(t, \alpha, \xi) \\
=\alpha - \beta + \frac{3}{2} t |\xi|^{1/2} &+t|\xi|^{1/2}\O(|t|+|\xi|^{-1/2})
\end{align*}
on $|t| \leq 2^{-j/2} T$ and for $2^{j-1}\leq |\xi| \leq 2^{j+1}$. 
The second equality uses that $\vartheta^j \in \s_k^0$, the $t$
localization due to $\chi_1^j(t)$, and \eqref{E:vartheta-t-1}-\eqref{E:vartheta-t-2}.
A critical point of the phase $\tphi^j(\xi;t,\alpha)-\beta \xi$, therefore, is given by
\[
\xi_c ^{1/2}= \frac23 \frac{ \beta - \alpha }{t} (1 +\O(2^{-j/2})).
\]
We observe that the localization of $\tchi_2^j$ implies 
that the integrand is zero unless $\xi \sim 2^j$, and hence $\xi_c \sim 2^j$. 

Next, we compute for $| t | \leq 2^{-j/2} T$: 
\begin{align*}
\partial_\xi^2 (\tphi^j(\xi; t,\alpha) - \beta \xi) = \frac{3}{4} t |\xi|^{-1/2} 
&+\frac34 \chi_1^j(t)\chi_2^j(|\xi|)| \xi |^{-1/2}   \vartheta^j(t,\alpha, \xi) \\
&+ \chi_1^j(t) 2^{-2j}\chi_2''(2^{-j}|\xi|)  |\xi|^{3/2} \vartheta^j(t, \alpha, \xi) \\
&+ \chi_1^j(t) \chi_2^j(|\xi|)  | \xi |^{3/2} \partial^2_\xi \vartheta^j(t, \alpha, \xi)\\
&+3 \chi_1^j(t) 2^{-j}  \chi_2'(2^{-j}|\xi|) |\xi|^{1/2} \vartheta^j(t, \alpha, \xi) \\
& +\chi_1^j(t)  2^{-j}\chi_2'(2^{-j}|\xi|) | \xi |^{3/2} \partial_\xi \vartheta^j(t, \alpha, \xi) \\
=  \frac{3}{4} t |\xi|^{-1/2} &(1 + \O(|t|+|\xi|^{-1/2})),
\end{align*}
where we have again used the $t$ localization of $\chi_1^j(t)$ and
\eqref{E:vartheta-t-1}-\eqref{E:vartheta-t-2}.  
Evaluating this at the critical point, we obtain 
\begin{align*} 
\partial_\xi^2 (\tphi^j(\xi; t,\alpha) - \beta \xi) \big|_{\xi=\xi_c} =& 
\frac{3}{4} t \cdot  \frac 32 \frac{t}{\beta-\alpha} (1 + \O(2^{-j/2})) \\
= & \frac{9}{8} \frac{t^2}{\beta-\alpha}(1 + \O(2^{-j/2})) .
\end{align*}
Thus, the critical point is nondegenerate for $t>0$.

Applying the method of stationary phase\footnote{Stationary phase gives a bound of $2^{j/4}t^{-1/2}$ plus a reminder term. Since we localize both in time and space this reminder is even smaller than $t^{-1/2}$.} completes the proof of the lemma, once we observe that 
the restriction $| t | \leq 2^{-j/2}T$ combined with the restriction $\xi_{\text{c}} \sim 2^j$ implies that 
$ (\beta - \alpha)^{1/2} / t^{1/2} $ is bounded by $2^{j/4}$ for this range of $t$.
\end{proof}

In order to apply Theorem \ref{T:general-sc-str}, we need to control $F(t') F^*(t)$.

\begin{lemma}\label{disp-lemma-2}
The Fourier integral operator $F$ given in \eqref{int-10} satisfies the estimate
\[
\| F(t') F^*(t) f(\alpha) \|_{L^\infty_\alpha} \leq C 2^{j/4}| t' - t |^{-1/2} \|f \|_{L^1_\alpha}
\]
for $| t' | , | t | \leq 2^{-j/2} T$.
\end{lemma}

\begin{proof}
We write
\[
F(t') F^*(t) f( \alpha) = \iint e^{i (\tphi^j(t', \alpha, \xi) - \tphi^j(t,\alpha', \xi))} 
\chi_1(t)\chi_1(t') (\tchi_2^j( \xi ))^2 f( \alpha') d \alpha' d \xi
\]
and its corresponding kernel
\[
K_1( \alpha, \alpha') = 
\int e^{i (\tphi^j(t', \alpha, \xi) - \tphi^j(t,\alpha', \xi))} 
\chi_1(t)\chi_1(t') (\tchi_2^j( \xi ))^2  d \xi.
\]
We cannot apply the same argument as in Lemma \ref{disp-lemma} to
estimate $K_1$,
since the difference of the two phases depends on $\alpha$ and $\alpha'$. 
However, the only part from which we cannot directly factor the $t'-t$ is of lower order.  
Precisely,
\begin{align*}
\tphi^j(t', \alpha, \xi) - \tphi^j(t, \alpha', \xi ) =  (\alpha - \alpha')\xi + &(t'-t) | \xi| ^{3/2} \\
+ | \xi |^{3/2} &\chi_2^j(|\xi|) ( \chi_1^j(t') \vartheta^j(t', \alpha, \xi) 
- \chi_1^j(t) \vartheta^j(t,  \alpha', \xi)) \\
=  (\alpha - \alpha')\xi(1 &+ \O(2^{-j/2})) + (t'-t) | \xi |^{3/2} (1 + \O(2^{-j/2})).
\end{align*}
We then apply the same arguments as in Lemma \ref{disp-lemma} to complete the proof.
\end{proof}

The parametrix constructed in Proposition \ref{parametrix-prop-2} is not an exact solution for
\eqref{lin-ww-1}, and it only exists for a short time scale depending on the dyadic frequency band.
We use energy estimates for the linear equation \eqref{lin-ww-1} 
to show that the parametrix is sufficient to estimate the actual solution.


\begin{proof}[Proof of Theorem \ref{T:general-str}]
We break $U$ into a piece solving the homogeneous equation
with the frequency-localized initial data 
and a piece solving the inhomogeneous equation with the null data.  
Once we prove the result for the homogeneous equation, 
then the result for the inhomogeneous equation follows immediately from Duhamel's principle. 
Hence, we may assume that $U$ solves the initial value problem for the homogeneous equation
\begin{equation*}
\begin{cases}
P U = 0 \qquad \text{for} \quad t \in [0,2^{-j/2} T], \\
U(0,\alpha)= U_0(\alpha), \qquad \partial_t U(0,\alpha) = U_1(\alpha)
\end{cases}
\end{equation*}
recalling the notation \eqref{D:P}.  We approximate $U$ on the interval $[0, 2^{-j/2}T]$ 
by the oscillatory integral parametrix given in Proposition \ref{parametrix-prop-2}. 
The parametrix $W$ solves
\begin{equation*}
\begin{cases}
P W = E \qquad \text{for} \quad t \in [0,2^{-j/2}T] , \\
W(0,\alpha) = U_0(\alpha),\qquad  \partial_t W(0,\alpha) = U_1(\alpha),
\end{cases}
\end{equation*}
where $E$ satisfies the error estimates \eqref{error-2}.  
We write
\[
\|U\|_{L^p([0, 2^{-j/2}T])L^q_\alpha}\leq
\|W\|_{L^p([0,2^{-j/2}T])L^q_\alpha}+\|U-W\|_{L^p([0, 2^{-j/2}T]) L^q_\alpha}.
\]
Upon applying Theorem \ref{T:general-sc-str}, it is readily seen that 
the first term on the right side satisfies \eqref{E:short-ts-str}.

The function $Z = U - W$ solves
\[
\begin{cases} 
P Z = -E \qquad \text{for}\quad t \in [0,2^{-j/2} T], \\
Z(0,\alpha) = 0, \qquad \partial_t Z(0,\alpha)=0,
\end{cases}
\]
and hence satisfies the energy estimate \eqref{Es:linear_energy1}
\begin{equation}
\label{E:energy-Z-2'}
\| \partial_t Z \|_{L^\infty([0,2^{-j/2}T])H^{s }_\alpha} +
\|Z\|_{L^\infty([0,2^{-j/2}T])H^{s+3/2}_\alpha} \leq C \|
E\|_{L^2([0,2^{-j/2}T])H^{s}_\alpha} 
\end{equation} 
for $-3/2 \leq s \leq k-3/2$.  

By  the Sobolev embedding,  H\"older inequality  in time and  the energy estimate with $s = 1/2 - 1/q$ we then have
\begin{align*}
\| Z \|_{L^p([0,2^{-j/2}T]) L^q_\alpha} 
\leq& C \| Z\|_{L^p([0,2^{-j/2}T]) H^{1/2 - 1/q}_\alpha} \\
\leq& C(T 2^{-j/2})^{1/p} \| Z\|_{L^\infty([0,2^{-j/2}T]) 
H^{1/2 - 1/q}_\alpha} \\
\leq&  C(T 2^{-j/2})^{1/p} \| E \|_{L^2([0,2^{-j/2}T])  H^{-1-1/q}_\alpha} \\
\leq& C T^{1/p} \| E \|_{L^2([0,2^{-j/2}T])  H^{-1 -1/q -1/2p}_\alpha} \\  
\leq &C T^{2/p} \| E \|_{L^\infty([0,2^{-j/2}T])  H^{-1 - 1/q - 1/p}_\alpha}.
\end{align*}
Then using the  estimates on $E$ given in
\eqref{error-2}, the estimate \eqref{E:short-ts-str} for $Z$ follows.
\end{proof}



\begin{proof}[Proof of Corollary \ref{C:homog-str}]
As above, once we prove \eqref{E:main-est-1}, 
then an application of Duhamel's principle and the Minkowski
inequality yields \eqref{E:main-est-2}.
In what follows, thus, we consider only the homogeneous equation.

The first step is to chop the actual solution of \eqref{lin-ww-1} into pieces 
each of which is localized in a dyadic frequency band. 
Let us choose a partition of unity in $\xi$ 
\begin{equation*}
1 = (1-\psi^0)(\xi) + \sum_{j \geq j_0} \psi^j(\xi)
\end{equation*}
as in Section \ref{S:parametrix}.
It is standard from the Littlewood-Paley theory that if $f \in L^q(\reals)$, $q < \infty$, 
then 
\[
\| f \|_{L^q_\alpha} \leq C
\| (1-\psi^0)(D_\alpha) f \|_{L^q_\alpha} +
\left( \sum_{j \geq j_0} \| \psi^j(D_\alpha) f
  \|_{L^q_\alpha}^2 \right)^{1/2},
\]
and hence it suffices to prove \eqref{E:main-est} on dyadic frequency bands.  

Let $U$ be the actual solution of \eqref{lin-ww-1},  and let $U^j = \psi^j(D_\alpha) U$.  
It is readily seen that $U^j$ solves
 \begin{equation} \label{lin-ww-1-j}
 \begin{cases} \partial_t^2U^j - H \partial_\alpha^3U^j +
     2V(t,\alpha) \partial_\alpha\partial_t U^j
      +  V^2(t,\alpha) \partial_\alpha^2 U^j  = R^j(U), \\ 
 U^j(0,\alpha) = U_0^j(\alpha) \quad \text{and} \quad  \partial_t U^j(0,\alpha) = U_1^j(\alpha),
 \end{cases}
 \end{equation}
 where $U_0^j= \psi^j(D_\alpha)U_0$, $U_1^j=\psi^j(D_\alpha) U_1$ and
 \begin{equation}\label{D:Rj}
 \begin{split}
 R^j(U) & = [2V (t,\alpha) \partial_\alpha \partial_t +
 V^2 (t,\alpha)\partial_\alpha^2 \, ,\, \psi^j(D_\alpha) ] U \\
 & = \tpsi^j(D_\alpha) 2^{-j} (\mathcal{A}_{2V}(t,\alpha, D_\alpha) D_t D_\alpha
 + \mathcal{A}_{V^2}(t,\alpha,D_\alpha) D_\alpha^2) U.
 \end{split}
 \end{equation}
Here, $[\cdot , \cdot]$ denotes the commutator, $\mathcal{A}_{2V}$ and $\mathcal{A}_{V^2}$ are 
zeroth-order pseudodifferential operators, and $\tpsi^j$ is a smooth function with support 
contained in a neighbourhood of the support of $\psi^j$.  
It is immediate to see that $D_\alpha$ is comparable to $2^j$ on the support of $\tpsi^j$. 
Consequently, we have the estimate
 \begin{equation}\label{E:Rj}
 \| R^j \|_{L^2_\alpha} \leq C(\| \tpsi^j U \|_{H^1_\alpha} + \| \tpsi^j \partial_t U \|_{L^2_\alpha} ),
 \end{equation}
 where $C>0$ is a constant independent of $j$ depending only on a
 finite number of derivatives of $V$.

Let $U^j = U^j_h + U^j_{i}$ be the solutions to
the homogeneous and inhomogeneous problems, respectively, corresponding to \eqref{lin-ww-1-j}.
That is, $U^j_h$ solves
\[
\begin{cases}
P U^j_h = 0 \qquad \text{for} \quad t \in [0,T], \\
U^j_h(0,\alpha) = U^j_0(\alpha), \qquad \partial_t U^j_h (0,\alpha) = U^j_1(\alpha),
\end{cases}
\]
while $U^j_{i}$ solves
\[
\begin{cases}
P U^j_{i} = R^j(U) \qquad \text{for} \quad t \in [0,T], \\
U^j_{i} (0,\alpha)= 0, \qquad \partial_t U^j_{i} (0,\alpha) = 0.
\end{cases}
\]
We will prove the Strichartz estimates for $U^j_h$, 
which, via Duhamel's principle, will imply the Strichartz estimates for $U^j_{i}$.  
Then summing in $j$ will imply the Strichartz estimates for $U$ solving \eqref{lin-ww-1}.  

We divide the interval $[0,T]$ into $2^{j/2}$ small intervals of the size $2^{-j/2}T$.  
Let $T_{m,j} = 2^{-j/2} (m-1) T$, where $1\leq m\leq 2^{j/2}$
and let $I^{m,j} = [T_{m,j},T_{m+1,j}]$ so that
\[
[0,T] = \bigcup_{1 \leq m \leq 2^{j/2}} I^{m,j}.
\]
Then, we apply the results of Theorem \ref{T:general-str} on each
short time interval $I^{m,j}$, $1 \leq m \leq 2^{j/2}$, to obtain
\begin{align*}
\| U^j_h \|_{L^p([0,T]) L^q(\reals)}^p & = \sum_{m = 1}^{2^{j/2}} \| U^j_h
\|_{L^p(I^{m,j}) L^q(\reals)}^p \\
& \leq C \sum_{m=1}^{2^{j/2}} \left(\| U^j_h (T_{m,j}) \|_{H^{1/2p}(\reals)} +
\|\partial_t U^j_h (T_{m,j}) \|_{H^{1/2p-3/2}(\reals)} \right)^p \\
& \leq C 2^{j/2}  \left(\| U^j_0  \|_{H^{1/2p}(\reals)} +
\|\partial_t U^j_1 \|_{H^{1/2p-3/2}(\reals)} \right)^p.
\end{align*}
The last inequality uses the energy estimate.  

On the other hand, $2^{j/2p} \sim D_\alpha^{1/2p}$ on the frequency support of $U^j_0$ and $U^j_1$,  and a localized version for $j>j_0$ of   \eqref{E:main-est-1} follows. To finish the proof we first notice that for small frequencies, $j\leq j_0$, the estimate follows from Sobolev embedding and the energy estimate \eqref{Es:linear_energy0}.  By Duhamel's principle and Minkowski's integral inequality then   
also \eqref{E:main-est-2} follows. This completes the proof.
\end{proof}

\section{Local well-posedness via energy estimates}\label{S:LWP}

This section concerns the local well-posedness 
of the initial value problem associated to \eqref{E:u}.
The result is of independent interest.

We recall that $R(u, \partial_t u)$ \eqref{E:u}, defined in \eqref{D:R}, involves $\theta$,
which is determined by \eqref{E:transp}. To be precise, hence, the local well-posedness 
is to be established for \eqref{E:u} coupled with \eqref{E:transp}.
Upon the observation that \eqref{E:u} and \eqref{E:transp} are of different type,
we proceed by a ``bootstrapping" argument.
First, given $\theta$ in an appropriate function space, 
the local well-posedness for \eqref{E:u} is established via the energy method,
with $R(u,\partial_t u)$ evaluated with the given $\theta$.
With $u$ so obtained , next \eqref{E:transp} is solved via the standard method of characteristics 
with the variable-coefficient $u$ and $r_1$ evaluated by the solution $u$ in the first step. 
We focus on the development of an energy estimate for \eqref{E:u} and its local well-posedness.

We begin by writing \eqref{E:u} as the first-order in time system 
\begin{equation}\label{E:uv}
\begin{cases}
\partial_tu=v-u\partial_\alpha u, \\
\partial_tv=H\partial_\alpha^3 u-u\partial_\alpha v +\tilde{R}(u, v).
\end{cases}
\end{equation}
In other words, $v=\partial_t u+u\partial_\alpha u$ is the {\em material derivative} of $u$. 
Here, \[
\tilde{R}(u,v)=R(u, v-u\partial_\alpha u)+v \partial_\alpha u +u(\partial_\alpha u)^2\]
satisfies an estimate similar to \eqref{Es:R}:
\begin{align*}
\|\tilde{R}(u,v)\|_{H^s} \leq C(\|u\|_{H^{s+1}}, \| & v\|_{H^s}) \qquad \text{for $s\geq 1$.}
\end{align*}

Let us define the $k$-th energy associated to \eqref{E:uv} as 
\begin{equation}\label{E:energy_s}
\mathcal{E}^k(t)=\frac12 \int^\infty_{-\infty} 
((\partial_\alpha^{k+1}u)H\partial_\alpha (\partial_\alpha^{k+1}u) 
+(\partial_\alpha^k v)^2) \ d\alpha
\end{equation}
and the energy function for \eqref{E:uv} of order $s$ as
\begin{equation}\label{E:energy}
\mathfrak{E}^s(t)=\|u\|_{L^2_\alpha}^2(t)+\|v\|_{L^2_\alpha}^2(t)+\sum_{k=1}^s \mathcal{E}^k(t).
\end{equation}
Note that the operator $H\partial_\alpha$ is a positive operator
with the symbol of its Fourier transform $|\xi|$ and that 
\[
\|f\|_{H^{1/2}_\alpha}^2 = \int^\infty_{-\infty} (f^2 +f H\partial_\alpha f ) \, d\alpha.
\]
In the energy estimates below, we make use of the fact  that 
\begin{equation}\label{E:energy_comm}
\int^\infty_{-\infty} h\partial_\alpha fH\partial_\alpha f \ d\alpha 
=-\frac12 \int^\infty_{-\infty} ([H, h]\partial_\alpha
f)\partial_\alpha f \ d\alpha \leq C\|h\|_{H^{5/2 +}}\|f\|_{L^2}^2.
\end{equation}
Indeed, by integration by parts, 
\begin{align*}
\int^\infty_{-\infty} h\partial_\alpha fH\partial_\alpha f \ d\alpha =& 
-\int^\infty_{-\infty} H(h\partial_\alpha f)\partial_\alpha f \ d\alpha \\
=&-\int^\infty_{-\infty} h(H\partial_\alpha f )\partial_\alpha f \ d\alpha 
-\int^\infty_{-\infty} ([H, h]\partial_\alpha f)\partial_\alpha f \ d\alpha.
\end{align*}
Then, \eqref{E:energy_comm} follows by \eqref{Es:H} and upon the observation that 
$[H,h]\partial_\alpha f=\partial_\alpha([H,h]f) - [H, \partial_\alpha h] f$.

It is readily seen that the energy function $\mathfrak{E}^s(t)$ is equivalent to 
$\|u(t)\|^2_{H^{s+3/2}} + \| \partial_t u(t)+u(t)\partial_\alpha u(t)\|^2_{H^s}$.
Furthermore, $\mathfrak{E}^s(t)$ is equivalent to 
$\|u(t)\|^2_{H^{s+3/2}}+\|\partial_t u(t)\|^2_{H^s}$ if $s >1/2$. 

We now state and prove the nonlinear energy estimate for \eqref{E:uv}.

\begin{proposition}[The nonlinear energy estimates]\label{P:energy_estimate}
If $(u,v) \in H^{s+3/2}(\reals) \times H^s(\reals)$ for $s>1/2$ 
solves \eqref{E:uv} on the interval $t \in [0,T]$ for some $T>0$ 
and if $\|u\|_{H^{5/2 +}_\alpha} <+\infty$ for $0<t<T$
then 
\begin{equation}\label{Es:energy_nonlinear0}
\mathfrak{E}^s(t)<C_1 \qquad \text{for $0<t<T$}
\end{equation}
and subsequently
\begin{equation*}\label{Es:energy_nonlinear}
\|u(t)\|_{H^{s+3/2}}+\|\partial_t u(t)\|_{H^s} < C_2 \qquad \text{for $0<t<T$},
\end{equation*}
where the constants $C_1, C_2>0$ depends on $\|u(0)\|_{H^{s+3/2}}+\|\partial_t u(0)\|_{H^s}$.
\end{proposition}

\begin{remark}[Remark on the energy expression]\rm\label{R:energy}
One may try
\begin{equation}\label{E:bad_energy}
 \int^\infty_{-\infty} ((\partial_\alpha^k \partial_t u)^2 
+ \partial_\alpha^{k+1}u H\partial_\alpha^{k+2}u ) \ d\alpha
\end{equation}
as the $k$-th energy function for \eqref{E:u}. 
Due to the multi-derivative nonlinear term $u^2 \partial_\alpha^2 u$, however, 
the application to \eqref{E:bad_energy} of the standard energy method is unwieldy. 
Indeed, one takes the $t$-derivative of the energy function 
and substitutes $\partial_t^2 u$ by \eqref{E:u}, but to arrive at an expression containing 
\[
\int^\infty_{-\infty} (\partial_\alpha^k \partial_t u)\partial_\alpha^k(u^2 \partial_\alpha^2 u) d\alpha,
\]
which cannot be controlled by the energy function \eqref{E:bad_energy}.

The idea of the proof is to write \eqref{E:u} as a system \eqref{E:uv} to resolve 
the multi-derivative term $u^2 \partial_\alpha^2 u$ into 
two single-derivative terms $u\partial_\alpha u$ and $u\partial_\alpha v$,
which work favorably in the application of the energy method 
by canceling higher Sobolev norms when integrated by parts.

\end{remark}

\begin{proof}
We begin by investigating the time derivative of $\mathcal{E}^k$ by calculating
\begin{equation}\label{E:E_k'}
\begin{split}
\frac{d}{dt}\mathcal{E}^k(t)=&\int_{-\infty}^\infty ((\partial_\alpha^{k+1} \partial_t u)
H\partial_\alpha (\partial_\alpha^{k+1} u) +
(\partial_\alpha^k \partial_t v) (\partial_\alpha^k v)) \ d \alpha \\ :=&\mathcal{E}^k_1+\mathcal{E}^k_2.
\end{split}
\end{equation}

Let us first compute $\mathcal{E}^k_1$. By using the first equation in \eqref{E:uv}
and the integration by parts we may write 
\begin{align}
\mathcal{E}^k_1=& \int_{-\infty}^\infty \partial_\alpha^{k+1}(-u\partial_\alpha u +v)
H\partial_\alpha (\partial_\alpha^{k+1}u) \ d\alpha \notag\\
\begin{split}\label{E:E_k1}
=&-\int_{-\infty}^\infty u (\partial_\alpha^{k+2}u)(H\partial_\alpha^{k+2}u) \ d\alpha\\
&+\int_{-\infty}^\infty (\partial_\alpha^{k+1}v) H\partial_\alpha \partial_\alpha^{k+1}u \ d\alpha
+ \text{(lower order terms)},
\end{split}
\end{align}
where (lower order terms) is a collection of terms which can be bounded in terms of energy
in a routine way.
The second inequality uses \eqref{E:energy_comm}. Note that 
\[ \int_{-\infty}^\infty u (\partial_\alpha^{k+2}u)(H\partial_\alpha^{k+2}u) \ d\alpha
\leq \|u\|_{H^{5/2+}} \|u\|_{H^{k+1}}^2.\]

Similarly, we compute 
\begin{align}
\mathcal{E}^k_2=& \int _{-\infty}^\infty 
\partial_\alpha^k (-u\partial_\alpha v+H\partial_\alpha^3 u +\tilde{R}(u,v))
(\partial_\alpha^k v)\ d\alpha \notag \\
\begin{split}\label{E:E_k2}
=& -\int^\infty_{-\infty} (\partial_\alpha^k (u\partial_\alpha v)-u\partial_\alpha^{k+1}v)
(\partial_\alpha^kv) \ d\alpha
-\int^\infty_{-\infty} (\partial_\alpha^{k+1}v)H\partial_\alpha^{k+2} u \ d\alpha\\
&+ \int^\infty_{-\infty} (\partial_\alpha^k v)\partial_\alpha^k \tilde{R}(u,v)\ d\alpha
+ \text{(lower order terms)}.
\end{split}
\end{align}
Again, (lower order terms) is made up of terms which can be bounded 
in terms of the energy in a routine way.
Note that 
\[ \| \partial_\alpha^k (u\partial_\alpha v)-u\partial_\alpha^{k+1}v\|_{L^2}
\leq C\|u\|_{H^k}\|\partial_\alpha v\|_{H^{k-1}}.\]

The third term on the right side of \eqref{E:E_k1}
and the third term on the right side of \eqref{E:E_k2}
cancel when added together. Other terms are bounded in terms of the energy,
provided that $\|u\|_{H^{5/2 +}}<+\infty$. 

We have proved 
$$\frac{d\mathfrak{E}^s}{dt} \leq C\mathfrak{E}^s(1+\mathfrak{E}^s)^p,$$
for some positive constant $C$ and for some $p>1$.
The proof then is complete by applying Gronwall's inequality. 
\end{proof}

Now, we make a few remarks on the existence, uniqueness and continuous dependence 
for \eqref{E:u}, or equivalently, for \eqref{E:uv}.
In order to establish the existence of solutions, 
we would need to regularize the equation in a certain way. 
The most straightforward way is to introduce mollifiers 
(approximations to the Dirac delta function) into the right sides of \eqref{E:uv}.
We then repeat the argument in the proof of Proposition \ref{P:energy_estimate} 
for the regularized problems to obtain energy estimates of the kind in 
\eqref{Es:energy_nonlinear0} independent of the mollification parameter. 
Subsequently, the Picard theorem for ordinary differential equations on a Banach space
applies to assert that solutions to the mollified equations exist 
on short intervals of time. 
The solutions can be continued on a time interval 
which is independent of the mollification parameter
thanks to the uniform bound from the energy estimate. 

Next, an estimate similar to the energy estimate but in a low norm 
($(u,v) \in H^{s'}(\mathbb{R}) \times H^{s'-3/2}(\mathbb{R})$ with $2\leq s' <3$ )
establishes that as the mollification parameter tends to zero,
the solutions of the mollified equations converge to a solution of the original (non-molified) 
equation \eqref{E:uv}. 
By interpolation, we find that this convergence occurs in the high Sobolev norms, as well.

Uniqueness follows by the energy estimates for the difference of two solutions, and
continuous dependence follows from the time-reversibility of the equations.

The detail of local well-posedness via the argument above is carried out in \cite{Am} 
for the vortex sheet problem with surface tension.

We summarize our result.

\begin{theorem}[The local well-posedness]\label{T:local-wellposedness}
The initial value problem for \eqref{E:u}, prescribed with the initial conditions 
$(u_0, u_1) \in H^s(\mathbb{R}) \times H^{s-3/2}(\mathbb{R})$, $s>2+1/2 $, 
is well-posed on the time interval $[0,T]$ for some $T>0$ and 
$(u(t), \partial_t u(t)) \in C([0,T]; H^s(\mathbb{R}) \times H^{s-3/2}(\mathbb{R}))$.
\end{theorem}

\begin{remark}
The assumption that $s > 2 + 1/2$ is due to the commutator estimate used in the proof of the energy estimates (Proposition \ref{P:energy_estimate}).  Although this commutator estimate can be improved by putting $1/2$ derivatives on each of two copies of $u$ appearing in the energy calculation, we need $2+$ derivatives to couple with the transport equation.  In this sense, the local well-posedness should hold at $s>2$, although we do not prove it in this paper.
\end{remark}

\section{Strichartz estimates for the nonlinear problem}\label{S:strichartz_nonlinear}

At last, we are in a position to prove the Strichartz estimates \eqref{E:main-est}
and \eqref{E:main-est-sc} for solutions to \eqref{E:u}.

Let us consider the initial value problem for \eqref{E:u} 
with the initial conditions 
\[
u(0,\alpha)=u_0(\alpha) \quad \text{and}\quad \partial_t u(0,\alpha)=u_1(\alpha),
\]
where $(u_0, u_1) \in H^s(\mathbb{R}) \times H^{s-3/2}(\mathbb{R})$ for $s>2+1/2$.
Theorem \ref{T:local-wellposedness} applies to ensure that
a unique solution exists on the time interval $[0,T]$ for some $T>0$ and 
$(u(t), \partial_t u(t)) \in C([0,T]; H^s(\mathbb{R}) \times H^{s-3/2}(\mathbb{R}))$.

We apply $\partial^s_\alpha$ to \eqref{E:u} and the above initial conditions to obtain 
\begin{equation}\label{lin-ww-3'}
\begin{cases}
(\partial_t^2 -H\partial_\alpha^3 u +2u \partial_\alpha \partial_t u 
+u^2 \partial_\alpha^2 ) \partial^s_\alpha u =\tilde{R}(u, \partial_t u), \\
\partial^s_\alpha u(0,\alpha)=\partial^s_\alpha u_0(\alpha) \quad \text{and}\quad
 \partial_t \partial^s_\alpha u(0,\alpha)=\partial^s_\alpha u_1(\alpha),
\end{cases}
\end{equation}
where 
\begin{equation}\label{rtilde}
\tilde{R}(u,\partial_t u) = \partial^s_\alpha R + 
[\partial^s_\alpha, u^2 \partial_\alpha^2 + 2 u \partial_\alpha \partial_t]u.
\end{equation}
We view the initial value problem \eqref{lin-ww-3'} as a linear problem for $\partial^s_\alpha u$ 
of the form in \eqref{lin-ww-1} and \eqref{lin-ww-0-nh}, 
where the solution $u$ is the coefficient function $V(t,\alpha)$
and $\tilde{R}(u,\partial_t u)$ is the inhomogeneous term $R(t,\alpha)$. 
This can be done thanks to uniqueness of these initial value problems.

Let us take $s \geq 1$ sufficiently large so that 
$u \in H^{l}([0,T]) H^{k}(\reals)$, where $l,k \geq 1$ are large enough for 
the results in Section \ref{S:parametrix} and Section \ref{SS:energy} and 
so that $(\partial^s_\alpha u_0, \partial^s_\alpha u_1) \in L^2(\mathbb{R}) \times H^{-3/2}(\mathbb{R})$ 
and $\tilde{R} \in L^2([0,T]) L^2_\alpha(\reals)$.
Therefore, Theorem \ref{T:general-str} applies to yield 
the semiclassical Strichartz estimates \eqref{E:main-est-sc}, and Corollary \ref{C:homog-str} 
apply to yield the fixed time Strichartz estimates \eqref{E:main-est}.
This completes the proof of Theorem \ref{Main1} and Theorem \ref{Main2}.

\begin{appendix}

\section{Local smoothing effect for the nonlinear problem}\label{S:local_smoothing}

Remarked here is how the method of positive commutator yields 
the local smoothing effect for \eqref{E:u} of the gain of a $1/4$
derivative's smoothness.  The method used here was suggested to us by
T. Alazard, N. Burq, and C. Zuily.

\begin{theorem}[The local smoothing effect]
\label{T:smoothing-comm-arg}
If $s>2+1/2$ is sufficiently large, then the solution $u\in C([0,T];H^s(\reals))$ 
of the initial value problem for \eqref{E:u} 
with the initial condition $(u_0, u_1) \in H^s(\reals) \times H^{s-3/2}(\reals)$ satisfies the estimate
\begin{equation}\label{E:local-smoothing}
\| \lll \alpha \rrr^{-\rho} D_\alpha^{s+1/4} u \|_{L^2([0,T])L^2_\alpha} 
\leq C(\| u_0 \|_{H^s}, \| u_1 \|_{H^{s-3/2}} ),
\end{equation}
where $\rho>1/2$. Here, $\lll \alpha \rrr=(1+\alpha^2)^{1/2}$ describes weighted Sobolev spaces
and $D_\alpha=-i\partial_\alpha$.
\end{theorem}

\begin{proof}
We recall that the energy method yields the local well-posedness 
of the initial value problem for \eqref{E:u}
and $u \in C([0,T];H^{s}_\alpha(\reals))$ for some $T>0$. 

Once we settle the issue of existence and uniqueness,
we view \eqref{E:u} as 
\[ 
\partial_t^2u-H\partial_\alpha^3u+2V(t,\alpha)\partial_t \partial_\alpha u
+V^2(t,\alpha)\partial_\alpha^2u=R(t,\alpha),\]
where the variable coefficient $V(t,\alpha)$ replaces the solution $u$ 
and $R(t,\alpha)$ replaces the remainder $R(u,\partial_t u)$,
as in Section \ref{S:strichartz_nonlinear}.
We choose $s>1$ large so that $V(t,\alpha)$ and $R(t,\alpha)$ have the required regularity. 

Let $\psi \in \mathcal{C}^\infty (\mathbb{R})$ be such that 
$\psi(\xi)\equiv1$ for $|\xi| \geq 1$ and $\psi(\xi)\equiv 0$ for $|\xi|\leq 1/2$, and let  
\[u=w+v,\] where $w=\psi(D_\alpha)u$ is the high frequency part of the solution 
and $v$ is the low frequency part of the solution. 
It is straightforward that $v$ belongs to every Sobolev space, 
and thus $v$ is as smooth as we want. 
It suffices to show \eqref{E:local-smoothing} for $w$. 

Further, let  $\psi(D_\alpha)=\psi^+(D_\alpha) + \psi^-(D_\alpha),$
where $ \psi^+=\psi|_{\xi\geq 0}$ and $\psi^-=\psi|_{\xi\leq 0}$, and let 
$w^\pm=\psi^\pm(D_\alpha)u$. 
We will present the argument only for positive high frequencies $w^+$. 
For simplicity of notation, in what follows, we write $w$ for $w^+$.

For positive high frequencies $\xi>0$, 
the Hilbert transform $H$ behaves like the multiplication by $-i$, and thus
\begin{equation}\label{E:w-tildeR}
Pw=(\partial_t^2+i\partial_\alpha^3+2V(t,\alpha)\partial_t\partial_\alpha
+V^2(t,\alpha)\partial_\alpha^2)w=\tilde{R}(t,\alpha),
\end{equation}
where $$\tilde{R}=\psi^+(D_\alpha)R(t,\alpha)+[2V(t,\alpha)\partial_t\partial_\alpha
+V^2(t,\alpha)\partial_\alpha^2, \psi^+(D_\alpha)]u$$ 
satisfies the estimate
\begin{equation}\label{Es:tildeR} 
\|\tilde{R}\|_{L^2_\alpha(\reals)} \leq \| \psi^+(D_\alpha) R \|_{L^2(\reals)}
+ C(\|w\|_{H^1_\alpha(\reals)} +\|\partial_t w\|_{L^2_\alpha(\reals)}),
\end{equation}
where $C>0$ depends on $V$ and its derivatives.

Let
\[A(\alpha)=\int_{-\infty}^\alpha \lll \beta \rrr^{-2\rho} \, d\beta\]
for $\rho>1/2$ as in the statement of the theorem. 
Let $\lll \cdot, \cdot \rrr_{L^2}$ denote the (Hermitian) $L^2_\alpha$-inner product. We compute 
\begin{align*}
2i \,\text{Im} \int_0^T \lll A\tilde{R}, w \rrr_{L^2} dt =& 
\int_0^T \lll A\tilde{R}, w \rrr_{L^2} dt -\int_0^T \lll w, A\tilde{R} \rrr_{L^2} dt  \\
= & \int_0^T \lll APw, w\rrr_{L^2} dt - \int_0^T \lll Aw, Pw \rrr_{L^2} dt \\
=& \int_0^T \lll [A,P] w, w \rrr_{L^2} dt + \int_0^T \lll (P - P^*)Aw, w\rrr_{L^2} dt \\
&+ \Big[ \lll A \partial_t w, w \rrr_{L^2} \Big]_0^T - \Big[ \lll A w, \partial_t w \rrr_{L^2} \Big]_0^T .
\end{align*}
Here, $P^*$ denotes the adjoint of $P$. 
The second equality uses that $A$ is self-adjoint, 
and the third inequality uses integrations by parts in $t$. 
The sum of the boundary terms is $2i[\text{Im} \lll A \partial_t w, w \rrr_{L^2}]^T_0$.  

It is important to note that although the solution to the nonlinear equation \eqref{E:u} 
may be assumed real-valued, the solution $w$ of \eqref{E:w-tildeR} is not, 
since we have cut it off to positive frequencies.  
Indeed, if we add back in the negative frequencies so that $w$ is once more real-valued 
and if we try to construct an appropriate commutant, then 
we are inevitably led to use the Hilbert transform as part of the commutant,
and the resulting boundary terms do not cancel.

Next, it is straightforward that 
\begin{multline*}
[P,A]w=3iA_\alpha \partial_\alpha^2w+ 3iA_{\alpha\alpha}\partial_\alpha w
+iA_{\alpha\alpha\alpha}w \\ +2V(t,\alpha)A_\alpha \partial_t w
+2V^2(t,\alpha)A_\alpha \partial_\alpha w+V^2(t,\alpha) A_{\alpha\alpha}w.
\end{multline*}

Now, we compute 
${\displaystyle 
\int^T_0 \lll Aw, Pw\rrr_{L^2} dt=\int^T_0 \lll P^*Aw, w\rrr_{L^2} dt.}$
Integrations by parts in $t$ and in $\alpha$ yield
\begin{align*}
\int^T_0 \lll Aw, \partial_t^2w\rrr_{L^2} dt =& 
\Big[\lll Aw, \partial_t w\rrr_{L^2} \Big]^T_0 -\Big[\lll A\partial_t w,w\rrr_{L^2}\Big]^T_0 
+\int^T_0 \lll A\partial_t^2w, w\rrr_{L^2} dt \\
=& -2i \Big[\text{Im} \lll A \partial_t w, w \rrr_{L^2}\Big]^T_0+ \int^T_0 \lll A\partial_t^2w, w\rrr_{L^2} dt,\\ 
\int^T_0 \lll Aw, i\partial_\alpha^3 w\rrr_{L^2} dt =&\int^T_0 \lll i\partial_\alpha(Aw),w\rrr_{L^2} dt.
\end{align*}
Similarly,
\begin{align*}
\int^T_0 \lll Aw, 2V\partial_t \partial_\alpha w\rrr_{L^2} dt =&
\Big[\lll 2VAw, \partial_\alpha w \rrr_{L^2} \Big]^T_0 
-\int^T_0 \lll 2V_t A w, \partial_\alpha w\rrr_{L^2} dt \\
&+\int^T_0 \Big( \lll 2V_\alpha A\partial_t w, w\rrr_{L^2} dt 
+ \lll 2V\partial_t\partial_\alpha (Aw),w\rrr_{L^2}\Big) dt,\\
\int^T_0 \lll Aw, V^2 \partial_\alpha^2 w\rrr_{L^2} dt =&
\int^T_0 \Big( \lll (V^2)_{\alpha\alpha}Aw, w\rrr_{L^2} \\ 
&\qquad + \lll 2(V^2)_\alpha \partial_\alpha (Aw),w\rrr_{L^2}
+\lll V^2\partial_\alpha^2 (Aw),w\rrr_{L^2} \Big) dt.
\end{align*}
Therefore, 
\begin{align*}
3i \int^T_0 \lll A_\alpha \partial_\alpha w, \partial_\alpha w\rrr_{L^2} dt=&
2i\int^T_0 \text{Im}\lll A\tilde{R},w\rrr_{L^2} dt +4i\Big[\text{Im} \lll A \partial_t w, w \rrr_{L^2}\Big]^T_0 \\
& + i\int^T_0 \Big( \lll A_{\alpha\alpha\alpha}w,w\rrr_{L^2}
 +  \lll 2VA_\alpha \partial_t w,w\rrr_{L^2} \Big) dt \\
&+\int^T_0 \Big( \lll 2V^2A_\alpha \partial_\alpha w,w\rrr_{L^2} 
+ \lll V^2A_{\alpha\alpha} w,w\rrr_{L^2} \Big) dt \\
&+\Big[\lll 2VAw, \partial_\alpha w\rrr \Big]^T_0-\int^T_0\lll 2V_tAw,\partial_\alpha w\rrr_{L^2} dt \\
&+\int^T_0 \Big( \lll 2V_\alpha \partial_t(Aw),w\rrr_{L^2} \\
&\qquad  +\lll (V^2)_{\alpha\alpha}Aw, w\rrr_{L^2} 
+\lll 2(V^2)_\alpha \partial_\alpha (Aw),w\rrr_{L^2} \Big) dt.
\end{align*}

Since $A$ and its derivatives are in $L^\infty(\reals)$, 
using the energy estimates, Sobolev embeddings, and the estimates for $\tilde{R}$, 
it follows that  
\begin{align*}
\|A_\alpha \partial_\alpha w\|^2_{L^2_TL^2_\alpha} \leq 
C(& \| \tilde{R} \|_{L^2_TH^{-3/4}}^2 + \| w \|_{L^2_T H^{3/4}}^2   + \| w \|_{L^\infty_T H^{3/4} }^2 + \| \partial_t w \|_{L^\infty_T
  H^{-3/4}}^2 \\
& + \|V\|_{L^\infty_T H^s_\alpha}^2
\|w\|_{L^2_TH^{3/4}_\alpha}^2 + \|\partial_t
w\|_{L^2_TH^{-3/4}_\alpha}^2 \\
& + (1 + \|V^2\|_{L^\infty_T H^s_\alpha}^2)
\|w\|^2_{L^2_TH^{1/2}_\alpha} + (1 + \|V\|_{L^\infty_T H^s_\alpha}^2)
\|w\|^2_{L^\infty_T H^{1/2}_\alpha} \\
& + (1 + \|\partial_t V\|_{L^\infty_T H^s_\alpha}^2)
\|w\|^2_{L^2_T H^{1/2}_\alpha} \\
&  + \|\partial_\alpha V\|_{L^\infty_T H^s_\alpha}^2
\|w\|_{L^2_TH^{3/4}_\alpha}^2 + \|\partial_t
w\|_{L^2_TH^{-3/4}_\alpha}^2 \\
& + (1 + \| \partial_\alpha^2 (V^2)\|_{L^\infty_T L^\infty_\alpha}^2  )
\| w \|_{L^2_T L^2_\alpha}^2 \\
& + (1 + \|\partial_\alpha (V^2)\|_{L^\infty_T H^s_\alpha}^2)
\|w\|^2_{L^2_T H^{1/2}_\alpha} )\\
\leq  C(&T,  \|V\|_{H^{3/2}_T H^{s'}_\alpha}) 
(\|w\|_{L^\infty_TH^{3/4}_\alpha}+\|\partial_t w\|_{L^\infty_TH^{-3/4}_\alpha})^2
\end{align*}
for some $s'>0$.  This completes the proof.
\end{proof}

\section{Assorted proofs: formulation}\label{A:proofs}

We collect the proofs of \eqref{Es:gamma_t}, Corollary \ref{C:(K[z]f)_t}, 
Lemma \ref{L:theta_estimate}, and Lemma \ref{L:r_t}.
  
In order to estimate $\|\gamma_t\|_{H^s}$ in terms of $u$ and $\theta$, we recall that
\[
\partial_t \gamma=S\partial_\alpha^2 \theta
+\partial_\alpha((\V-\mathbf{W}\cdot \hat{\mathbf{t}})\gamma)
-2\mathbf{W}_t \cdot \hat{\mathbf{t}}-\frac{1}{2}\gamma\partial_\alpha\gamma
+2(\V-\mathbf{W}\cdot \hat{\mathbf{t}}) \mathbf{W}_\alpha \cdot \hat{\mathbf{t}}.
\]

We expand $\mathbf{W}_t \cdot \hat{\mathbf{t}}$ as
\begin{align*}
\mathbf{W}_t\cdot \hat{\mathbf{t}}=&\text{Re}(z_\alpha\overline{\Phi}(\mathbf{W}_t))\\
=& \text{Re}\left( \frac{1}{2\pi i} z_\alpha(\alpha)\text{PV}
\int \frac{\gamma_t(\alpha')}{z(\alpha)-z(\alpha')}d\alpha' \right)\\
&+\text{Re} \left( \frac{1}{2\pi i} z_\alpha(\alpha)\text{PV}
\int \gamma(\alpha') \frac{z_t(\alpha)-z_t(\alpha')}{(z(\alpha)-z(\alpha'))^2}d\alpha'\right) 
:=\mathcal{J}[z]\gamma_t +R_5,
\end{align*}
where
$\mathcal{J}[z]f=\text{Re} \left(2iz_\alpha H\left(\frac{f}{z_\alpha}\right)
+z_\alpha(\alpha)\mathcal{K}[z]f\right)$.
Accordingly, the above equation for $\gamma_t$ takes the form as
\begin{equation}\label{E:gamma_t}
(1+2\mathcal{J}[z])\gamma_t=S\partial_\alpha^2\theta
+\partial_\alpha((\V-\mathbf{W}\cdot \hat{\mathbf{t}})\gamma ) 
-\frac12 \gamma \partial_\alpha \gamma+2(\V-\mathbf{W}\cdot \hat{\mathbf{t}})
\mathbf{W}_\alpha \cdot \hat{\mathbf{t}}-2R_5.
\end{equation}

It is proved in \cite[Lemma 6.1]{Am} that $(1+2\mathcal{J}[z])^{-1}:L^2 \to L^2$ is bounded. 
One observes that $R_5$ may be written in such a way that it is the sum of terms 
which differ from $R_3$ and $R_4$ by multiplication by $i$, 
and therefore, they are estimated mutandis mutandi to yield that
\[
\|R_5\|_{H^s} \leq C(\|\theta\|_{H^{s+1}})(1+\|u\|_{H^{s+1}})^2
\]
for $s\geq 1$. 
The argument in the proof of \cite[Lemma A.4]{Am} applies to assert \eqref{Es:gamma_t}

\begin{proof}[Proof of Corollary \ref{C:(K[z]f)_t}]
We write 
$$\partial_t(\mathcal{K}[z]f)=
\mathcal{K}[z](\partial_t f)+\frac{1}{2i}H\left(\frac{f}{z_\alpha^2}z_{\alpha t}\right)
-\frac{1}{2\pi i} \int f(\alpha') \frac{z_t(\alpha)-z_t(\alpha')}{(z(\alpha)-z(\alpha'))^2} d\alpha'.$$
Here, the last term is related to $R_3$ and $R_4$ 
in the proof of Lemma \ref{L:W_t}, and thus it is estimated in a similar way.

By the usual product rule, \eqref{Es:H_t} follows.
\end{proof}

\begin{proof}[Proof of Lemma \ref{L:theta_estimate}]
For $s=0,1$, we use the transport equation \eqref{E:transp}.
By multiplication by $\theta$ to \eqref{E:transp} and integration by parts yield
$$ \frac{d}{dt} \int \theta^2 d\alpha = \frac{1}{2} \int \theta ^2 \partial_\alpha u \ d\alpha 
+\int \theta H \partial_\alpha u \ d\alpha + \int \theta r_1(t,\alpha) \ d\alpha.$$
We obtain the analogous identity for $\int (\partial_\alpha \theta )^2$, 
and by adding,
$$\frac{d}{dt} \|\theta\|_{H^1} \leq \|\partial_\alpha u\|_{L^\infty} \|\theta\|_{H^1}
+2(\|u\|_{H^2}+\|r_1\|_{H^1}).$$
Gronwall's inequality then applies to give that 
\begin{align*}
\|\theta(t)\|_{H^1} \leq & \|\theta_0\|_{H^1} +\int^t_0 (\|u \|_{H^2} +\|r_1\|_{H^1} )dt'  \\
&+C\int^t_0 \|\partial_\alpha u\|_{L^\infty}\left( \|\theta_0\|_{H^1} +
\int^{t'}_0 (\|u \|_{H^2} +\|r_1\|_{H^1})\right) \exp \left( \int^t_{t'} \|\partial_\alpha u\|_{L^\infty}\right) dt' \\
\leq & C(\|u\|_{H^2})(1+\|u\|_{H^2} +\|r_1\|_{H^1}).
\end{align*}
Indeed, $\|r_1\|_{H^1} \leq C(\|\theta\|_{H^2})(1+\|u\|_{H^1})$.

Next, for $s=2$ by multiplying (\ref{E:system}b) by $\partial_\alpha^2 \theta$ 
and by integrating it it follows that 
$$\|\partial_\alpha^2 \theta \|_{L^2}^2 \leq \|\partial_t u\|_{L^2} +\|u\|_{H^1}^2 
+\|\partial_\alpha \theta\|_{L^\infty}^2 \|\partial_t u\|_{L^2} 
+C(\|\theta\|_{H^2})(1+\|u\|_{H^1})^2.$$
Together with the $\|\theta\|_{H^1}$ estimate above, this proves \eqref{Es:theta} for $s=0$.
For $s > 2$, we take derivative of (\ref{E:system}b) and repeat the argument. 
This proves the assertion.
\end{proof}

\begin{proof}[Proof of Lemma \ref{L:r_t}]
First, it is straightforward to see that 
\[ 
\partial_t r_1=-H(\mathbf{m}_t \cdot \widehat{\mathbf{t}})-H(\mathbf{m} \cdot \widehat{\mathbf{n}}) \theta_t
+\mathbf{m}_t \cdot \widehat{\mathbf{n}} + (\mathbf{m} \cdot \widehat{\mathbf{t}}) \theta_t,
\]
where
\begin{align*}
\overline{\Phi}(\mathbf{m}_t)=& z_{\alpha t} \mathcal{K}[z]\left( \frac{\gamma_\alpha}{z_\alpha}-
\frac{\gamma z_{\alpha\alpha}}{z_\alpha^2}\right)
+z_\alpha \partial_t \left( \mathcal{K}[z]\left(\frac{\gamma_\alpha}{z_\alpha}-
\frac{\gamma z_{\alpha\alpha}}{z_\alpha^2}\right)\right) \\
&+\frac{z_{\alpha t}}{2i} \left[H, \frac{1}{z_\alpha^2}\right]\left(\gamma_\alpha-
\frac{\gamma z_{\alpha\alpha} }{z_\alpha} \right) 
+\frac{z_{\alpha}}{2i} \partial_t \left[H, \frac{1}{z_\alpha^2}\right]\left(\gamma_\alpha-
\frac{\gamma z_{\alpha\alpha}}{z_\alpha} \right).
\end{align*}
Then, \eqref{Es:K_t} and \eqref{Es:H_t} apply to assert that 
\[
\| \mathbf{m}_t\|_{H^s} \leq C(\|\partial_t u\|_{H^1}, \|u\|_{H^{s+1}}, \|\partial_t u\|_{H^{s-1}}),
\]
and, in turn, it follows \eqref{Es:r_1t}.
The difference is estimated in the usual way.

Next is the estimate for $\partial_t r_2$. 
We recall from the proof of Lemma \ref{L:W_t} that 
\[
\partial_t r_2=\frac12 H\partial_t^2 \gamma +\partial_t (R_1+R_2+R_3+R_4).
\]
In order to estimate for $\partial_t^2 \gamma$, we take the derivative with respect to $t$-variable 
of \eqref{E:gamma_t} to obtain 
\begin{align*}
(id+J[z])\gamma_{tt} =& -\text{Re} \Big( 2i z_{\alpha t} H\left( \frac{\gamma_t}{z_\alpha}\right) 
-2iz_\alpha H\left(\frac{\gamma_t}{z_\alpha^2}z_{\alpha t}\right)  \\
&+z_{\alpha t}\mathcal{K}[z]\gamma_t 
+\frac{z_\alpha}{2i} H\left(\frac{\gamma_t}{z_\alpha^2}z_{\alpha t}\right) 
-\frac{z_\alpha}{2\pi i} \int \gamma_t(\alpha') 
\frac{z_t(\alpha)-z_t(\alpha')}{(z(\alpha)-z(\alpha'))^2}d\alpha'\Big)\\
&+\theta_{\alpha \alpha t} +\partial_t \partial_\alpha(\gamma (\V-\mathbf{W}\cdot \widehat{\mathbf{t}}))
-2\partial_t \left(\frac14 \gamma \gamma_\alpha-(\V-\mathbf{W}\cdot \widehat{\mathbf{t}})
\mathbf{W}_\alpha\cdot \widehat{\mathbf{t}}\right).
\end{align*}
Each term on the right side of the equation is estimated 
by using various estimates we established previously, and then
we conclude that
\[
\|\partial_t^2 \gamma\|_{H^s} \leq C(\|u\|_{H^{s+2}}, \|\partial_t u\|_{H^{s+1}}).
\]
Again, using the estimates established previously, we obtain 
\[
\|\partial_t (R_1+R_2+R_3+R_4)\|_{H^s} \leq 
C(\|\partial_t u\|_{H^1}, \|u\|_{H^{s+1}}, \|\partial_t u\|_{H^{s-1}}).
\]
The differences of $\partial_t^2 \gamma $ 
and $\partial_t R_j$'s, $j=1,2,3,4$, are obtained in the usual way.
Therefore follows \eqref{Es:r_2t}.

That is, without the cancellation of the highest-order derivative term 
$\partial_\alpha \theta \partial_\alpha^2 u$ in $r_2$, 
the remainder $R(u, \partial_t u)$ is of second-order in $\alpha$. 
This completes the proof.
\end{proof}

\section{Assorted proofs: parametrix construction}\label{A:parametrix}

We first derive \eqref{dxh-eqn}. Let us first write $iHD_\alpha^3w$ via the Fourier transform as  
\begin{multline*}
iHD_\alpha^3 w(t,\alpha) = - H\partial_\alpha^3 w(t,\alpha) \\ 
= \frac{1}{4 \pi } \iint e^{i(\alpha-\alpha') \xi} | \xi |^3
\iint e^{-i\beta \xi' } e^{i \phi(t, \alpha', \xi')} f(\beta) \, d\beta d\xi' \, d\alpha' d \xi.
\end{multline*}
In what follows, we recall that we implicitly assume 
that both $\xi$ and $\phi_\alpha$ are large and positive and that $\xi \sim 2^j$.  Our goal is to
eliminate the dependence on $\xi$ in the above integral so that the
integration in $\alpha'$ and $\xi$ yields $\delta(\alpha - \alpha')$
(so that the above representation reduces to an integral in $\beta$ and $\xi'$ only).  
To this end, we write 
\[
\phi(t, \alpha, \xi') = \phi(t, \alpha', \xi') + \varPhi(\alpha,\alpha')(\alpha-\alpha')
\]
and we perform a change of variables to obtain
\begin{multline*}
iHD_\alpha^3 w(t,\alpha)  = \frac{1}{4 \pi } \iint e^{i(\alpha-\alpha') \eta} | \eta +
\varPhi(\alpha,\alpha')|^3 \\ \cdot \iint e^{-i\beta \xi' } e^{i \phi(t, \alpha, \xi')}  f(\beta)  
d\beta d\xi'  \, d\alpha' d \eta.
\end{multline*}
We further write $\varPhi(\alpha,\alpha') =  \phi_\alpha (t, \alpha', \xi') + \varPhi_1(t,\alpha,\alpha', \xi')$, where
\begin{multline*}
\varPhi_1(t,\alpha,\alpha',\xi') =  \frac{1}{2} \phi_{\alpha\alpha}(t,\alpha',\xi')(\alpha-\alpha')  \\ +
\frac{1}{6}  \phi_{\alpha\alpha\alpha}(t,\alpha',\xi')(\alpha-\alpha')^2 +
\tilde{\varPhi}(t,\alpha,\alpha',\xi')(\alpha-\alpha')^3
\end{multline*}
for some $\tilde{\varPhi}=\mathcal{O}(\sup_\alpha |\partial_\alpha^4 \phi|)$, 
and accordingly, the above integral becomes
\begin{multline*}
iHD_\alpha^3 w (t,\alpha)  \\
 = \frac{1}{4 \pi}\iint e^{i(\alpha-\alpha') \eta} \left(| \eta +
\phi_\alpha|^3 + 3   | \eta +\phi_\alpha|^2 \varPhi_1
 + 3 | \eta + \phi_\alpha| \varPhi_1^2 + \varPhi_1^3\right) \\
  \cdot \iint e^{-i\beta \xi' } e^{i \phi(t, \alpha, \xi')}  f(\beta) \, d\beta d\xi' \, d\alpha' d \eta.
\end{multline*}
We keep in mind that $\phi_\alpha$ in the above expression is evaluated at $(t,\alpha', \xi')$.
Now, $\varPhi_1$ is a sum of terms multiplied with powers of $\alpha-\alpha'$, which
upon integrations by parts in $\eta$ are cancelled and the above integral, in turn, becomes
\begin{equation}\label{int-form-1}
\begin{split}
iHD_\alpha^3 w(t,\alpha)  = \frac{1}{4 \pi}\iint e^{i(\alpha-\alpha') \eta} 
\Big(& |  \eta + \phi_\alpha|^3 + 3i| \eta + \phi_\alpha| \phi_{\alpha\alpha} 
- \phi_{\alpha\alpha\alpha}   \Big)  \\ &  \cdot
\iint e^{-i\beta \xi' } e^{i \phi(t, \alpha, \xi')}  f(\beta) \, d\beta d\xi' d\alpha' d \eta.
\end{split}
\end{equation}
Under the assumption that either both $\eta$ and $\phi_\alpha$ are
large and positive or both are large and negative the above formal argument is justified.
Indeed, the dyadic frequency localization assumption \eqref{dy-loc} on $f$ 
implies that $w^+$ is also localized to dyadic frequencies (possibly with different constants), 
and hence the singularity of $| \eta|$ at $\eta = 0$ does not enter into the above calculation.  

We now expand $| \eta + \phi_\alpha |^3$ for both $\eta$ and
$\phi_\alpha$ large and positive (see Lemma \ref{freq-lemma} for a justification
of this)
\[
| \eta + \phi_\alpha |^3 = | \eta|^3 + 3 |\phi_\alpha|| \eta|^2 + 3
|\phi_\alpha|^2 |\eta| + |\phi_\alpha|^3.
\]
Again, we keep in mind that $\phi_\alpha$ is evaluated at $(t, \alpha', \xi')$.  
Substituting this in \eqref{int-form-1} and integrations by parts in $\alpha'$ then yield that
\begin{multline*}
iH D_\alpha^3 w(t,\alpha)  =  \frac{1}{4\pi}\iint e^{i(\alpha-\alpha') \eta} \\
\cdot \Big( | \eta|^3 +3 | \eta|^2 | \phi_\alpha| + 3 | \eta| |\phi_\alpha|^2 
+ |\phi_\alpha|^3    +3i(| \eta | + |\phi_\alpha|) \phi_{\alpha\alpha} -
\phi_{\alpha\alpha\alpha}   \Big) \\
 \cdot  \iint e^{-i\beta \xi' } e^{i \phi(t, \alpha, \xi')}  f(\beta)  d\beta d\xi' d\alpha' d \eta. 
  \end{multline*}
Finally, \eqref{dxh-eqn} follows upon integrations in $\alpha'$ and $\eta$.

\

Next, we show that the mapping in \eqref{y-x} is invertible 
for each $\zeta \in [-\epsilon,\epsilon]$ and $0\leq t\leq 2^{j/2}T$. 
It suffices to show that
\[
\left| \frac{ \partial \alpha}{\partial \beta} \right| \geq C^{-1}>0\qquad 
\text{for $0\leq  t  \leq 2^{-j/2}T$}.
\]

First, if $\eta(t)$ is a solution of (\ref{Ham-system-1}a) with the initial condition $\eta(0)=\zeta$ then 
\begin{align*}
[ (1+2^{-j/2}  \xi^{1/2}\eta)^2]^{\cdot}
 = 2(1+ 2^{-j/2} \xi^{1/2}\eta) 2^{-j/2} \xi^{1/2}\dot{\eta} 
 = 2 V_\alpha(t,2^{j/2}\alpha)  (1+2^{-j/2} \xi^{1/2}\eta)^2,
\end{align*} 
whence 
\[
 -C(1+2^{-j/2} \xi^{1/2}\eta)^2 \leq [(1+2^{-j/2} \xi^{1/2}\eta)^2]^\cdot \leq C(1+2^{-j/2} \xi^{1/2}\eta)^2
 \] 
for some $C \geq \|V_\alpha\|_{L^\infty_TL^\infty_\alpha}$. 
We recall that the dot represents differentiation in the $t$-variable.
By Gronwall's inequality it then follows that
\[
(1+2^{-j/2} \xi^{1/2}\zeta)^2 \exp(-Ct) \leq (1+2^{-j/2} \xi^{1/2}\eta)^2 \leq (1+2^{-j/2} \xi^{1/2}\zeta)^2 \exp(Ct).
\]
That is, $(1+2^{-j/2} \xi^{1/2}\eta(t))^2=(1+2^{-j/2} \xi^{1/2}\zeta)^2(1+\mathcal{O}(t))$. 

Next, we calculate
\begin{align*}
\Big[ \Big( & \frac{\partial (1+2^{-j/2} \xi^{1/2}\eta)}{\partial  \beta}  \Big)^2\Big]^\cdot \\ &  = 
2 \cdot 2^{-j/2} \xi^{1/2} \frac{\partial (1+2^{-j/2} \xi^{1/2}\eta)}{\partial \beta} 
 \frac{\partial \dot{\eta}}{\partial \beta} \\
&=  2  \frac{\partial (1+2^{-j/2} \xi^{1/2}\eta)}{\partial \beta}  \Bigg( V_\alpha(t,2^{j/2}\alpha) 
  \frac{\partial (1+2^{-j/2} \xi^{1/2}\eta)}{\partial \beta} \\
 &\qquad \qquad \qquad \qquad \qquad \qquad 
 + 2^{j/2}V_{\alpha\alpha}(t,2^{j/2}\alpha) (1+ 2^{-j/2} \xi^{1/2}\eta)
  \frac{\partial \alpha}{\partial \beta} \Bigg) \\
 & \leq  \left(2 |V_\alpha| + 2^{j/2}|V_{\alpha\alpha}(1+2^{-j/2} \xi^{1/2}\eta)|^2\right) 
  \left( \frac{\partial (1+2^{-j/2} \xi^{1/2}\eta)}{\partial \beta}\right)^2 
  +   2^{j/2}\left( \frac{\partial  \alpha}{\partial \beta} \right)^2.
\end{align*}
By Gronwall's inequality it follows that
\be
\left( \frac{\partial (1+2^{-j/2} \xi^{1/2}\eta)}{\partial \beta}  \right)^2(t) \leq  
\exp (Ct2^{j/2}\|V\|_{L^\infty_{T} W^{2,\infty}_\alpha }^2) 
\left(C'+ 2^{j/2}\int_0^t  \left( \frac{\partial  \alpha}{\partial \beta} \right)^2  \right)
\ee
for some $C, C'>0$.

Finally, we calculate
\begin{align*}
\Big[\Big( & \frac{ \partial \alpha }{\partial \beta} \Big)^2 \Big]^2  =  2
\frac{ \partial \alpha }{\partial \beta} \frac{ \partial \dot{\alpha} }{\partial \beta} \\
&=  2  \frac{ \partial \alpha }{\partial \beta} 
\Bigg( \frac{3}{4}2^{-j/2} \xi^{1/2} (1+2^{-j/2} \xi^{1/2}\eta)^{-1/2} \frac{ \partial (1+2^{-j/2} \xi^{1/2}\eta) }{\partial \beta} 
- 2^{j/2} \xi^{-1/2} V_{\alpha}(t,2^{j/2}\alpha)  \frac{ \partial \alpha }{\partial \beta} \Bigg)\\
& \geq  -(2^{j/2}\xi^{-1/2}|V_\alpha|+1) \left( \frac{ \partial \alpha }{\partial \beta}\right)^2 \ - C (1+2^{-j/2} \xi^{1/2}\eta)^{-1}\left(\frac{\partial (1+2^{-j/2}
     \xi^{1/2}\eta)}{\partial \beta}\right)^2 \\
& \geq  -C\left( \frac{ \partial \alpha }{\partial \beta}\right)^2 -
C  \sup_{0 \leq t \leq 2^{-j/2}T} 2^{-j/2}\left( \frac{ \partial \alpha }{\partial \beta}\right)^2 .
\end{align*}
The claim then follows by Gronwall's inequality once 
$\frac{\partial \alpha }{\partial \beta} \Big|_{t=0} = 1$ and $0 \leq t \leq  2^{-j/2}T$ are observed. 
Consequently, the inverse function theorem applies to give that 
the mapping \eqref{y-x} is invertible for $0\leq t\leq 2^{-j/2}T$.

\section{Energy estimates for the linearized equation}\label{A:energy}

This appendix concerns the energy estimates \eqref{Es:linear_energy0}
and \eqref{Es:linear_energy1} for the linear problems \eqref{lin-ww-1} and \eqref{lin-ww-0-nh},
respectively, the idea of which will be used repeatedly throughout this work.

We write the second-order equation \eqref{lin-ww-1} as the following first-order system 
\begin{equation}\label{E:lin_uv}
\begin{cases}
\partial_tu=-V(t,\alpha)\partial_\alpha u+v, \\
\partial_tv=-V(t,\alpha)\partial_\alpha v+H\partial_\alpha^3 u
 +V_t(t,\alpha) \partial_\alpha u +V(t,\alpha)V_\alpha(t,\alpha) \partial_\alpha u.
\end{cases}
\end{equation}
In other words, $v=\partial_t u+V(t,\alpha)\partial_\alpha u$ is the directional derivative. 

Let us define the $k$-th energy associated to the above system by
\begin{equation}\label{E:lin_energy_s}
\mathcal{E}^k(t)=\frac12 \int_{-\infty}^\infty
\left((\partial_\alpha^{k+1}u)H\partial_\alpha (\partial_\alpha^{k+1}u) 
+(\partial_\alpha^k v)^2\right) d\alpha
\end{equation}
and define the energy function of order $s$ by
\begin{equation}\label{E:lin_energy}
\mathfrak{E}^s(t)=\|u\|_{L^2}^2(t)+\|v\|_{L^2}^2(t)+\sum_{k=1}^s \mathcal{E}^k(t).
\end{equation}
Note that $H\partial_\alpha$ is a positive operator with the symbol
of its Fourier transform $|\xi|$ and that
\[
\|f\|^2_{H^{1/2}} =\int^\infty_{-\infty} (f^2+fH\partial_\alpha f) \ d\alpha.
\]
It is immediate that $\mathfrak{E}^k(t)$ is equivalent to 
$\|u(t)\|^2_{H^{s+3/2}}+\|\partial_t u(t)\|^2_{H^s}$
provided that $\|V\|_{L^\infty([0,T])W^{s,\infty}_\alpha(\reals)}$ is bounded. 

We begin by investigating the time derivative of $\mathcal{E}^r$, by calculating
\begin{equation}\label{E:E_j'}
\begin{split}
\frac{d}{dt}\mathcal{E}^k(t)=&\int_{-\infty}^\infty \left((\partial_\alpha^{k+1} \partial_t u)
H\partial_\alpha (\partial_\alpha^{k+1} u) +
(\partial_\alpha^k \partial_t v) (\partial_\alpha^k v\right)) \ d \alpha \\ 
:=&\mathcal{E}^k_1+\mathcal{E}^k_2.
\end{split}
\end{equation}
The first equality uses that $H\partial_\alpha$ is self-adjoint. 

Let us first compute $\mathcal{E}^k_1$. By using the first equation in \eqref{E:lin_uv}
and the integration by parts we may write 
\begin{align*}
\mathcal{E}^k_1=& \int_{-\infty}^\infty \partial_\alpha^{k+1}(-V(t,\alpha)\partial_\alpha u +v)
H\partial_\alpha (\partial_\alpha^{k+1}u) \ d\alpha \\
=&- \int_{-\infty}^\infty V(t,\alpha)
(\partial_\alpha^{k+2}u)H\partial_\alpha^{k+2}u\ d\alpha \\
&+\int_{-\infty}^\infty (\partial_\alpha^{k+1}v) H\partial_\alpha \partial_\alpha^{k+1}u \ d\alpha 
+ \text{(lower order terms)}\\
=&\int_{-\infty}^\infty (\partial_\alpha^{k+1}v) H\partial_\alpha \partial_\alpha^{k+1}u \ d\alpha 
+ \text{(lower order terms)},
\end{align*}
where (lower order terms) is a collection of terms which can be bounded 
in terms of energy in a routine way.
The third equality uses \eqref{E:energy_comm}.

Similarly, we compute 
\begin{align*}
\mathcal{E}^k_2=& \int _{-\infty}^\infty \partial_\alpha^k
(-V(t,\alpha)\partial_\alpha v+H\partial_\alpha^3 u  \\ 
& \qquad \qquad \qquad -V_t(t,\alpha) \partial_\alpha u
-V(t,\alpha)V_\alpha(t,\alpha) \partial_\alpha u)(\partial_\alpha^k v)\ d\alpha \\
=& - \int_{-\infty}^\infty  (\partial_\alpha^k(V(t,\alpha)\partial_\alpha v)
- V(t,\alpha)\partial_\alpha^{k+1}v)(\partial_\alpha^kv)\ d\alpha \\
&-\int_{-\infty}^\infty (\partial_\alpha^{k+1}v)H\partial_\alpha^{k+2} u\ d\alpha
+ \text{(lower order terms)}.
\end{align*}
Again, (lower order terms) is made up of terms which can be bounded 
in terms of the energy in a routine way.
The first term on the right side is bounded by $\|V\|_{H^k}\|\partial_\alpha v\|_{H^{k-1}}$. 

The first term on the right side of $\mathcal{E}^k_1$ 
and the second term on the right side of $\mathcal{E}^k_2$
cancel when added together. Other terms are bounded in terms of the energy.
Therefore we have proved 
\[
\frac{d\mathfrak{E}^s}{dt} \leq C\mathfrak{E}^s,
\]
for some positive constant $C$, provided that $\|V\|_{L^\infty([0,T])W^{s,\infty}(\reals)}$ is bounded.
The energy estimates \eqref{Es:linear_energy0} with $s \geq 0$ then follows once Gronwall's inequality applies.

For the inhomogeneous problem \eqref{lin-ww-0-nh} with $s \geq 0$,
we proceed similarly to obtain 
\[
\frac{d\mathfrak{E}^s}{dt} \leq C\mathfrak{E}^k + \|R(t)\|_{H^{s}(\reals)}^2,
\]
from which \eqref{Es:linear_energy1} follows.

We next show that this can be extended to negative indices $s$.  Let
$0 \leq s' \leq 3/2$, and observe that if $u$ solves
\eqref{lin-ww-1} for $t \in [0,T_1]$, then $U  = \lll D_\alpha \rrr^{-s'}u$ satisfies 
\[
\begin{cases} 
P U = [P, \lll D_\alpha \rrr^{-s'}] u \qquad \text{for}\quad t \in [0, T_1], \\
U(0,\alpha) = \lll D_\alpha \rrr^{-s'} u_0, \qquad \partial_t
U(0,\alpha)=\lll D_\alpha \rrr^{-s'} u_1.
\end{cases}
\]
Thus $U$ satisfies the energy estimates \eqref{Es:linear_energy0}-\eqref{Es:linear_energy1} for
$s \geq 0$:
\begin{align}
\label{E:energy-Z-1}
 \| \partial_t U& \|_{L^\infty([0,T_1])H^{s}_\alpha} + 
\|U\|_{L^\infty([0,T_1])H^{s+3/2}_\alpha} \\
& \leq C_{T_1}
 ( \|[P,\lll D_\alpha \rrr^{-s'}]
u\|_{L^2([0,T_1])H^s_\alpha} + \| u_0 \|_{H^{s -s' + 3/2}_\alpha} + \| u_1
\|_{H^{s -s'}_\alpha}). \notag
\end{align}
The commutator term is bounded by
\begin{align*}
\|[P,\lll D_\alpha \rrr^{-s'}]
u\|_{L^2([0,T_1])H^s_\alpha} & \leq C ( \| \partial_t u \|_{L^2([0,T_1])H^{s -
    s'}_\alpha} + \| u \|_{L^2([0,T_1])H^{s -s' + 1}_\alpha} ) \\
& \leq C(T_1 )^{1/2} ( \| \partial_t u \|_{L^\infty([0,T_1])H^{s -
    s'}_\alpha} + \| u \|_{L^\infty([0,T_1])H^{s -s' + 1}_\alpha} ),
\end{align*}
but for $T_1>0$ sufficiently small and fixed, this can be absorbed
into the left hand side of \eqref{E:energy-Z-1}  to get
\begin{equation}
\label{E:energy-Z-2}
\| \partial_t u \|_{L^\infty([0,T_1])H^{s - s'}_\alpha} +
\|u\|_{L^\infty([0,T_1])H^{s-s'+3/2}_\alpha} \leq C (\| u_0 \|_{H^{s -s' + 3/2}_\alpha} + \| u_1
\|_{H^{s -s'}_\alpha})
 .
\end{equation} 
Applying this argument over a finite number of time steps of size
$T_1>0$ yields the estimate for any finite $T>0$, with constants
dependent on $T$.  A similar argument applies in the case of
\eqref{lin-ww-0-nh} to yield the estimate \eqref{Es:linear_energy1}
for $-3/2 \leq s \leq 0$.

\end{appendix}

\bigskip

\noindent{\bf Acknowledgement. }
We would like to thank N. Burq for pointing out several mistakes
in earlier versions of this work and for many helpful suggestions.  We would also like to thank T. Alazard, N. Burq, and
C. Zuily for suggesting the method used
in Appendix \ref{S:local_smoothing}.

HC was supported by an NSF Postdoctoral Fellowship while in residence
at the Mathematical Sciences Research Institution (MSRI). 
The work of VMH was supported partly by the NSF grant DMS-0707647.
The work of GS was supported partly by the NSF grant DMS-0602678.

\bigskip

\bibliographystyle{plain}
\bibliography{WWbib}

\end{document}